\documentclass[aos, preprint]{imsart}

\RequirePackage{amsthm,amsmath,amsfonts,amssymb}
\RequirePackage[numbers]{natbib}
\RequirePackage[colorlinks,citecolor=blue,urlcolor=blue]{hyperref}%% uncomment this for coloring bibliography citations and linked URLs
\RequirePackage{graphicx}%% uncomment this for including figures

\usepackage{bm}
\usepackage[dvipsnames]{xcolor}
\usepackage{subfig}
\usepackage[normalem]{ulem}
\usepackage{enumitem}

\startlocaldefs
\theoremstyle{plain}
\newtheorem{theorem}{Theorem}[section]
\newtheorem{proposition}[theorem]{Proposition}
\newtheorem{lemma}[theorem]{Lemma}
\newtheorem{corollary}[theorem]{Corollary}

\theoremstyle{remark}
\newtheorem{definition}[theorem]{Definition}
\newtheorem{assumption}[theorem]{Assumption}
\newtheorem{remark}[theorem]{Remark}

\newcommand{\myappendix}{\hyperref[appendix]{Appendix}} %\myappendix{} prints "Appendix"  hyperlinked to Appendix.

\endlocaldefs

\begin{document}

\begin{frontmatter}
%%%%%%%%%%%%%%%%%%%%%%%%%%%%%%%%%%%%%%%%%%%%%%
%%                                          %%
%% Enter the title of your article here     %%
%%                                          %%
%%%%%%%%%%%%%%%%%%%%%%%%%%%%%%%%%%%%%%%%%%%%%%
\title{
Estimating the hyperuniformity exponent of point processes}
%\title{A sample article title with some additional note\thanksref{T1}}
\runtitle{Estimating the hyperuniformity exponent of point processes}
%\thankstext{T1}{A sample of additional note to the title.}

\begin{aug}
%%%%%%%%%%%%%%%%%%%%%%%%%%%%%%%%%%%%%%%%%%%%%%%
%% Only one address is permitted per author. %%
%% Only division, organization and e-mail is %%
%% included in the address.                  %%
%% Additional information can be included in %%
%% the Acknowledgments section if necessary. %%
%% ORCID can be inserted by command:         %%
%% \orcid{0000-0000-0000-0000}               %%
%%%%%%%%%%%%%%%%%%%%%%%%%%%%%%%%%%%%%%%%%%%%%%%
\author[A,B]{\fnms{Gabriel}~\snm{Mastrilli}\ead[label=e1]{gabriel.mastrilli@ensai.fr}},
\author[B]{\fnms{Bart{\l}omiej}~\snm{B{\l}aszczyszyn}\ead[label=e2]{Bartek.Blaszczyszyn@ens.fr}}
\and
\author[A]{\fnms{Fr\'ed\'eric}~\snm{Lavancier}\ead[label=e3]{frederic.lavancier@ensai.fr}}
%%%%%%%%%%%%%%%%%%%%%%%%%%%%%%%%%%%%%%%%%%%%%%
%% Addresses                                %%
%%%%%%%%%%%%%%%%%%%%%%%%%%%%%%%%%%%%%%%%%%%%%%

\address[A]{Univ Rennes, Ensai, CNRS, CREST\printead[presep={,\ }]{e1,e3}}
\address[B]{Inria, ENS/PSL Univ, Paris, France \printead[presep={,\ }]{e2}}

\end{aug}

\begin{abstract}
We address the challenge of estimating the hyperuniformity exponent $\alpha$ of a  spatial point process, given only one realization of it. Assuming that the structure factor $S$ of the point process  follows a vanishing power law at the origin (the typical case of a hyperuniform point process), this exponent is defined as the slope near the origin of $\log S$. Our estimator is built upon the (expanding window) asymptotic variance of some wavelet transforms of the point process. By combining several scales and several wavelets, we develop  a multi-scale, multi-taper estimator $\widehat{\alpha}$. We analyze its asymptotic behavior, proving its consistency under various settings, and enabling the construction of  asymptotic confidence intervals for $\alpha$ when $\alpha < d$ and under Brillinger mixing. This construction is derived from a multivariate central limit theorem where the normalisations are non-standard and vary among the components.  We also present a non-asymptotic deviation inequality providing insights into the influence of tapers on the bias-variance trade-off of  $\widehat{\alpha}$. Finally, we investigate the performance of $\widehat{\alpha}$ through simulations,  and we apply our method to the analysis of hyperuniformity in a real dataset of marine algae.

\end{abstract}



\end{frontmatter}
%%%%%%%%%%%%%%%%%%%%%%%%%%%%%%%%%%%%%%%%%%%%%%
%% Please use \tableofcontents for articles %%
%% with 50 pages and more                   %%
%%%%%%%%%%%%%%%%%%%%%%%%%%%%%%%%%%%%%%%%%%%%%%
%\tableofcontents

%%%%%%%%%%%%%%%%%%%%%%%%%%%%%%%%%%%%%%%%%%%%%%
%%%% Main text entry area:
 
 \section{Introduction}

Hyperuniform point processes exhibit slower growth in the variance of the number of points at large scales compared to Poisson point processes. Formally, a stationary point process  in Euclidean space is hyperuniform if the variance of its cardinality in a ball with radius $R$ is negligible with respect to the volume of this ball, as $R$ tends to infinity. 
This property distinguishes hyperuniform processes from homogeneous Poisson point processes and popular models derived from them, delving into the rigid structures created by long-range correlations.
Hyperuniform point process models include some Gibbs models with long-range interactions such as the sine-beta process  \cite{valko2009continuum, dereudre2021dlr}, Coulomb gases~\cite{kunz1974one, leble2021two} and Riesz gas \cite{boursier2022decay}, some determinantal point processes such as the Jinc \cite{soshnikov2002gaussian} and Ginibre~\cite{ginibre1965statistical} 
 models, and some cloaked and perturbed lattices \cite{kim2018effect, klatt2020cloaking}. 
For further insights and examples in physical literature, refer to Torquato~\cite{torquato2018hyperuniform}, in mathematical literature to Ghosh and Lebowitz~\cite{ghosh2017fluctuations} and to the recent  monograph by Coste~\cite{coste2021order}.

Initially conceptualized in statistical physics by Torquato and Stillinger~\cite{torquato2003local}, hyperuniform systems have attracted significant interest due to their unique position between perfect crystals, liquids, and glasses \cite{torquato2016hyperuniformity, zachary2009hyperuniformity, ouguz2017hyperuniformity, weijs2015emergent, hexner2017noise, lin2017hyperuniformity}. These distinctive properties make them valuable for designing innovative materials, as seen in works like \cite{florescu2009designer, muller2013silicon, froufe2016role, leseur2016high, gorsky2019engineered, yu2023evolving, degl2015thz}.
Recently, hyperuniformity has gained attention in various applied contexts, offering insights into phenomena ranging from DNA and the immune system to active matter theory, urban systems, ices, rock dispersion on Mars, hydrodynamics, avian photoreceptors, and cosmology \cite{wilken2023spatial, mayer2015well, huang2021circular, dong2023hyperuniform, martelli2017large, zhu2023dispersed, lei2019hydrodynamics, jiao2014avian, philcox2023disordered}. Detecting and quantifying hyperuniformity is crucial across these diverse domains. Despite this, statistical inference for hyperuniformity has only recently gained attention by \cite{hawat2023estimating, klatt2022genuine}. Our work specifically addresses this important problem.

It is widely understood that spectral representation is a valuable tool for analyzing the variance of signals. Similarly, hyperuniformity in point processes can be redefined using Bartlett's spectral measure, whose density (if it exists) is known in physical terms as the {\em structure factor}, denoted by $S$ (see \eqref{defS}). To be more precise, under certain conditions on the process, hyperuniformity becomes equivalent to the structure factor $S(k)$ vanishing at zero frequency $k=0$. This approach  allows also for the classification of processes based on how quickly their structure factor $S$ diminishes at zero. A common assumption is that $S$ follows a power-law behavior near zero: $S(k) \sim t|k|^{\alpha}$ as $|k| \to 0$, with $t>0$  and $\alpha \geq 0$ (note that $\alpha = 0$ corresponds to the non hyperuniform case). Within this framework, estimating the parameter $\alpha$, called {\em hyperuniformity exponent,} not only helps in detecting hyperuniformity but also provides valuable insights into correlations at  large scale.

The main contribution of our paper is the introduction of a family of estimators for $\alpha$ that can be computed using only one realization of the point process. 
We  demonstrate their consistency and establish asymptotic confidence intervals, in an expanding window regime, thereby providing  the first theoretically well-grounded estimators of the hyperuniformity exponent $\alpha$. 

\smallskip

Previous works in this area, 
starting with \cite{torquato2018hyperuniform}, involve: (i) estimating the structure factor $S(k)$ for small, but not zero, frequencies $k$; and then (ii), analyzing the behavior of this estimator near zero. 
Recent advancements in estimating the structure factor, as demonstrated by \cite{hawat2023estimating, klatt2022genuine, rajala2023fourier, grainger2023spectral, yang2024fourier}, allow for the estimation of  $\alpha$ through this double-limit procedure, though the theoretical guarantees and error control remain unclear.  In contrast, our approach stands out from these prior works by directly estimating the rate of the structure factor at zero frequency, i.e., determining the value of $\alpha$ (which can be $0$ or positive, thereby allowing for both detecting hyperuniformity and quantifying it), without the need for an intermediate step of estimating the structure factor for non-zero frequencies. In what follows, we present the key observations in this regard.

Our construction of an estimator of $\alpha$ for a point process $\Phi$ on $\mathbb{R}^d$ involves considering the variance of the linear statistic $\sum_{x \in \Phi} f(x/R)$, which scales like $R^{d - \alpha}$ as $R$ goes to infinity, for any suitable smooth function $f$. This key asymptotic result motivates the definition of a simple estimator for $\alpha$, which, while not yet refined, serves as a conceptual starting point: For a smooth and rapidly decreasing function $f$ with zero integral, $\operatorname{Var}\left[\sum_{x \in \Phi} f(x/R)\right]$ can be estimated by $\left(\sum_{x \in \Phi} f(x/R)\right)^2$ and leads  to the following simple estimator of $\alpha$:
\[ d - \frac{\log\left(\sum_{x \in \Phi_R} f(x/R)\right)^2}{\log(R)},\]
where $\Phi_R:=\Phi\cap [-R,R]^d$ accounts for the fact that we observe  the process within a finite yet expanding window $[-R,R]^d$. 
The  consistency of this estimator, as $R\to\infty$, is formulated in Theorem~\ref{thm_cv_alpha}. However, due to its reliance on estimating variance from a single realization, this estimator is inefficient in practice, converging at a rate of $\log(R)$. But it can be enhanced by a more efficient exploitation of the available spectrum, and by employing several tapers. 

Indeed, firstly, we  examine the linear statistics $\sum_{x \in \Phi_R} f(x/R^j)$ for several scales $j\in(0,1)$ in a discrete set $J = \{j_1, \ldots,  j_{|J|}\}$, whose variances scale now like $R^{(d-\alpha)j}$ as $R$ approaches infinity. 
Similar to the classical estimation of the long-range memory exponent of times series~\cite{fay2009estimators}, by combining the logarithms of squares of these statistics, we define the following ``multi-scale'' estimator for $\alpha$:
\[ d-\sum_{j \in J}\frac{w_j}{\log(R)}\log\left(\sum_{x \in \Phi_R} f(x/R^j)\right)^2,\]
where $(w_j)_{j \in J}$ represents explicit weights derived from the least squares optimization; for further context, in the field of time series analysis, refer to~\cite{percival2020spectral}.
Still, the variance of this new estimator remains relatively high for practical applications due to potential strong correlations among the different scales $j$ for $R<\infty$.

To address this limitation, we further leverage the concept of multi-taper \cite{rajala2023fourier, hawat2023estimating}: We move away from relying solely on a single function $f$ by averaging statistics derived from several carefully selected smooth and rapidly decreasing, centered functions $(f_i)_{i \in I}$, called tapers, where $I$ is a finite subset of $\mathbb{N}$. 
This approach leads us to the {\em multi-scale, multi-tapered estimator} of $\alpha$, which serves as a generalization of both previous estimators and is formally defined as:
\[\widehat{\alpha} = d-\sum_{j \in J}\frac{w_j}{\log(R)}\log\left(\sum_{i \in I} \left(\sum_{x \in \Phi_R } f_i(x/R^j)\right)^2 \right).\]%\label{e.tapers-estim}
It is worth noting that this estimator can be practically computed with only one realization of the point process. Indeed, its variance is reduced by using several scales and tapers, which contribute in a decorrelated way, making this estimator self-averaging.
Furthermore, it will be proven to be consistent under the same assumptions as the non-tapered estimator. 

\smallskip

To provide a more precise result, we assume Brillinger-mixing for $\Phi$ and $\alpha < d$,  to  derive a multivariate central limit theorem (refer to Theorem \ref{thm_clt}). This theorem  asserts that a vector of individual estimators, each based on smooth and rapidly decreasing centered taper functions $(f_i)_{i \in I}$ at various scales $j \in J$, converges to a zero mean Gaussian vector $(N_{i,j})_{i \in I, j \in J}$ as $R$ tends to infinity:
\[ \left(R^{\frac{\alpha - d}{2} j}\sum_{x \in \Phi_R} f_i(x/R^j)\right)_{i \in I, j \in J} \xrightarrow[R \to \infty]{Law} \sqrt t \, ( N_{i,j})_{i \in I, j \in J}, \]
with an explicit covariance matrix
This central limit theorem is unusual  in that the normalizations vary among the components and exceed the rate of $R^{-dj/2}$ observed in the non-hyperuniform scenario. Thanks to this result, we can establish an asymptotic confidence interval for $\alpha$ (refer to Proposition \ref{prop_ICA}). 

Additionally, through a non-asymptotic deviation inequality stated in Proposition~\ref{prop_biais_var}, we examine the impact of the number of tapers $|I|$. In the non hyperuniform case of $\alpha = 0$,  and with the functions $(f_i)_{i \in I}$ being orthonormal in $L^2(\mathbb{R}^d)$, it is quite straightforward to prove that the asymptotic variance of $\log(R) (\widehat{\alpha} - \alpha)$ scales as $|I|^{-1}$ (see Proposition \ref{prop_var_lim_alpha0}). To address the case of $\alpha > 0$, Proposition \ref{prop_biais_var} considers the set of taper functions $(f_i)_{i \in I}$ defined by the Hermite wavelets, and reveals that the variance still scales as $|I|^{-1}$ in this setting. On the other hand, as it is already well known with the multi-taper technique in the context of univariate time series~\cite{riedel1995minimum, percival2020spectral}, this result  also states that not too many tapers should be used in order to control a small bias.

\smallskip

As our theoretical results are asymptotic and our estimators (calculated in a finite window) demand to set some parameters, like the scales $J$ and  the tapers $f_i$, $i\in I$, we provide some practical recommendations in Section~\ref{sec_prat_imp}. To verify the practical robustness of our approach, numerical benchmarks on simulated point processes are conducted in Section~\ref{sec_sim_sim}.  We then  apply our estimation method to address the conjecture that the hyperuniformity exponent of matched point processes is $\alpha=2$ \cite{andreas2020hyperuniform}, which our numerical study confirms.  Finally, in Section~\ref{sec_prat_imp}, we analyse a real dataset of marine algae from \cite{huang2021circular}, providing new insights on the hyperuniformity phenomenon of this system.

\smallskip

The remaining part  of the paper is organized as follows. Section~\ref{sec_prelim} introduces some basics about point processes and defines the concept of hyperuniformity, including the exponent $\alpha$. Some examples  of hyperuniform and non-hyperuniform point process models are given. Section~\ref{sec_alpha}
 presents our estimator  $\widehat{\alpha}$ and investigates its main properties: well-definiteness, consistency, limiting distribution, asymptotic confidence intervals, choice of tapers based on a bias-variance trade-off. 
In Section~\ref{sec_prat_imp}, we discuss practical recommendations for the implementation of $\widehat{\alpha}$ and we assess its performances by numerical simulations. We then apply our method to various theoretical models of point processes and to some real dataset of marine algae. Finally, Section~\ref{app_proof} gathers all technical proofs of the results presented in Section~\ref{sec_alpha}, while \myappendix{} provides a brief reminder of cumulant measures and Brillinger mixing for point processes. 

The Python code to implement our estimator and reproduce our experiments is available in our online GitHub repository at \url{https://github.com/gabrielmastrilli/Estim_Hyperuniformity}.

\section{Preliminaries}\label{sec_prelim}

In this section, after establishing our notation and presenting some basics related to simple stationary point processes, we define hyperuniformity in both Fourier and spatial domains and introduce the hyperuniformity exponent $\alpha$. We conclude by the presentation of several examples of hyperuniform point processes.

\subsection{General notations}
We consider functions and point processes in Euclidean space $\mathbb{R}^d$ of dimension  $d \geq 1$. For $a, b \in \mathbb{C}^d$, $a.b = \sum_{i = 1}^d a_i \overline{b_i}$ denotes the Hermitian scalar product  between $a$ and $b$, while the associated  norm is denoted by $|a|$. We denote $\bm{i} = \sqrt{-1} \in \mathbb{C}$. The volume of a set $A\subset \mathbb R^d$ is also denoted $|A|$, while for a finite set $I$, the notation $|I|$ stands for its cardinality. The Euclidean ball of radius $R$ is denoted by $B(0, R)$. For $x = (x_1, \dots, x_d) \in \mathbb{R}^d$, $|x|_{\infty}:= \max_{i = 1, \dots, d} |x_i|$ and $|x|_1 = \sum_{i = 1}^d |x_i|$. Finally, for $(a,b) \in \mathbb{R}$, we use the notations $a \vee b := \max(a, b)$ and $a \wedge b:= \min(a, b)$. 

For $1\leq p < \infty$, we denote by $L^{p}(\mathbb{R}^d)$ the space of measurable functions $f: \mathbb{R}^d \to \mathbb{R}$ such that $\|f\|_p^p:=\int_{\mathbb{R}^d} |f(x)|^p dx < \infty$.  
We denote by $L^{\infty}(\mathbb{R}^d)$ the space of functions $f: \mathbb{R}^d \to \mathbb{R}$ such that $\|f\|_{\infty}:= \sup_{x \in \mathbb{R}^d} |f(x)| < \infty$. For $p = 2$, the scalar product between the $L^{2}(\mathbb{R}^d)$ functions $f_1$ and $f_2$ is $\langle f_1, f_2 \rangle = \int_{\mathbb{R}^d} f_1(x) f_2(x) dx$. We adopt the following convention for the Fourier transform of a function $f \in L^{1}(\mathbb{R}^d)$: 
$$\forall k \in \mathbb{R}^d,\quad \mathcal{F}[f](k):= \frac1{(2\pi)^{d/2}} \int_{\mathbb{R}^d} f(x) e^{\bm{i} k.x} dx.$$ 
As usual, the Fourier transform is extended to $L^2(\mathbb{R}^d)$ functions thanks to the Plancherel Theorem~\cite{folland2009fourier}: $\forall f_1, f_2 \in L^2(\mathbb{R}^d)$, $\langle f_1, f_2 \rangle = \langle \mathcal{F}[f_1], \mathcal{F}[f_2] \rangle$.  We say that a function $f: \mathbb{R}^d \to \mathbb{R}$ is in the Schwartz space $\mathcal{S}(\mathbb{R}^d)$ if $f$ is infinitely differentiable and if for all multi-indexes $(\beta_1, \beta_2) \in (\mathbb{N}^d)^2$, then $\sup_{x \in \mathbb{R}^d} |x^{\beta_1} \partial_{\beta_2} f(x)| < \infty$. Finally, a function $f : \mathbb{R}^d \mapsto \mathbb{C}$ is said analytic if $$\forall x_0 \in \mathbb{R}^d, ~ \exists h \in \mathbb{R}^d\setminus\{0\},~\forall x \in ]x_0-h, x_0+h[^d,~f(x) = \sum_{k \in \mathbb{N}^d} a_k(x_0) x^k,$$
where $(a_k(x_0))_{k \in \mathbb{N}^d}$ is a sequence of complex scalars and where the convergence of the series is uniform.

\subsection{Hyperuniform point processes}\label{sec_defs} 
In this section, we briefly review necessary notions and results related to hyperuniform point processes. For a comprehensive introduction to point processes, we recommend consulting the standard two-volume textbook \cite{daley2007introduction, daley2003introduction} or the more concise recent manuscript \cite{baccelli2020random}. Additionally, for a presentation specifically related to hyperuniformity, we suggest referring to the unpublished monograph \cite{coste2021order}. 
For a foundational understanding of power spectra of point processes, consult~\cite[Chapter 5]{bremaud2014fourier}.
 
The set of points configurations in $\mathbb{R}^d$ is defined as:
$$\text{Conf}(\mathbb{R}^d):= \{\phi \subset \mathbb{R}^d|~ \text{For all } K \text{ compact of } \mathbb{R}^d, \text{ then } |\phi \cap K| < \infty\}.$$ 
This set is endowed with the $\sigma$-algebra generated by the mapping $\phi \mapsto |\phi\cap K|$ for all compact sets $K$. A point process $\Phi$ is a random element of $\text{Conf}(\mathbb{R}^d)$. A point process $\Phi$ is called simple if it contains almost surely only distinct points, and it is called stationary if for all $x \in \mathbb{R}^d$, $\Phi + x:= \{y + x|~y \in \Phi\}$ is equal in distribution to $\Phi$. As a consequence, the intensity measure $\rho^{(1)}$ of a stationary point process $\Phi$ (defined for any subset $A$ of $\mathbb{R}^d$ by $\rho^{(1)}(A) = \mathbb{E}(|\Phi \cap A|)$), is proportional to the Lebesgue measure on $\mathbb{R}^d$: $\rho^{(1)} = \lambda dx$. The scalar $\lambda \geq 0$ is called the intensity of the point process. The second  order factorial moment measure   $\rho^{(2)}$ of a simple point process $\Phi$ is a measure on $(\mathbb{R}^d)^2$  defined by
\begin{equation}\label{rho2}
\rho^{(2)}(A_1\times A_2) = \mathbb{E}\Big[\sum_{x, y \in \Phi}^{\neq} \mathbf{1}_{x \in A_1, y \in A_2}\Big], \end{equation}
for all $A_1, A_2$ subsets of $\mathbb{R}^d$. The symbol $\neq$ over the sum means that we consider only distinct points. 
\begin{assumption}\label{ass_rho_leb}
Throughout the paper, we tacitly assume that the point process $\Phi$ is simple, stationary, with an intensity $\lambda > 0$, and that its second-order intensity measure $\rho^{(2)}$ is absolutely continuous with respect to the Lebesgue measure on $\mathbb{R}^d\times \mathbb{R}^d$. This allows us to represent it as follows:
\begin{align}\label{eq_rho_2}
	\rho^{(2)}(dx,dy) = \lambda^{2} g(x - y) dx dy,
	\end{align}
where $g(x)$ is a function on $\mathbb{R}^d$. Additionally, we assume that $g - 1 \in L^{1}(\mathbb{R}^d)$.
\end{assumption}

The function $g : \mathbb{R}^d \to \mathbb{R}$ is known as the pair-correlation function of $\Phi$. With this established, we can proceed to define the {\em structure factor}  of $\Phi$.
This is a function  $S$  defined for any $k\in\mathbb R^d$ as
\begin{equation}\label{defS}
S(k) := 1 + \lambda \int_{\mathbb{R}^d}(g(x) - 1) e^{- \bm{i}k.x} dx.
\end{equation}

This function represents the density of the Bartlett spectral measure of $\Phi$, cf~\cite[Section~8.2]{daley2003introduction}. 
Under Assumption~\ref{ass_rho_leb}, $S$ is a non-negative, bounded and continuous function. Furthermore, by application of the Campbell formula \cite{baccelli2020random} and the Plancherel Theorem \cite{folland2009fourier}, we deduce the following useful property: For all $f_1, f_2 \in L^1(\mathbb{R}^d)\cap L^2(\mathbb{R}^d)$,
\begin{equation}\label{e.prop_campbell}
\operatorname{Cov}\Big[\sum_{x \in \Phi} f_1(x), \sum_{x \in \Phi} f_2(x)\Big] = \lambda \int_{\mathbb{R}^d} \mathcal{F}[f_1](k) \overline{\mathcal{F}[f_2]}(k) S(k) dk.
\end{equation}

We are now in position to define hyperuniformity in the Fourier domain.
\begin{definition}\label{def_HU}
		Under the conditions specified in  Assumption~\ref{ass_rho_leb},  $\Phi$ is said to be hyperuniform in the Fourier domain if $S(0) = 0$.
\end{definition} 
As per Definition \ref{def_HU}, hyperuniformity is connected to the behavior of the structure factor at low frequencies. Thanks to formula~\eqref{e.prop_campbell}, this definition aligns with the conventional understanding of hyperuniformity in the spatial domain, which focuses on the number variance's behavior at large scales.
\begin{definition}\label{def_HU_space}
	$\Phi$ is said to be hyperuniform in the spatial domain if for any compact convex set $W$ of $\mathbb{R}^d$, then $$\operatorname{\operatorname{Var}}[|\Phi \cap r W|] \underset{r \to \infty}{=} o(|r W|).$$
\end{definition} 
The equivalence between these two notions is discussed in \cite{coste2021order}. The ``degree'' of hyperuniformity is often quantified in the Fourier domain, based on the following assumption.

\begin{assumption}\label{ass_S_0}
 Under the conditions outlined in Assumption~\ref{ass_rho_leb}, we additionally assume that the structure factor $S$ scales near the origin as 
\begin{equation}\label{eq_S_0}
    S(k) \underset{|k| \to 0}{\sim} t |k|^{\alpha},
\end{equation} 
where $t > 0$ and $\alpha \geq 0$.
\end{assumption}

The parameter $\alpha$ in  Assumption~\ref{ass_S_0} is the hyperuniformity exponent, which we aim to estimate. Clearly, the point process is classified as hyperuniform (in the Fourier domain) if and only if  $\alpha > 0$.

\begin{remark}\label{rmq_class_HU}
	When $\alpha > 0$, Assumption \ref{ass_S_0} leads to three classes of hyperuniformity, whether $0<\alpha < 1$,  $\alpha = 1$ or $\alpha > 1$, see \cite{torquato2018hyperuniform} and \cite{coste2021order}. These classes are distinguished based on the behavior of the number variance $\operatorname{Var}[|\Phi \cap B(0, R)|]$ as $R \to \infty$. Specifically:
	\begin{enumerate}
		\item[(I)] For $0 < \alpha < 1$, we have $\operatorname{Var}[|\Phi \cap B(0, R)|] = O(R^{d - \alpha})$.
		\item[(II)] For $\alpha = 1$, the variance scales as $\operatorname{Var}[|\Phi \cap B(0, R)|] = O(R^{d - 1}\log(R))$.
		\item[(III)] For $\alpha > 1$, the variance behaves as $\operatorname{Var}[|\Phi \cap B(0, R)|] = O(R^{d - 1})$.
	\end{enumerate}
\end{remark}

To wrap up this brief introduction to hyperuniformity in point processes, it's important to note that throughout the paper 
we assume that our point process $\Phi$ satisfies Assumptions~\ref{ass_rho_leb} and~\ref{ass_S_0}. To simplify our analysis without compromising generality  we also normalise the intensity to assume  $\lambda=1$. We detail in Section~\ref{sec_bench} how to rescale in practice the observed point patterns to match this theoretical normalisation.

\subsection{Examples of hyperuniform and non-hyperuniform point processes}\label{sec_example}

We review below standard point process models, highlighting their hyperuniform or non-hyperuniform property.
We also show how we can construct hyperuniform processes with a prescribed exponent $\alpha$.
 
As already pointed out, the homogeneous Poisson point process on $\mathbb{R}^d$ is clearly not hyperuniform. Indeed, for a stationary Poisson point process $\Phi$ with intensity $\lambda > 0$, $\operatorname{Var}[|\Phi\cap rW|] = \lambda |rW|$ for any compact convex set $W$ of $\mathbb{R}^d$, so that the property in Definition~\ref{def_HU_space} is not satisfied. 
More generally, point process models exhibiting weak dependencies are not hyperuniform. 
For instance, Gibbs point processes with short-range interactions have been proven to be non hyperuniform in \cite{dereudre2023hyperuniformity}. Most Cox processes \cite{chiu2013stochastic, moller2003statistical} 
are not hyperuniform either, because for these models we typically have $\int_{\mathbb{R}^d}(g(x) -1) dx >0$, like for the subclass of Newman-Scott processes, ruling out the hyperuniform property $S(0)=0$, see~\eqref{defS}. On the other hand, consider a stationary determinantal point process \cite{macchi1975coincidence, soshnikov2000determinantal}  with correlation kernel $K: \mathbb{R}^d\times \mathbb{R}^d \to \mathbb{C}$ and denote $K_0(x -y):=|K(x, y)|$. Then it is hyperuniform if and only if $\|K_0\|_2^2 = K_0(0)$. This implies that 
Gaussian, Whittle-Mat\'ern, and Cauchy determinantal point process models \cite{lavancier2015determinantal} are not hyperuniform.  

Determinantal point processes provide nonetheless a first simple class of hyperuniform point processes, as long as their kernel satisfies $\|K_0\|_2^2 = K_0(0)$. This corresponds to the family of ``most repulsive'' determinantal point processes, as studied in \cite{Biscio_2016} and \cite{moller2021couplings}.  In particular, Jinc  \cite{soshnikov2002gaussian} and Ginibre \cite{ginibre1965statistical}  point processes are hyperuniform, with exponent $\alpha = 1$ and $\alpha = 2$, respectively.

In the class of  Gibbs point processes, important examples of long-range interactions models exhibit hyperuniformity. This is notably the case of the sine-beta process \cite{valko2009continuum, dereudre2021dlr}, one-dimensional and two-dimensional Coulomb gases \cite{kunz1974one, leble2021two}, and the one-dimensional Riesz gas \cite{boursier2022decay}.

A stationary lattice process  is, in a way, an extreme example of a hyperuniform process in the sense of Definition~\ref{def_HU_space} \cite{ torquato2018hyperuniform, ghosh2017fluctuations, coste2021order}. It corresponds to the process $\mathbb{Z}^d+ U$, where $U$ follows a uniform distribution on $[-1/2, 1/2]^d$. 
This process does not satisfy Assumption~\ref{ass_rho_leb} but it is possible to build upon it to provide a wide class of more regular hyperuniform models. This idea is exploited in \cite{andreas2020hyperuniform}, where hyperuniform processes are obtained by the thinning of a homogeneous Poisson point process so that the retained points match the stationary lattice. Their hyperuniformity exponent, however, is unknown, even if it is thought to be $\alpha=2$. In our simulation study in Section~\ref{sec_sim_match}, we question this conjecture thanks to our estimator. 

Another construction based on the stationary lattice leads to the model of cloaked and perturbed lattices~\cite{gabrielli2004point, kim2018effect, klatt2020cloaking}, that allows to construct hyperuniform processes with a prescribed exponent $\alpha$. This point process is $\{x + U + U_x + \xi_x|~x \in~\mathbb{Z}^d\}$, where $U, (U_x)_{x \in \mathbb{Z}^d}, (\xi_x)_{x \in \mathbb{Z}^d}$  are independent random variables, called perturbations. The random variables $U$ and $(U_x)_{x \in \mathbb{Z}^2}$, that are i.i.d. and uniform on $[-1/2, 1/2]^d$, ensure that the process satisfies Assumption~\ref{ass_rho_leb}  \cite{klatt2020cloaking}. The random variables $(\xi_x)_{x \in \mathbb{Z}^d}$ are  in turn i.i.d. with characteristic function $\varphi$ satisfying  $1 - |\varphi(k)|^2 \sim~c |k|^{\alpha}$ as $|k| \to 0$, where $c > 0$ and $\alpha > 0$. They ensure that the process is hyperuniform with exponent $\alpha$~\cite{kim2018effect, gabrielli2004point}. These processes serve as a benchmark in our simulation study of  Section~\ref{sec_sim_pert_latt}.

\section{Estimating the hyperuniformity exponent~$\alpha$} \label{sec_alpha}
	We first introduce in Section~\ref{wavlets}  truncated wavelet transforms of point processes, that constitute the basic statistic whose variance scales as a function of $\alpha$. Based upon these transforms, we then present in Section~\ref{sec_construc} our multi-scale, multi-taper estimator of $\alpha$, with some guarantees on its well-defined nature. Section~\ref{sec_asymp_prop}
investigates its asymptotic properties, including its limiting distribution. This allows us to derive in Section~\ref{sec_conf_int} asymptotic confidence intervals for $\alpha$. We finish in Section~\ref{sec_multi_taper} by inspecting the bias and variance trade-off in the choice of tapers  involved in our estimator.

\subsection{Truncated wavelet transforms of point processes}\label{wavlets}
An effective estimator for the exponent $\alpha$ in the structure factor  $S(k)$ given by~\eqref{eq_S_0} should consider the relationship between the frequency regime $|k| \to 0$ and the size of the observation window $[-R,R]^d$ as $R$ tends to infinity, within which these frequencies can be observed.  
Our primary tools for addressing this are the following general linear statistics:
\begin{definition}\label{def_T_j}
	Let $f \in \mathcal{S}(\mathbb{R}^d)$, $R > 1$ and $j > 0$. The $R$-truncated wavelet transform of the point process $\Phi$ at scale $j$ associated to $f$ is defined by
	\begin{equation}\label{eq_wav}
		T_j(f, R)=T_j(f, R;\Phi):= \sum_{x \in \Phi} \mathbf{1}_{[-R, R]^d}(x) f\left(\frac{x}{R^j}\right)= \sum_{x \in \Phi_R}  f\left(\frac{x}{R^j}\right),
	\end{equation}
	denoting  $\Phi_R:=\Phi\cap [-R,R]^d$. 
\end{definition}
The scaling factor $R^j$ in \eqref{eq_wav} corresponds to the observation scale of the point process $\Phi$ in the spatial domain, while $R^{-j}$ represents the observation frequency in the Fourier domain. Given Definitions \ref{def_HU} and \ref{def_HU_space}, it is logical to assess the hyperuniformity of $\Phi$  focusing on both large scales and low frequencies by studying the behavior of the truncated wavelet transforms $T_j(f, R)$ as $R$ approaches infinity.

Our first key result, stated in the following proposition, demonstrates this  relevance by showing that  the scale of variance of these transforms is directly linked to the value of the exponent $\alpha$, particularly for all values $\alpha\ge0$. The proof is a direct corollary of Lemma \ref{lemma_cov}, which is postponed to Section~\ref{sec_clt}

\begin{proposition}\label{prop_variance}  
Let $\Phi$ satisfy Assumptions~\ref{ass_rho_leb} and~\ref{ass_S_0}, with intensity $\lambda=1$.
 Let $f \in \mathcal{S}(\mathbb{R}^d)$, $R > 1$ and $0 < j < 1$. Then:
	$$\operatorname{Var}\Big[T_j(f, R)\Big] \underset{R \to \infty}{\sim}  R^{(d - \alpha)j} \int_{\mathbb{R}^d} |\mathcal{F}[f](k)|^2 t |k|^{\alpha} dk.$$
\end{proposition}
The relation formulated above has already been observed and exploited in the literature for some specific models of hyperuniform point processes; see, for example, \cite[Theorem~3]{soshnikov2002gaussian} for hyperuniform determinantal point processes in dimension~1 and 
\cite[Lemma~5.4]{nazarov2011fluctuations} for the point process of
zeros of {\em Gaussian Analytic Functions} (GAF's) for \(d=2\) and
\(\alpha = 4\). It has also been confirmed by numerical simulations for other models in the thesis \cite{brochard2022wavelet}. Our result in Proposition~\ref{prop_variance} provides an explicit asymptotic equivalent of the variance of \(T_j(f, R)\) in the general setting of an isotropic power
law scaling of the structure factor near the origin.

Before transforming the above result into an estimator of $\alpha$,  which we will present in the next section, we would  like to make three remarks: the significance of restricting scales $j$,  the connections to previous works on estimating the structure factor $S$, and finally the use of wavelet terminology.

\begin{remark}\label{rmk_finite_size_effect}
The restriction of scales $j$ to $0 < j < 1$ in Proposition~\ref{prop_variance} holds significance, for the purpose of estimating $\alpha$. Beyond $j > 1$, border effects start to emerge. Specifically, when $j > 1$, the variance of $T_j(f, R)$ converges asymptotically to the number variance $\operatorname{Var}\left[\sum_{x \in \Phi} \mathbf{1}_{[-R, R]^d}(x) f(0)\right]$ as $R$ tends to infinity, which does not necessarily depend on $\alpha$, see Remark~\ref{rmq_class_HU}. In Section~\ref{app_finite_size_eff}, we offer a brief proof of this phenomenon, replacing the indicator function $\mathbf{1}_{[-R, R]^d}$ with a smooth, compactly supported function.
\end{remark}

\begin{remark}\label{rem_classical_scatering-method}
Several estimators for the structure function $S$ have been developed, as documented in \cite{hawat2023estimating, klatt2022genuine, rajala2023fourier, grainger2023spectral, yang2024fourier}. One of the most well-known is the ``scattering intensity''~\cite{hawat2023estimating, klatt2022genuine, torquato2018hyperuniform}, which is defined for all frequencies $k \in \mathbb{R}^d\setminus\{0\}$ as
\begin{equation}\label{eq_scat_intens}
	\widehat{S}_{SI}^R(k) := \frac1{|[-R, R]^d|^2}\left|\sum_{x \in \Phi_R} e^{-\bm{i} k.x}\right|^2.
\end{equation} 
The estimation of $\alpha$ in previous studies typically involves a two-step asymptotic approach: first, taking the limit as $R \to \infty$ of $\widehat{S}_{SI}^R(k)$, and then as $|k| \to 0$. 
One way to unify and rigorously address this double asymptotic approach is to consider the limit, as $R\to\infty$, of $\widehat{S}_{SI}^R(k_0/R^j)$, where $k_0 \in \mathbb{R}^d$ is a fixed direction and $j > 0$. The latter is nothing else that $|T_j(f,R)|^2$ in Definition~\ref{def_T_j} associated with the (non-Schwartz) function $f(x) = e^{-\bm{i} k_0.x}$, up  to the normalization by $|[-R, R]^d|^2$. Proposition~\ref{prop_variance} 
supports this idea with the additional advantage of substituting $x \mapsto e^{-\bm{i} k_0.x}$ with a smooth and well-localized function $f$ in \eqref{def_T_j},  both in the Fourier and spatial domains. The concept of using such functions $f$, as we will discuss in the upcoming section, draws inspiration from the principles of tapers in spectral analysis, as highlighted in works such as \cite{percival2020spectral,rajala2023fourier, hawat2023estimating}.
\end{remark}

\begin{remark}
	The term ``wavelet transform'' for $T_j(f, R)$ originates from the fact that we can interpret $\sum_{x \in \Phi_R} f(x/R^j) = \int \mathbf{1}_{[-R, R]^d}(x)f(R^{-j}(x - 0)) \Phi(dx)$ as an $f$-wavelet coefficient (with shift zero and scale $R^j$) of the point process  $\Phi$; cf.~\cite{brochard2022wavelet}. It can be seen as an  unbiased estimator of the wavelet coefficient $d_{j, 0}(f)$ of the function (signal) $x\mapsto \mathbf{1}_{[-R, R]^d}(x)$ in $L^2(\mathbb{R}^d)$, where $ d_{j, l}:= \int_{\mathbb{R}^d} f(R^{-j}(x+l)) \mathbf{1}_{[-R, R]^d}(x) dx$, for all $j \in \mathbb{R}$ and all $l \in \mathbb{R}^d$; cf.~\cite{mallat1999wavelet,taleb2019wavelet, brochard2022wavelet}.  Indeed, by stationarity of $\Phi$, $\mathbb{E}[T_j(f, R)] = d_{j, 0}(f)$ for all $j \in \mathbb{R}$. Additionally, well-localized behavior of $f$ in both spatial and Fourier domains are typical characteristics of mother wavelets. For a more comprehensive understanding of wavelet analysis in point processes, we recommend consulting seminal papers such as \cite{brillinger1997some} as well as recent works like~\cite{taleb2019wavelet, wav_sp_cohen_ta,brochard2022wavelet}.
\end{remark}

\subsection{Multi-scale, multi-taper estimator}\label{sec_construc}
The concept behind constructing an estimator for $\alpha$, as outlined in our Introduction, involves a linear regression of the logarithm of the square of the wavelet transforms \eqref{eq_wav}, supported by the result in Proposition \ref{prop_variance}. This method has been previously utilized in univariate time series to estimate the long-memory exponent, which in our context is akin to the exponent $\alpha$ (refer to a survey on wavelet and spectral methods for estimating the long-memory exponent in time series \cite{fay2009estimators}). Additionally, see~\cite{percival2020spectral} for a general overview of spectral density estimation in time series and \cite{percival2000wavelet} for wavelet methods. The specifics are elaborated below.

According to Proposition~\ref{prop_variance}, for a Schwartz function $f$ with zero integral, we heuristically have $T_j(f, R)^2 \sim R^{(d - \alpha)j}$ as $R \to \infty$. Therefore, we expect that for all $j \in (0, 1)$, the expression $\log(T_j(f, R)^2)/\log(R)$ behaves as $(d - \alpha)j$.
Considering the multi-taper approach~\cite{percival2020spectral, rajala2023fourier}, we can extend our analysis to multiple wavelet transforms. Specifically, for a finite discrete subset $I$ and a finite family of Schwartz functions $(f_i)_{i \in I}$ with zero integral, we anticipate the following linear scaling law:
\begin{equation}\label{eq_heur_scaling}
	\forall j \in (0, 1),~\frac{\log\left(\frac{1}{|I|}\sum_{i \in I}T_j(f_i, R)^2\right)}{\log(R)} \sim  (d - \alpha)j.
\end{equation}
Subsequently, given a finite discrete subset $J \subset (0, 1)$ of scales, we define the estimator $\widehat{\alpha}$ of $\alpha$ as the slope obtained from the least squares procedure:
$$(\widehat{\alpha}, \widehat{b}) = \underset{a, b \in \mathbb{R}^d}{\operatorname{\arg\min}} \sum_{j \in J} \left(\frac{\log\left(\frac{1}{|I|}\sum_{i \in I}T_j(f_i, R)^2\right)}{\log(R)} - (d -a) j - b\right)^2.$$
This leads to the explicit solution:
\begin{equation}\label{eq_LQ_al}
	\widehat{\alpha} = d - \displaystyle\sum_{j \in J} \frac{\hat w_j}{\log(R)}\log\left(\frac{1}{|I|}\sum_{i \in I}T_j(f_i, R)^2\right),
\end{equation}
with the weights:
\begin{equation}\label{eq_weights}
	  \forall j \in J,~\hat w_j = \frac{|J|j - \sum_{j' \in J}j'}{|J| \left(\sum_{j' \in J} {j'}^2\right) - \left(\sum_{j' \in J} j'\right)^2}.
\end{equation}
These  weights define a natural estimator suitable for our simulations. However, a more general definition is presented below, with weights $(w_j)_{j \in J}$ satisfying two minimal conditions, as explained further. The first one is $\sum_{j \in J}  w_j = 0$, as verified by \eqref{eq_weights}, and results in a slight simplification of the expression in \eqref{eq_LQ_al} with the elimination of the normalization by $|I|$. The second one, also satisfied by \eqref{eq_weights}, is $\sum_{j \in J} j w_j = 1$. It plays a fundamental role in ensuring the consistency of the estimator $\widehat{\alpha}$, as detailed at the beginning of Section~\ref{sec_asymp_prop}.

\begin{definition}\label{def_mult_scale_est}
	Let $R >0$, $J\subset(0,1)$ be a finite subset (of scales),  $(f_i)_{i \in I}$ be a finite family of Schwartz functions (tapers). We assume that at least one function of the family of wavelets $(f_i)_{i \in I}$ is analytic and takes at least one non zero value. We consider real scalar weights $(w_j)_{j \in J}$ satisfying conditions:
	\begin{align}
	\sum_{j \in J} w_j &= 0 \label{eq_sum_wj_j_0},\\
	\sum_{j \in J} j w_j &= 1 \label{eq_sum_wj_j_1}.
	\end{align}
We define the $(I,J)$-estimator of $\alpha$ 
 (which tacitly  depends also on the weights $(w_j)_{j \in J}$)  by: 
 $$\widehat{\alpha}(I, J, R) = \widehat{\alpha}(I, J, R;\Phi):=\left(d-\sum_{j \in J}\frac{w_j}{\log(R)}\displaystyle\log\left(\sum_{i \in I} T_j(f_i, R)^2\right)\right) \mathbf{1}_{|\Phi_R| \geq 1}.$$
\end{definition}
The indicator function $\mathbf{1}_{|\Phi_R| \geq 1}$ and the assumption of analyticity of at least one function $f_i$ guarantee that $\widehat{\alpha}(I, J, R)$ is well defined. Specifically, under this setting, $T_j(f_i, R;\Phi)$ is non-zero almost surely, for all $j\in J$, so that the logarithmic term is well defined. This is because, on the one hand, the analytic function $f_i$ only has Lebesgue-null sets of zeros (as long as it is not identically zero), and on the other hand,  by stationarity, the point process~$\Phi$, almost surely,  has no points on any fixed Lebesgue-null set. This forms the core of the argument supporting the following statement, the proof of which is deferred to Section~\ref{app_prop_no_atom_0}.

\begin{proposition}\label{prop_no_atom_0}
Let $\Phi$ satisfy Assumptions~\ref{ass_rho_leb} and~\ref{ass_S_0}.  Let  $f \in \mathcal{S}(\mathbb{R}^d)$ be an analytic function taking at least one non-zero value. Then, $$\mathbb{P}(T_j(f, R;\Phi) = 0, |\Phi_R| \geq 1) = 0.$$
\end{proposition}
Certainly, requiring analyticity of $f$ is more than enough to ensure the above conclusion.  However, finding a less restrictive assumption that would rule out cases where $T_j(f, R;\Phi)$ becomes zero for certain configurations of $\Phi$ with non-zero probability is not straightforward.

\subsection{Asymptotic properties}\label{sec_asymp_prop}

In this section, we provide assumptions under which the estimator $\widehat{\alpha}(I, J, R)$ given in 
Definition~\ref{def_mult_scale_est} converges in probability to $\alpha$ as $R \to \infty$. An easy yet key observation, relying on the normalization \eqref{eq_sum_wj_j_1} of the weights $w_j$, is that the estimator can be decomposed as follows:
\begin{equation}\label{eq_hat_a-a}\widehat{\alpha}(I, J, R) =\alpha + \epsilon(I, J, R),
\end{equation}
where the remainder term $\epsilon(I, J, R)$ is given by
\begin{align}\label{eq_decom_al}
\epsilon(I, J, R) = \sum_{j \in J} \frac{w_j}{\log(R)} \log\left(\sum_{i \in I} R^{(\alpha - d)j} T_j(f_i, R)^2\right)\mathbf{1}_{|\Phi_R| \geq 1} - \alpha \mathbf{1}_{|\Phi_R| = 0}.
\end{align}
To ensure consistency, it is now sufficient to assume that with probability 1:
$$\log\left(\sum_{i\in I} R^{(\alpha - d)j} T_j(f_i, R)^2\right) = o(\log(R)),$$
which, notably, would result from the tightness of the random variables on the left-hand side above, uniformly in $R$. The asymptotic variance of $T_j(f_i, R)$, as given in Proposition~\ref{prop_variance}, reduces this issue to ensuring that the random variables $R^{(\alpha - d)j} T_j(f_i, R)^2$ do not concentrate around 0 as $R\to\infty$.
In this context, we propose two assumptions: (i) a less restrictive one, albeit not too explicit,  involves assuming that these variables converge in distribution to {\em some} distribution without an atom at 0, as formulated in Theorem~\ref{thm_cv_alpha}, or (ii) a mixing assumption for $\Phi$, specifically Brillinger mixing, that implies a joint central limit theorem for these random variables, as formulated in Theorem~\ref{thm_clt}, ensuring in particular the previous setting.
The former approach (i) has the advantage of being applicable in a larger setting when $\alpha \ge d$ and the variance does not increase to infinity. As discussed further in Remark~\ref{rem_compar_cond}, this happens for some specific examples excluded from the second approach (ii). 
On the other hand, the multivariate central limit theorem in approach (ii) explicitly provides the asymptotic error law of $\log(R)(\widehat{\alpha}(I, J, R)-\alpha)$, see Corollary \ref{cor_ic_mult_scala_est_taper}, and allows for the construction of confidence intervals; see Section~\ref{sec_conf_int}.

Our minimal assumptions for the consistency of $\widehat{\alpha}(I, J, R)$, akin to the first framework described above, are formulated in the following result, whose proof is postponed to Section~\ref{app_cv}.
\begin{theorem}\label{thm_cv_alpha}
Let $\Phi$ satisfy Assumptions~\ref{ass_rho_leb} and~\ref{ass_S_0}, with intensity $\lambda=1$.  Suppose that $\Phi(\mathbb{R}^d) = \infty$ with probability~1. 
Let $J\subset(0,1)$ be a finite subset and $(f_i)_{i \in I}$ be a finite family of Schwartz functions, i.e., $f_i \in \mathcal{S}(\mathbb{R}^d)$, such that $\int_{\mathbb{R}^d} f_i(x)dx = 0$. We assume that for each $j \in J$, there exists $i_j \in I$ such that $R^{\frac{\alpha-d}2 j}T_j(f_{i_j}, R)$ converges in distribution to a random variable $X_j$ with no atom at $0$, i.e., $\mathbb{P}[X_j = 0] = 0$. Then $\widehat{\alpha}(I, J, R)$ converges in probability to $\alpha$ as $R$ tends to infinity; in symbols:
$\widehat{\alpha}(I, J, R) \xrightarrow[R \to \infty]{\mathbb{P}} \alpha$.
\end{theorem}

We now turn to the second framework and present the statement of the multivariate central limit theorem based on the Brillinger mixing condition and assuming $\alpha < d$. 
The Brillinger assumption, recalled in \myappendix{}, has been utilized in spatial statistics \cite{heinrich1985normal, heinrich2011central, heinrich2012asymptotic} and proved to be satisfied for models such as $\alpha$-determinantal point processes with $L^1(\mathbb{R}^d)$ kernel~\cite{heinrich2016strong}, determinantal point processes with $L^2(\mathbb{R}^d)$ kernel \cite{biscio2016brillinger}, Thomas Cluster, and Mat\'ern hard-core point processes~\cite{heinrich2013absolute}. The proof of the following theorem is provided in Section~\ref{sec_clt}. 
 
\begin{theorem}\label{thm_clt}
Let $\Phi$ satisfy Assumptions~\ref{ass_rho_leb} and~\ref{ass_S_0}, with intensity $\lambda=1$. Suppose that $\alpha < d$ in Assumption~\ref{ass_S_0} and assume that $\Phi$ is Brillinger mixing. Let $J\subset(0,1)$ be a finite subset and $(f_i)_{i \in I}$ be a finite family of Schwartz functions, i.e., $f_i \in \mathcal{S}(\mathbb{R}^d)$, such that $\int_{\mathbb{R}^d} f_i(x)dx = 0$. 
  Then,   
	$$(R^{\frac{\alpha - d}2 j} T_j(f_i, R))_{i \in I, j \in J} \xrightarrow[R \to \infty]{Law} (\sqrt{t} N(i, j, \alpha))_{i \in I, j \in J},$$
	where $(N(i, j, \alpha))_{i \in I, j \in J}$ is a Gaussian vector with zero mean and covariance matrix:
	\begin{equation}\label{eq_cov_mat}
		\Sigma := \left(\mathbf{1}_{j_1  = j_2} \int_{\mathbb{R}^d} \mathcal{F}[f_{i_1}](k) \overline{\mathcal{F}[f_{i_2}]}(k)|k|^{\alpha} dk \right)_{(j_1, j_2) \in J^2, (i_1, i_2) \in I^2}.
	\end{equation}
\end{theorem}

Note that when $\alpha > 0$, the strong correlation among points results in a slower rate of convergence compared to the typical Poisson-like rate of $R^{-\frac{d}{2}j}$. Furthermore, the wavelet transforms associated with different scales $j_1 \neq j_2$ are asymptotically uncorrelated, similar to what is observed in spectral analysis of univariate time series \cite{percival2020spectral}. However, in practice, we will not solely rely on this theoretical asymptotic independence. Instead, we will provide a correlated Gaussian representation, corresponding to the pre-limit value of $\widehat{\alpha}(I, J, R)$,  as presented in Section~\ref{sec_conf_int}.

As a corollary of Theorem \ref{thm_clt}, we obtain the asymptotic error law of $\widehat{\alpha}(I, J, R)$. The proof is postponed to  Section~\ref{app_cv}. 
\begin{corollary}\label{cor_ic_mult_scala_est_taper}
	Under the same setting as in Theorem \ref{thm_clt}, the estimator $\widehat{\alpha}(I, J, R)$ given in 
Definition~\ref{def_mult_scale_est} satisfies the following convergence in distribution:   
	\begin{equation}\mathbf{1}_{|\Phi_R| \geq 1}\log(R)\left(\widehat{\alpha}(I, J, R) - \alpha\right) \xrightarrow[R \to \infty]{Law} \sum_{j \in J} w_j \log\left(\sum_{i \in I} N(i, j, \alpha)^2\right),\label{eq_ic_mult_scala_est_taper}
 \end{equation}
	where $(N(i, j, \alpha))_{i \in I, j \in J}$ is Gaussian vector with zero mean and covariance matrix $\Sigma$, defined in~\eqref{eq_cov_mat}. Consequently, $\widehat{\alpha}(I, J, R) \xrightarrow[R \to \infty]{\mathbb{P}} \alpha$.
\end{corollary}  

Before delving into the construction of confidence intervals for $\alpha$, we would like to make three remarks.

\begin{remark}
As seen in~\eqref{eq_ic_mult_scala_est_taper}, the convergence rate of $\widehat{\alpha}(I, J, R)$ to $\alpha$ is logarithmic in terms of $R$. While convergence is achieved even when $|I| = 1$, meaning that $\widehat{\alpha}(I, J, R)$ is defined with the wavelet transforms $(T_j(f, R))_{j \in J}$ based on  just one function $f$, this convergence is notably slow. This is why we employ multi-taper techniques, averaging over multiple wavelet transforms associated with various functions $(f_i)_{i \in I}$ where $|I| \geq 2$, to reduce the variance of $\widehat{\alpha}(I, J, R)$. In Section \ref{sec_multi_taper}, we delve into the impact of $|I|$ on both the bias and variance of $\widehat{\alpha}(I, J, R)$.
\end{remark}

\begin{remark}\label{rem_compar_cond}
Corollary~\ref{cor_ic_mult_scala_est_taper} yields the consistency of the estimator $\widehat{\alpha}(I, J, R)$ in the setting of Theorem~\ref{thm_clt}, which is proved with $0\le \alpha < d$ under the Brillinger mixing for point processes. This should be compared to the consistency stated in Theorem \ref{thm_cv_alpha}, that may apply  even when $\alpha \geq d$. As a matter of fact, the central limit theorem of  $T_j(f, R)$ may still hold for certain point processes with $\alpha \ge d$, despite  the variance of this statistic, of order $R^{(d - \alpha)j}$  by Proposition~\ref{prop_variance}, does not increase. In such cases, the approach to prove  the central limit theorem is specific to each situation and cannot follow a general scheme as in the proof of Theorem~\ref{thm_clt}. Notably, it has been proven for  the Ginibre point process on the plane  $\mathbb{R}^2$ ($\alpha=2$) \cite{rider2006complex}, the sinc determinantal point process in dimension 1 ($\alpha=1$)~\cite{soshnikov2002gaussian, gaultier2016fluctuations}, planar determinantal point processes with reproducing kernels ($\alpha = d=  2$) \cite{haimi2024normality}, the zero set of the GAF function ($d = 2$ and $\alpha = 4$) \cite{sodin2004random, nazarov2011fluctuations} and for some dependently perturbed lattices \cite{sodin2004random} ($d = 2$ and $\alpha \in \{2, 4\}$). 
 For these processes, Theorem~\ref{thm_cv_alpha} therefore guarantees the consistency of our estimator of $\alpha$. 
\end{remark}

\begin{remark}
Though being the exception and not the rule, there exist point process models that meet neither the conditions of Theorem~\ref{thm_cv_alpha}, nor those of Theorem~\ref{thm_clt}. 
For instance, it may be proven that for a one-dimensional lattice perturbed by stable distributions with parameter $\alpha \in (1, 2)$, then for all $j \in J$ and all $f \in \mathcal{S}(\mathbb{R})$, $f\neq 0$, with zero integral, $R^{\frac{\alpha-1}2 j}T_j(f, R)$ converges to $0$ in distribution, in contradiction with the assumptions of Theorem~\ref{thm_cv_alpha} and with the non-degenerated limit of Theorem~\ref{thm_clt}. 
In fact, for this peculiar example,  $R^{\frac{\alpha-1}{\alpha} j}T_j(f, R)$ converges toward a $\alpha$-stable distribution, so that our estimator $\widehat{\alpha}(I, J, R)$ converges in probability not to $\alpha$ but to $(3\alpha -2)/\alpha \in (\alpha, \alpha + 0.172)$. 
\end{remark}

\subsection{Confidence intervals}\label{sec_conf_int}
In this section, we construct asymptotic confidence intervals for our estimator $\widehat{\alpha}(I, J, R)$. Instead of relying on the asymptotic limit presented in Corollary~\ref{cor_ic_mult_scala_est_taper}, which is based on the Gaussian representation with the covariance matrix $\Sigma$ given in~\eqref{eq_cov_mat}, we opt for another Gaussian representation with covariance $\Sigma_R$, which is transient in the sense that it depends on $R$. This approach is motivated by the observation that in the asymptotic regime $R=\infty$ the wavelet transforms associated with different scales $j_1 \neq j_2$ are asymptotically uncorrelated. However, in a pre-limit regime, where $R<\infty$, the random vectors $(T_j(f_i, R))_{i \in I}$ across $j \in J$ are not yet independent, and their actual covariance might deviate significantly from its limit $\Sigma$. This will be reflected by the transient covariance matrix $\Sigma_R$. 
Our numerical simulations in Section~\ref{sec_bench} confirm that this modification leads to a satisfactory coverage of the confidence intervals, see Table~\ref{table_coverage_R}, whereas confidence intervals based on $\Sigma$ (not reported) barely achieve a coverage of $50\%$ for the same simulations, instead of the nominal rate of $95\%$. 
The proof of the following proposition is postponed to  Section~\ref{app_ICA}, where the keypoint is to theoretically justify plugging in $\hat\alpha$ instead of $\alpha$ in the quantile function. 

\begin{proposition}\label{prop_ICA}
	Let the assumptions of Theorem \ref{thm_clt} be satisfied. Let $a \in (0, 1/2)$. For all $\beta\ge 0$,  denote by  $(N(i, j, R, \beta))_{i \in I, j \in J}$ the Gaussian vector with zero mean and covariance matrix:
	\begin{equation}\label{eq_cov_mat_R}
		\Sigma_R:= \Bigg(R^{\frac{\beta +d}2 (j_1 +j_2)} \int_{\mathbb{R}^d} \mathcal{F}[f_{i_1}](R^{j_1} k) \overline{\mathcal{F}[f_{i_2}]}(R^{j_2}k) |k|^{\beta} dk \Bigg)_{(j_1, j_2) \in J^2, (i_1, i_2) \in I^2}.
	\end{equation}
	Moreover, for $q \in (0, 1)$, we denote by $F_{R}^{-1}(q; \beta)$ the quantiles of order $q$ of $Z_R(\beta)$ where:
	$$Z_R(\beta) := \sum_{j \in J} w_j \log\left(\sum_{i \in I} N(i, j, R, \beta)^2\right).$$ 
	Then the probability that the hyperuniformity exponent  $\alpha$  is covered by the interval
	 \begin{equation}\label{eq_conf_int}
	 \left[\widehat{\alpha} - \frac{F_{R}^{-1}(1-a/2; \widehat{\alpha}\vee 0)}{\log(R)}\mathbf{1}_{|\Phi_R| \geq 1},~ \widehat{\alpha}- \frac{F_{R}^{-1}(a/2; \widehat{\alpha}\vee 0)}{\log(R)}\mathbf{1}_{|\Phi _R| \geq 1}\right],
 \end{equation}
 where $ \widehat{\alpha}:=\widehat{\alpha}(I, J, R)$, converges to $1 - a$ when $R \to \infty$.
\end{proposition}

\begin{remark}\label{rmk_cov_Hermites}
	In the case of $d = 2$ and for the Hermite wavelet functions $(f_i)_{i \in I}$ (introduced in the next section), the covariance matrix $\Sigma_R$ can be computed without numerical integration. Detailed explanations are provided in Section~\ref{app_cov_mat}.
\end{remark}

\subsection{Bias and variance trade-off in multi-tapering}\label{sec_multi_taper}
In this section, our primary focus is on studying the impact of tapers $f_i$ with indexes $i$ in set $I$ on the  estimation error $\widehat{\alpha}(I, J, R) - \alpha$. Additionally, we discuss the significance of the scales $j$ in  set $J$.
We operate within the asymptotic regime established by Corollary~\ref{cor_ic_mult_scala_est_taper} and begin with the non-hyperuniform scenario, where $\alpha=0$, serving as a benchmark for comparison.  In  Proposition~\ref{prop_var_lim_alpha0}, we establish a bound of order $|I|^{-1}$ for the variance of the limiting distribution when $\alpha=0$,  provided the taper functions are orthogonal, that translates to the order $|I|^{-1}\log^{-2}(R)$ on the variance of the estimation error. 
Moving on to the hyperuniform case $\alpha>0$, a more complex covariance structure arises due to the asymptotic dependence of wavelets $(T_j(f_i, R))_{j\in J}$ even when the tapers $f_i$ are orthogonal. This dependence potentially amplifies the pre-limit bias of our estimator as $|I|$ increases. To address this issue, we consider specific tapers from the family of Hermite wavelets, and in Proposition~\ref{prop_biais_var} show a balance to be found  in terms of the number $|I|$ of tapers while pursuing conflicting objectives of minimizing the variance and bias in our estimator.
This issue is well-known in signal processing studies, as discussed in \cite{percival2020spectral, riedel1995minimum}, and we will also illustrate it in Figure~\ref{fig_matched}.

Recall from Corollary~\ref{cor_ic_mult_scala_est_taper} that the estimation error $\widehat{\alpha}(I, J, R) - \alpha$ decreases as  $1/\log(R)$ to a limiting  random variable given in the right-hande side of~\eqref{eq_ic_mult_scala_est_taper},
where $(N(i, j, \alpha))_{i \in I, j \in J}$ is a Gaussian vector with covariance matrix $\Sigma$, as given in~\eqref{eq_cov_mat}. 
Assuming $\alpha=0$, we can establish a bound on the variance of this limiting random variable explicitly depending on the number of scales and tapers, when these tapers are assumed to be orthonormal.
\begin{proposition}\label{prop_var_lim_alpha0}
	We assume that $\alpha = 0$ and that $(f_i)_{i \in I}$ is an orthonormal family of $L^{2}(\mathbb{R}^d)$: $\forall (i_1, i_2) \in I^2$, $\langle f_{i_1}, f_{i_2} \rangle = \mathbf{1}_{i_1 = i_2}$. Let $(N(i, j, \alpha))_{i \in I, j \in J}$ be a centered Gaussian random vector with covariance matrix $\Sigma$ defined in \eqref{eq_cov_mat}. Then, 
 \begin{equation} \label{eq_var_lim_alpha0}
      \operatorname{Var}\left[\sum_{j \in J} w_j \log\big(\sum_{i \in I} N(i, j, \alpha)^2\big)\right] \leq \left(\sum_{j \in J} w_j^2\right) \frac{|I| + 1}{|I|^2}.
 \end{equation}
\end{proposition}
\begin{proof}
	When $\alpha = 0$ and $\langle f_{i_1}, f_{i_2} \rangle = \mathbf{1}_{i_1 = i_2}$, the covariance matrix $\Sigma$ becomes the identity matrix. 
Therefore
$$\sum_{j \in J} w_j \log\big(\sum_{i \in I} N(i, j, \alpha)^2\big) = \sum_{j \in J} w_j \log\big(U_j\big),$$ 
where $(U_j)_{j \in J}$ are independent $\chi^2(|I|)$ random variables with $|I|$ degrees of freedoms.
Hence,  $\operatorname{Var}\left[\log\big(U_j\big)\right]= \psi^1(|I|)$ where $\psi^1$ is the trigamma function, i.e. the second derivative of the logarithm of the gamma function, see \cite{bartlett1946statistical}. We deduce 
that
	\begin{align*}
		\operatorname{Var}\left[\sum_{j \in J} w_j \log\big(U_j\big)\right] &= \sum_{j \in J} w_j^2 \operatorname{Var}\left[\log\big(U_j\big)\right]  = \sum_{j \in J} w_j^2  \psi^1(|I|),
	\end{align*}
and 	\eqref{eq_var_lim_alpha0} follows from the standard upper bound \cite{abramowitz1968handbook}, $\forall x > 0, \psi^1(x) \leq (x+1)/x^2$.
\end{proof}

While Proposition \ref{prop_var_lim_alpha0} sheds light on the benefits of employing multiple tapers to reduce the  asymptotic variance of $\widehat{\alpha}$, it does not explain how the tapers $(f_i)_{i \in I}$ influence the bias of  $\widehat{\alpha}(I,J, R)$ for a finite size $R$ of the observation window.
Furthermore, the result in Proposition \ref{prop_var_lim_alpha0} does not cover the hyperuniform case $\alpha > 0$. To address both  these questions theoretically, we focus on the family of Hermite wavelets as defined below. These functions will also be utilized in the numerical simulations outlined in Section \ref{sec_bench}.

Let $i = (i_1, \dots, i_d) \in \mathbb{N}^d$. The Hermite wavelet $\psi_i$ is defined for any $x=(x_1,\dots,x_d)\in\mathbb R^d$ by:
 \begin{equation}\label{eq_herm}
\psi_i(x) = e^{-\frac12|x|^2} \prod_{l = 1}^d H_{i_l}(x_l), 
 \end{equation}
	where  for $n \in \mathbb{N}$ and $y \in \mathbb{R}$, $H_n(y) = (-1)^n (2^{n} n! \sqrt{\pi})^{-1/2} e^{y^2} \dfrac{d^n}{dy^n} e^{-y^2}$ are the Hermite polynomials.
It is well known that these functions are orthonormal: $\langle \psi_{i_1}, \psi_{i_2} \rangle = \mathbf{1}_{i_1 = i_2}$ for all $(i_1, i_2) \in (\mathbb{N}^d)^2$;
cf e.g. \cite{abramowitz1968handbook}, and that they are eigenvectors of the Fourier transform: $\mathcal{F}[\psi_i] =~(-\bm{i})^{i_1 + \dots + i_d} \psi_i$.
Moreover, since the Hermite wavelets are products of exponential functions and polynomials, they are analytic. Therefore, using them as tapers $f_i$ in our estimator $\widehat{\alpha}(I, J, R)$ ensures that the conditions of Proposition \ref{prop_no_atom_0} are met.

In this context, our next result  focuses on  the impact of the number of these Hermite tapers in $I$ on both the bias and variance of $\widehat{\alpha}(I, J, R)$, considering a large but finite window size $R<\infty$. The proof of this result, which is deferred to Section~\ref{app_mult_taper}, relies on lower and upper bounds on the fractional moments of the Hermite wavelets and their $L^2$ tail (refer to Lemmas~\ref{lemma_frac_norm_herm}-\ref{lemma_remain_herm}). These bounds can be interpreted as a localization property of these functions in both the spatial and Fourier domains.
The next proposition also sheds  light on the possible occurence of border effects, as already noted in  Remark~\ref{rmk_finite_size_effect}. These arise when the maximal support of the scaled wavelets approaches the size of the observation window, which happens when the term $d(R)$ below tends to 0.

\begin{proposition}\label{prop_biais_var} We suppose that $\Phi$ satisfies Assumptions~\ref{ass_rho_leb} and~\ref{ass_S_0}, and is Brillinger mixing with intensity $\lambda=1$. 
We moreover strengthen Assumption~\ref{ass_S_0} by assuming that $\alpha<d$ and  that for all $k \in \mathbb{R}^d$, $|S(k) - t|k|^{\alpha}| \leq C_S |k|^{\beta}$ with $\beta > \alpha$ and $C_S> 0$. 
We consider the family of Hermite wavelets as tapers, i.e. $f_i=\psi_i$ given by \eqref{eq_herm}, for all $i\in I$,  where $I$ is of the form $I = \{i \in \mathbb{N}^d |~|i|_{\infty} < i_{\text{max}}, \text{ and at least one component of } i \text{ is odd}\}$ with $i_{\text{max}} \geq 1$. Let $J\subset(0,1)$ be a finite subset. Finally, we denote by $j_{\text{max}} := \max\{j|~j \in J\}<1$ and $j_{\text{min}} := \min\{j|~j \in J\}>0$ and assume that 
\begin{align}\label{eq_support_hermites}
	d(R) := R^{1-j_{\text{max}}} - \sqrt{2 i_{\text{max}}} > 0.
\end{align}
Then, there exists $R_0 > 0$ and $0 < C=C(\epsilon,|J|) < \infty $, such that for all $R \geq R_0$:
	\begin{align}\label{eq_biais_var}
		\mathbb{P}\big(\mathbf{1}_{|\Phi_R| \geq 1}\log(R)|\widehat{\alpha}(I, J, R) - \alpha| \geq \epsilon \big) \leq C \left(\frac1{|I|} + \left(\frac{|I|^{1/d}}{R^{2{j_{\text{min}}}}}\right)^{\beta - \alpha} + \frac{e^{-\frac{d(R)^2}{4}}}{d(R)} \right).
	\end{align}
\end{proposition}
The following remark concludes the main concern of this section, discussing the influence of tapers on the estimator's bias and variance.
\begin{remark} \label{rem:bias-variance}
The term $|I|^{-1}$ in the bound of \eqref{eq_biais_var} signifies the reduction in the asymptotic variance of our estimator when more tapers are used. This aligns with our earlier observation in Proposition \ref{prop_var_lim_alpha0} for $\alpha=0$, extending to the hyperuniform case of $\alpha>0$. On the other hand, the term $(|I|^{1/d}/R^{2j_{\text{min}}})^{\beta - \alpha}$ accounts for the increasing estimation error (bias) as the number $|I|$ of tapers grows while $R$ remains fixed at a finite value.  
This  highlights  the risk of employing an excessive number of tapers for a fixed observation window size, as the estimator becomes more sensitive  to higher frequencies in the data, that are the higher-order terms, represented by the term  $|k|^{\beta}$, with $\beta>\alpha$, in the
expansion of the structure factor $S$ near zero frequency.
Similar observations have been made in the context of univariate time series~\cite{percival2020spectral, riedel1995minimum}. Finally the term involving $d(R)$ in  \eqref{eq_biais_var} concerns border effects that also have an impact on the bias. Indeed, $R^{j_{\text{max}}}\sqrt{2 i_{\text{max}}}$ can be viewed as the largest range  of the numerical support $\prod_{p = 1}^d[-R^{j_{\max}}\sqrt{2 i_{p}}, R^{j_{\max}}\sqrt{2 i_{p}}]$ of the scaled Hermite wavelets $\psi_i(./R^{j})$ for $i = (i_1, \dots, i_d)\in I$ and $j\in J$. Border effects appear when this support approaches the size $R$ of the window, that is when $d(R)$ tends to 0. We come back to the practical choice of $j_{\text{max}}$ that  mitigate border effects in point \ref{i.upper_bound_J} of Section~\ref{sec_prat_imp}.

\end{remark}
We conclude this section by noting that further reduction in variance may be achieved by utilizing several realizations of the point process.
\begin{remark}\label{rem_several_images}
	The proof strategy of Proposition~\ref{prop_biais_var} easily extends to scenarios where multiple independent realizations of point patterns are observed. If we denote the estimators corresponding to $M$ observed realizations $\Phi_m$ of a point process as $\widehat{\alpha}(I, J, R; \Phi_m)$ for $m = 1, \dots, M$, then we can demonstrate that the average of these estimators, $M^{-1}\sum_{m = 1}^M \widehat{\alpha}(I, J, R; \Phi_m)$, still satisfies the bound established in Proposition \ref{prop_biais_var}.  In this case, the variance is  further reduced by the factor $M^{-1}$, i.e., the term $|I|^{-1}$ in~\eqref{eq_biais_var} is replaced by $|I|^{-1} M^{-1}$.
\end{remark}

\section{Numerical study}\label{sec_bench}
In this section, we explore the numerical behavior of $\widehat{\alpha}(I, J, R)$. We first explain in Section~\ref{sec_prat_imp} our practical implementation of the estimator, whether it concerns the choice of taper functions $f_i$, their cardinality, or the choice of scales $j$.  Second, in Section~\ref{sec_sim_sim}, we apply our estimation method to simulated point processes. We start by assessing the performances of our estimator  to independently perturbed 
{\em cloaked lattices}~\cite{klatt2020cloaking}, that are benchmark models covered by our theoretical framework  and with a tunable hyperuniformity exponent  $\alpha$. Next, we consider simulated {\em matched point processes}~\cite{andreas2020hyperuniform} to address  the conjecture that their hyperuniformity exponent is $\alpha=2$.  Finally, Section~\ref{sec_real_data} deals with a real data-set of algae system, to  investigate their  hyperuniformity feature.

The codes and data concerning this section are available in our online GitHub repository at \url{https://github.com/gabrielmastrilli/Estim_Hyperuniformity}.

\subsection{Practical implementation}\label{sec_prat_imp}

To illustrate how the theory developed in Section \ref{sec_alpha} can be applied in practice for estimating $\alpha$, we consider two point process models:

\begin{itemize}
    \item The random sequential adsorption model (RSA), extensively discussed in physics and chemistry; see~\cite{evans1993random}. In the stochastic geometry literature, it is known as the Mat\'ern III hard-core process; for more details, refer to \cite[Section~6.5.3]{chiu2013stochastic}.
     \item The Ginibre process, introduced as a two-dimensional Coulomb gas in \cite{ginibre1965statistical} and as a prominent example of determinantal point processes in stochastic geometry by \cite{macchi1975coincidence}.
\end{itemize}
These two models showcase border behaviors within a spectrum of models that span the hyperuniformity exponent $\alpha$ ranging from 0 to the dimension of the space. The RSA model stands as a non-hyperuniform model with $\alpha=0$, whereas the Ginibre process represents a hyperuniform model with $\alpha=d=2$. These models aid us in illustrating and discussing implementation issues regarding our estimator, particularly in association with the Poisson point process, which acts as a reference model. In the following illustration, the RSA process has been simulated with an underlying Poisson intensity of $3$ and an exclusion radius of $r = 1$ in the observation window of $[-70, 70]^2$, resulting in approximately 20 000 points. In turn the Ginibre process has been simulated in the observation window of $[-30, 30]^2$,  resulting in about 3 500 points.

\begin{figure}
	\centering 
	\begin{tabular}{cc}
	   	\includegraphics[width=0.47\linewidth]{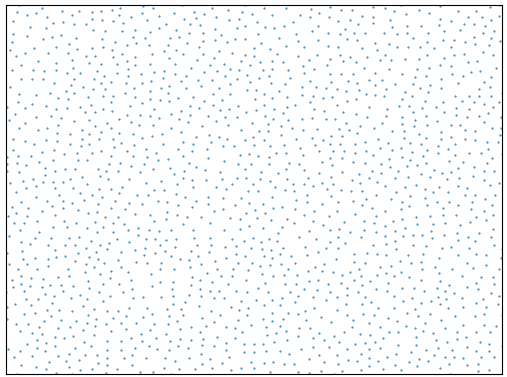}
		&
	\includegraphics[width=0.47\linewidth]{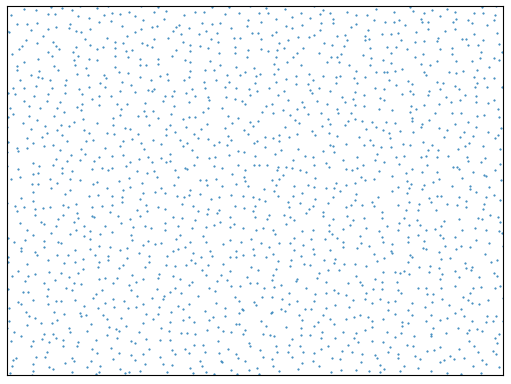}\\
	RSA & Ginibre\\
	\includegraphics[width=0.48\linewidth]{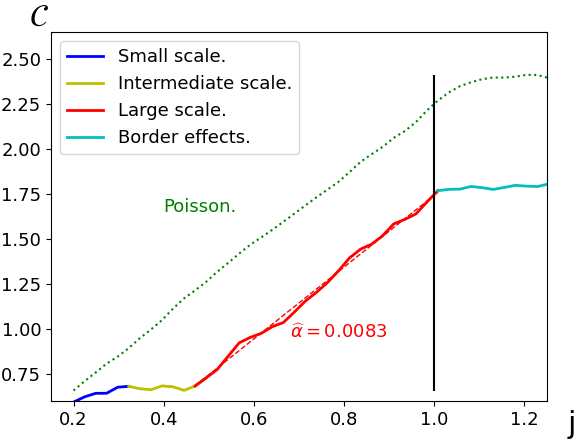} 
	& \includegraphics[width=0.48\linewidth]{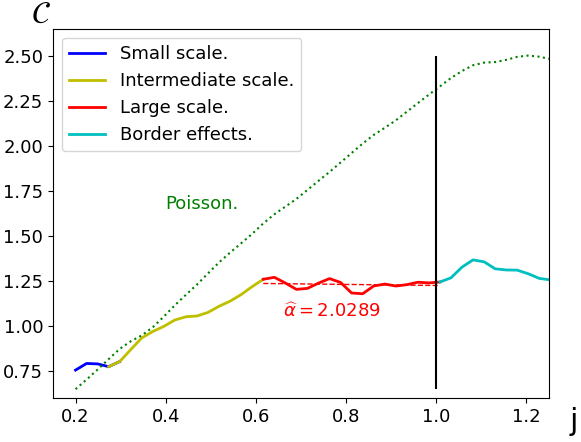}
	\end{tabular}
	\caption{Estimation of $\alpha$ using 75 Hermite tapers for one realization each of the RSA (left) and the Ginibre (right) models, having approximately 20 000 and 3 500 points, respectively (see the text for details). (Top) Example of point patterns inside $[-20, 20]^2$. (Bottom)	
	The horizontal axis represents scales $j \in (0.1, 1.3)$ and the solid lines represent the curve $\mathcal{C}$ as defined in \eqref{eq_C}. The parts in blue, yellow, red, and cyan denote different slopes of $\mathcal{C}$, depending on the scale. The estimator $\hat \alpha$ corresponds to 2 minus the estimated slope of the red part. The green dotted line represents the curve $\mathcal{C}$ for the Poisson process with similar intensity, allowing to identify the value of $j_{max}$ beyond which border effects occur, as indicated by the black vertical line. 
}
 \label{fig_expl}
\end{figure}

Figure~\ref{fig_expl} shows the two generated point patterns (zoomed in $[-20,20]^2$) and illustrates the estimation of $\alpha$  using the linear regression-based estimator $\widehat{\alpha}$ (using equation \eqref{eq_LQ_al} with weights~\eqref{eq_weights}). Specifically, the full lines in  Figure~\ref{fig_expl} represent, respectively for RSA and Ginibre model,  the function
\begin{align}\label{eq_C}
\mathcal{C}: j \in J \mapsto& \frac{1}{\log(R)} \log\left(\sum_{i \in I} T_j(\psi_i, R)^2\right) 
= \frac{1}{\log(R)} \log\left(\sum_{i \in I} \left(\sum_{x \in \Phi_R} \psi_i\left(\frac{x}{R^j}\right)\right)^2\right) 
\end{align}
 for  scales $j$ ranging from 0.1 to 1.3,  and utilizing $|I|=75$ Hermite tapers $(\psi_i)_{i \in I}$ given by~\eqref{eq_herm}.
Recall from~ \eqref{eq_LQ_al} (see also \eqref{eq_heur_scaling}) that the estimator $\widehat{\alpha}$ for $\alpha$ is defined as the dimension $d=2$ minus the slope of the function $j\mapsto\mathcal{C}(j)$. In principle all scales $j$ within the range of $(0, 1)$ can be included in the set $J$ for the estimator $\widehat{\alpha}(I,J,R)$. However, it is essential to note in Figure~\ref{fig_expl} that the slope of the function $\mathcal{C}$ varies (indicated by different colors), and obtaining a precise estimation of $\alpha$ necessitates a careful choice of the scale range $j_{min} \leq j \leq j_{max}$ within the set $J$. The pertinent ranges for this purpose  for the RSA and Ginibre models are illustrated in Figure~\ref{fig_expl} by the red segments of the curves. For $j<j_{min}$, the asymptotic regime defined in Proposition~\ref{prop_variance} is not met, and for $j>j_{max}$, border effects become apparent (refer to Remark~\ref{rmk_finite_size_effect}). 
In both illustrations of Figure~\ref{fig_expl}, $j_{max}$ happens to be close to the theoretical value of 1, where, in principle, we can capture the smallest frequencies of $S(k)$ but also observe the onset of border effects. This is also where these effects for the auxiliary estimation of $\alpha=0$ for the Poisson point process with the same intensity (shown in green dotted style) begins to manifest. 
This alignment with the auxiliary Poisson point process behavior for the choice of $j_{max}$ always occurs. 
 Further details regarding the choices of $j_{min}$ and $j_{max}$ are outlined in points \ref{i.upper_bound_J} and \ref{i.lower_bund_J} below.

Once the ranges $j_{min}$ and $j_{max}$  are chosen, the estimator $\widehat{\alpha}$ is deduced from  the slope of the least squares linear regression on a discrete set of $|J|=50$ scales selected within $[j_{min}, j_{max}]$ (further details on the cardinality of $J$ are provided in point~\ref{i.card_J}). In the examples depicted in Figure \ref{fig_expl}, this resulted in an estimation of $\widehat{\alpha}(I, J, R) = 0.0083$ for the RSA  points pattern and $\widehat{\alpha}(I, J, R) = 2.0289$ for the Ginibre points pattern.

\smallskip

In the following six points, we provide more detailed suggestions regarding the selection of parameters to implement  $\widehat{\alpha}(I, J, R)$.

\begin{enumerate}[itemsep=5pt]
	\item \label{i.scales} \textbf{Normalizing  intensity.} Remember that throughout the paper  we assumed for simplicity that the intensity of the point process $\Phi$ is equal to $\lambda = 1$. To align the scale of the wavelet transform with the observed points pattern, we estimate the intensity and choose a unit distance to obtain a realization of  points of intensity close to~1. Specifically, if  $\Phi=\{x_1, \dots, x_n\}$  is observed in the region $W_R$, which is $[-R, R]^d$ or $B(0, R)$, the numerical computations are performed with $\overline{x_1}, \dots, \overline{x_{n}}$ observed in the region $W_{\overline{R}}$, where $\overline{x_i} = \widehat{\lambda}^{1/d} x_i$, $\overline{R} = \widehat{\lambda}^{1/d} R$, and $\widehat{\lambda}=n/|W_R|$.
	
	\item \label{i.tapers} \textbf{Choice of taper functions.} In all simulations, we choose as tapers the Hermite wavelets $(\psi_i)_{i \in \mathbb{N}^d}$ as described in~\eqref{eq_herm} due to their well-localized behavior in both the spatial and Fourier domains. Additionally, they are orthogonal for the $L^2(\mathbb{R}^d)$ scalar product, making them suitable candidates for multi-tapering. Specifically, we set $f_{i} = \psi_{i}(5 \times \cdot)$. While the factor $5$ is not crucial, it allows us to observe border size effects near $j = 1$, that is to choose $j_{max}\approx 1$, as per the explanations in point~\ref{i.upper_bound_J}.  This is due to the fact that with this choice of factor 5, the maximal support of $f_i$, for $i\in I$ and $I$  chosen as in point~\ref{i.card_I} below, is $\max_{i\in I} \sigma_i\approx 1$, where $\sigma_i$ is defined in  \eqref{defsigma} below,   leading to $j_{max}\approx 1$ in view of \eqref{choice jmax}.

	\item \label{i.card_I}\textbf{Number  of tapers in $I$.} As discussed in Section \ref{sec_multi_taper}, the number of tapers plays a crucial role in balancing variance and bias. Increasing the number of tapers generally decreases the variance but can introduce a significant bias. Determining the optimal number of tapers depends on the unknown structure factor and requires practical adjustments. One approach is to progressively increase the number of tapers $|I|$. 
For small $|I|$, the curve $\mathcal{C}$ defined in~\eqref{eq_C} tends to be noisy, while for large $|I|$, the relevant scale range $[j_{min}, j_{max}]$ (depicted by the full red segments in Figure \ref{fig_expl}) becomes narrow, potentially leading to high bias. This bias can be attributed to the wavelets $\psi_i(5\times\cdot)$ corresponding to $d$-dimensional indexes $i = (i_1, \dots, i_d) \in \mathbb{N}^d$ with large components, which are less localized in both spatial and Fourier domains. Consequently, the lower bound $j_{min}$ tends to increase. More detailed information is provided in point \ref{i.lower_bund_J}.
Furthermore, as border effects are influenced by the localization of the wavelets $\psi_i(5\times\cdot)$, the upper
bound $j_{max}$ decreases as $|I|$ increases (refer to point \ref{i.upper_bound_J} below for more details). 
In our simulations, we use the set $I = \{i \in \mathbb{N}^d~|~|i|_{\infty} < i_{\text{max}}, \text{ and at least one component of } i \text{ is odd}\}$, with $i_{\text{max}} = 10$, which leads to $75$ tapers in dimension $d = 2$.

	\item \label{i.upper_bound_J} \textbf{Choice of the largest  scale $j_{max}$ in $J$.} While theoretically all scales $j$ of $(0, 1)$ can be included in the set of scales $J$, in the non-asymptotic regime, the upper bound $j_{max}$ has to be chosen to avoid border effects (see Remark \ref{rmk_finite_size_effect}). Note that with  our choice of tapers as in point~\ref{i.tapers},  there is no border effect as long as: 
$$\forall i \in I,~\sum_{x \in \Phi_R} \psi_{i}(5x/R^j) = \sum_{x \in \Phi} \psi_{i}(5x/R^j).$$
	Denoting by $\sigma_i$ the numerical support of $\psi_i(5 \times \cdot)$, defined by 
\begin{equation}\label{defsigma}
\sigma_i= \inf\{\sigma > 0~|~\forall x \in \mathbb{R}^d, ~(|x|_{\infty} \geq \sigma \implies |\psi_i(5x)|\leq~\epsilon)\}
\end{equation} 
where $\epsilon \ll 1$ is the computer's precision, the latter equality holds numerically true whenever $R^j \sigma_i \leq R$ for all $i\in I$.
The theoretical condition $j < 1$ from Proposition \ref{prop_variance} therefore translates in practice to 
\begin{equation}\label{choice jmax}
j \leq 1 - \log(R)^{-1} \max_{i \in I}\log(\sigma_i) = j_{max}.
\end{equation}
 Accordingly, as the system size $R$ increases, higher values of $j$ become available. The upper bound $j_{max}$ of $J$ solely depends on the set of tapers function and can be tabulated. Importantly, it does not depend on the underlying point process $\Phi$ but only on its intensity (remember that the system has been rescaled so that $\lambda=1$, cf point \ref{i.scales}). One practical method is to simulate several realizations of a Poisson point process with intensity~1 in the same observation window, and calculate the function $\mathcal{C}$ for a given set of tapers, in order to clearly observe the border effects to choose $j_{max}$ (this method is illustrated by the green dotted curves in Figures \ref{fig_expl} and \ref{fig_algues}). Alternatively, one can visually inspect the largest numerical support of the wavelets $\psi_i(5 \times \cdot/R^j)$ to confirm that it remains within the observation window, as shown in the left-hand plot of Figure \ref{fig_algues} for our real-data example. 

	\item \label{i.lower_bund_J} \textbf{Choice of the smaller scale $j_{min}$  in $J$.} The accuracy of the asymptotic properties described in Proposition \ref{prop_variance} relies on how large is $R^{j}$. Consequently, for smaller systems, $j_{min}$ needs to be relatively large. Unlike the upper bound $j_{max}$ for the set of scales $J$, the lower bound $j_{min}$ depends on the specific point process under consideration and should be adjusted in practice. To this regard, we visualize the curve $\mathcal{C}$ defined in~\eqref{eq_C} and select the last single slope portion of the curve before $j_{max}$, as highlighted by the full red lines of Figure \ref{fig_expl}. 

	\item \label{i.card_J} \textbf{Number of scales in  $J$.} The asymptotic independence of wavelet transforms associated with two distinct scales $j_1$ and $j_2$, as expressed in the correlation matrix $\Sigma$ in Theorem~\ref{thm_clt}, is not practically satisfied for a given $R<\infty$ if $j_1$ and $j_2$ are too close to each other. (This is best understood through the   bounds~\eqref{eq_correlations-j1-j2} for the pre-limit correlation matrix $\Sigma_R$ developed in the proof of Theorem~\ref{thm_clt}.)
  	Consequently, considering a huge number of scales does not reduce the variance. In our simulations, we considered $|J| = 50$ scales uniformly subdividing  the range $[j_{min}, j_{max}]$.
\end{enumerate}

\subsection{Simulated point processes}\label{sec_sim_sim}

In this section, we consider two families of distributions in dimension $d=2$: perturbations of cloaked lattices and matched point processes. For several realisations of them, we compute  the estimator $\widehat{\alpha}(I, J, R)$ and analyse its distribution. The first model will help us to assess the performances of our method, since in this case the hyperuniformity exponent $\alpha$ is known and can be tuned. For the second model of matched point processes, the parameter $\alpha$ is conjectured to be 2, a value that we will confirm by simulations using our estimator. 
 In a concern for consistency with applied literature, the results of this section are stated in terms of the number of observed points $n$ within the observation window $[-R, R]^d$. Consequently, to increase the number of observed points, we simulate the point processes in larger windows, and the asymptotic $R \to \infty$ discussed in Section \ref{sec_alpha} corresponds to the asymptotic $n \to \infty$ of this section. 

\subsubsection{Perturbation of cloaked lattices}\label{sec_sim_pert_latt}
 Cloaked-and-perturbed lattices,  introduced in~\cite{klatt2020cloaking}, have already been mentioned  in Section~\ref{sec_example}. Starting from a cloaked lattice  $\Phi_0 = \{q + U + U_q| q \in \mathbb{Z}^2\}$, where $U$ and $(U_q)_{q \in \mathbb{Z}^2}$ are i.i.d. and uniform on $[-1/2, 1/2]^2$, they consist on  the perturbed point process $\{x + \xi_x| x \in \Phi_0\}$, where the random variables $(\xi_x)_{x \in \Phi_0}$ are i.i.d. with a characteristic function $\varphi$ satisfying $1 - |\varphi(k)|^2 \sim c|k|^{\alpha}$ as $|k| \to 0$, where $c > 0$ and $\alpha>0$. These models achieve hyperuniformity with a targeted exponent $\alpha$ and satisfy the conditions of Theorem~\ref{thm_cv_alpha}. 
For simulation, we leveraged the representation $\xi_x \stackrel{\text{Law}}{=} \sqrt{Y(\alpha/2)} Z_2$, where $Y(\alpha/2)$ is a one-sided $\alpha/2$-stable law \cite{uchaikin2002simulation, kanter1975stable} and $Z_2$ is a standard bivariate Gaussian variable with variance $\sigma^2$.  
It is known that the choice of $\sigma$ is sensitive in these models, cf \cite{kim2018effect} Section IV.B, in the sense that a bad choice may blur the hyperuniformity feature.  
We chose $\sigma = 0.15, 0.25, 0.35$ for respectively $\alpha = 0.5, 1, 1.5$. 
Finally, for each scenario, we varied the average number of points observed in the window from 900 to 6 400.

To estimate $\alpha$ from the considered realizations, we utilize our estimator $\widehat{\alpha}(I,J,R)$ with the choice of parameters discussed in Section \ref{sec_prat_imp}.
Figure \ref{fig_lattice} displays the results, based in each case on 500 replications. For the smaller system size of $n\approx 900$, there is a  bias present, that is all the more important when $\alpha$ is large. However, as the number of points increases, the bias almost disappears. Furthermore, the empirical standard deviation shown in Figure~\ref{fig_lattice} is small enough to distinguish between class I and class III hyperuniform point processes (refer to Remark~\ref{rmq_class_HU}), even with a moderate number of observed points.

\begin{figure}
	\centering    
	\includegraphics[width=14cm]{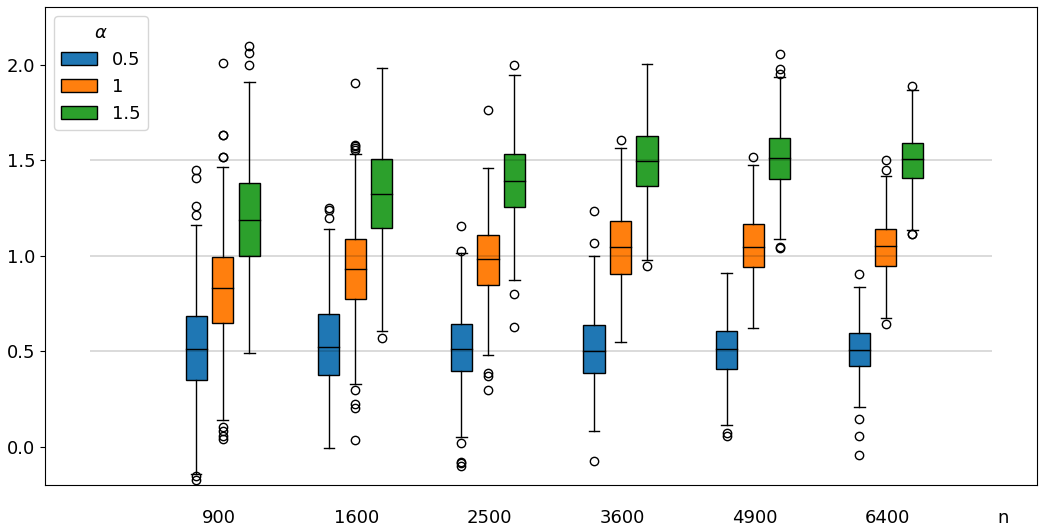}\\
	\caption{
 Empirical distribution of $\hat\alpha$ for perturbed cloaked lattices: blue, orange, and green represent the estimation for the theoretical values of $\alpha=0.5$, $\alpha=1$, and $\alpha=1.5$, respectively. Each boxplot is based on 500 estimations, each using a single realization with an average number of points of 900, 1600, 2500, 3600, 4900 and 6400, respectively (which corresponds to observation windows $[-R, R]^2$ with $R \in \{15, 20, 25, 30, 35, 40\}$).}
	\label{fig_lattice}
\end{figure} 

We next assess the quality of the asymptotic confidence intervals given by \eqref{eq_conf_int} and based on the covariance matrix $\Sigma_R$ in \eqref{eq_cov_mat_R}. Table~\ref{table_coverage_R} shows the  coverage rate of  these intervals for the same simulations as above. However, for computational ease, we used the quantiles of $Z_R(\alpha)$ instead of $Z_R(\widehat{\alpha})$ in \eqref{eq_conf_int}. Indeed the former requires the computation of $\Sigma_R$ only once for each situation, that is for a given $\alpha$ and $R$, and not 500 times as the use of $Z_R(\widehat{\alpha})$ would demand. In fact, while the computation of each entry of $\Sigma_R$ is very fast thanks to the expression of Proposition~\ref{prop_computation_sig_R} and the related remarks, the whole matrix contains $(|I|\times |J|)^2$ terms and might be  time consuming to get. %(approximately 1.5 hours on a standard computer when $|I| = 75$ and $|J| = 50$ as in our case). 
This matrix is, however, sparse,  and  some dedicated approaches could be suggested to speed up its computation, but we did not enter into these refinements. Overall, the empirical results shown in Table~\ref{table_coverage_R} are in decent agreement with the nominal rate of $95\%$, especially for large systems ($n\approx 6\,400$) and when $\alpha$ is small. These are situations where $\alpha$ is also easier to estimate; see Figure~\ref{fig_lattice}. For comparison, the same simulations (not reported)  using the asymptotic matrix $\Sigma$ in \eqref{eq_cov_mat} instead of the pre-limit matrix $\Sigma_R$  showed a coverage rate not greater than $50\%$.

\begin{table}
	\begin{tabular}{ |c|c|c|c|c|c|c|  }
		\hline
		$\alpha$  & \multicolumn{6}{|c|}{Average number of points} \\
		\hline
		& 900&1600 & 2500 & 3600 & 4900 & 6400\\
		\hline
		\centering 0.5   & 93.2\%&  93.4\% & 95.6\% & 89.9\% & 94.4\% & 95.2\% \\
		1 & 92\% & 93\% & 95.2\%  & 88.6 \% & 90.6\% & 93.8\%\\
		1.5 & 84.2\% & 88\%& 91.2\%  & 91.2\% & 95.2\% & 95.2\% \\
		\hline 
	\end{tabular}
	\caption{Coverage rate of the confidence intervals \eqref{eq_conf_int}
  for the same simulations as in Figure~\ref{fig_lattice}.}
	\label{table_coverage_R}
\end{table}

%\begin{table}[h!]
%	\begin{tabular}{ |c|c|c|c|c|c|c|  }
%		\hline
%		& \multicolumn{6}{|c|}{Average number of points.} \\
%		\hline
%		$\alpha$ & 900&1600 & 2500 & 3600 & 4900 & 6400\\
%		\hline
%		\centering 
%		0.5& 49.4\%&  40.8\% & 45.4\% & 38.4\% & 45\% & 47.4\% \\
%		1 & 42\% & 38.6\% & 42.8\%  & 36.4 \% & 40.8\% & 43\%\\
%		1.5 & 27.4\% & 36.6\%& 37.8\%  & 39.4\% & 44.2\% & 47.2\% \\
%		\hline 
%	\end{tabular}
%	\caption[]{Coverage using $\Sigma$.}
%	\label{table_coverage}
%\end{table}

\subsubsection{Matched point processes}\label{sec_sim_match}

In this section, we consider matched point processes introduced in~\cite{andreas2020hyperuniform} for a general dimension $d$, focusing here on $d=2$ for our estimation experiences. These processes are essentially  subsets of points of  a Poisson point process of intensity $\lambda_p > 1$, resulting from a sequential, mutual-nearest-neighbour matching  of Poisson points to those of a lattice with intensity~1. When $\lambda_p$ is close to 1, the resulting process is challenging to differentiate from a Poisson point process. Conversely, a large $\lambda_p$ inherits spatial regularity from the lattice, making the process more discernible. These matched point processes are proven to be hyperuniform in \cite{andreas2020hyperuniform}, with the hyperuniformity exponent $\alpha$ conjectured to be $2$. 
While matched point processes do not necessarily meet Assumptions~\ref{ass_rho_leb} and~\ref{ass_S_0}, Figure~\ref{fig_matched} demonstrates that our estimator yields values around $2$. We conducted simulations using the Python library \verb|structure-factor| developed by \cite{hawat2023estimating}, considering systems of $6\, 400$ points. 

Specifically, for $\lambda_p = 1.2$ and 75 tapers (depicted by the red box), the estimation averages below 2. This can be attributed to the challenge of observing hyperuniformity in this scenario where the resulting process closely resembles a Poisson point process; larger system sizes may be necessary to discern hyperuniformity more clearly. Besides, to highlight the impact of the number of tapers on bias and variance (see Section \ref{sec_multi_taper}), we considered the same setting of $\lambda_p = 1.2$ but with fewer taper functions. The green box represents the estimation of $\alpha$ with only 16 tapers instead of 75. With 16 tapers, the bias is reduced, but the standard deviation is higher compared to the 75-taper case.

\begin{figure}
	\includegraphics[width = 6.7cm]{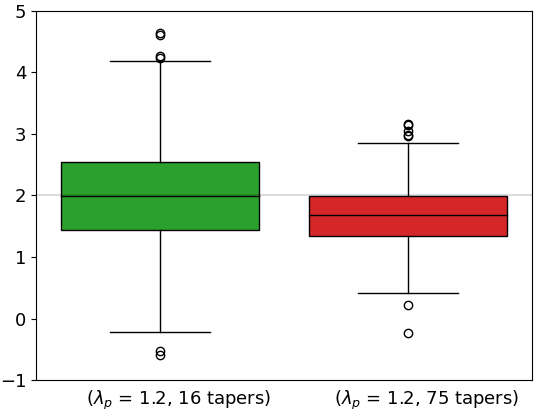}
		\includegraphics[width = 6.7cm]{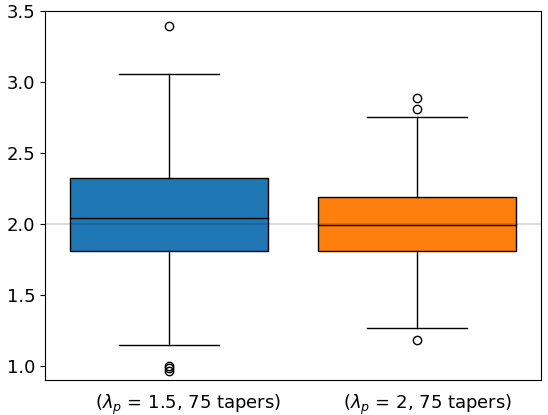}
	\caption{Empirical distribution of $\hat\alpha$ for matched point processes of 6 400 points (which corresponds to observation window $[-40, 40]^2$). Each boxplot is based on 500 estimations, each using one realization of the point process. (Left) The Poisson intensity is $\lambda_p=1.2$ for both the green and red boxes, but the green box correspond to an estimation using $|I|=16$ tapers, while the red one to $|I|=75$ tapers.  Observe the bias-variance tradeoff in the choice of the number of tapers: fewer tapers result in smaller bias, while more tapers reduce variance; see the discussion in Remark~\ref{rem:bias-variance}. (Right) The number of tapers is $|I|=75$ for both the blue and orange boxes, but $\lambda_p=1.5$ for the blue box and $\lambda_p=2$ for the orange box.}
	\label{fig_matched}
\end{figure}

\subsection{Application to a real data set}\label{sec_real_data}

In this section, we delve into real data concerning marine algae known as Effrenium Voratum, as examined in \cite{huang2021circular}. This algae system has been a subject of study in active matter theory, which deals with large numbers of interacting agents such as schools of fish or flocks of birds \cite{ramaswamy2010mechanics}. According to \cite{huang2021circular}, hyperuniformity is observed in the Effrenium Voratum system. This is attributed to the swimming behavior of the algae, which generates fluid flow and establishes long-range correlations. In their study,~\cite{huang2021circular} estimated $\alpha$ to be approximately 0.6  from a sequence of frames of the system, using a log-linear regression near zero of the scattering intensity function (refer to Remark~\ref{rem_classical_scatering-method} and~\cite{hawat2023estimating} for more details).

We apply our estimator to estimate $\alpha$ with the implementation outlined in Section \ref{sec_prat_imp}. 
Figure~\ref{fig_algues}  shows on the left-hand side the configuration of the system in one frame from the video sequence. In order to illustrate the occurence of border effects, the curve of $x_1\in\mathbb R \mapsto e^{-x_1^2/2} H_{9}(5x_1/R^j)$ has been overlayed, for different scales $j$. This function constitutes the $x$-axis part of the 2-dimensional tensorial Hermite wavelet $\psi_{9,i_2}$, $i_2\in\mathbb N$; see \eqref{eq_herm}. Note that for a given scale $j$, $\psi_{9,i_2}$ are the wavelets that exhibit the largest numerical support in the $x$-direction amongst all wavelets $\psi_i$, $i\in I$, that we use, for our choice of $I$  explained in point~\ref{i.card_I} of Section~\ref{sec_prat_imp}. Accordingly, we deduce from the left-hand side plot of Figure~\ref{fig_algues} that for all scales $j<0.95$, the support of our wavelets are within the observation window, while for $j>0.95$ border effects may occur. In the right-hand side plot of  Figure~\ref{fig_algues}, 
 each black line represents the curve $\mathcal{C}$ from Equation \eqref{eq_C}, derived for each of the 100 frames of the video in \cite{huang2021circular}. The red line depicts the mean of these black lines. Note that equations \eqref{eq_LQ_al} and~\eqref{eq_weights} imply that estimating $\alpha$ with the red line is equivalent to averaging the estimated values $(\widehat{\alpha}_j)_{j = 1, \dots, 100}$ from each black line.
The curve $\mathcal{C}$ for the Poisson point process with a similar intensity is added in a green dotted line. It confirms that for $j>j_{max}$, with $j_{max}=0.95$ shown by a vertical black line, border effects begin to appear, in agreement with the observation made on the left-hand side plot of Figure~\ref{fig_algues}.

Figure \ref{fig_algues} demonstrates for small scales ($j < 0.5$)  the value 0.6 reported in~\cite{huang2021circular}.  However, as we observe larger scales ($j > 0.7$), it becomes apparent that the system exhibits a stronger form of hyperuniformity than predicted by the classical approach via scattering intensity function. Indeed, conducting linear regression on the segment of the red curve bordered by the blue dotted lines ($0.7 < j < 0.95$) in Figure \ref{fig_algues} leads to approximately $\hat\alpha=1$.

\begin{figure}
	\centering
	\includegraphics[width=6.2cm]{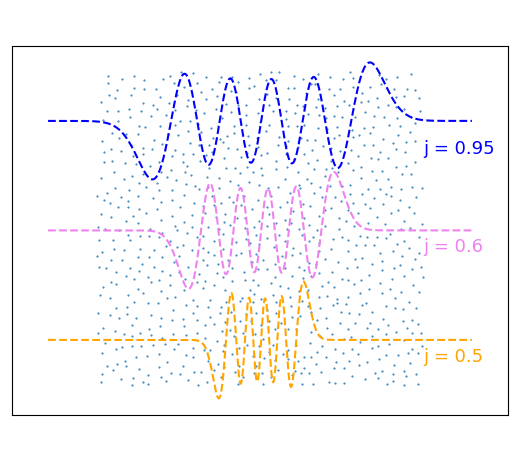}   \includegraphics[width=7cm]{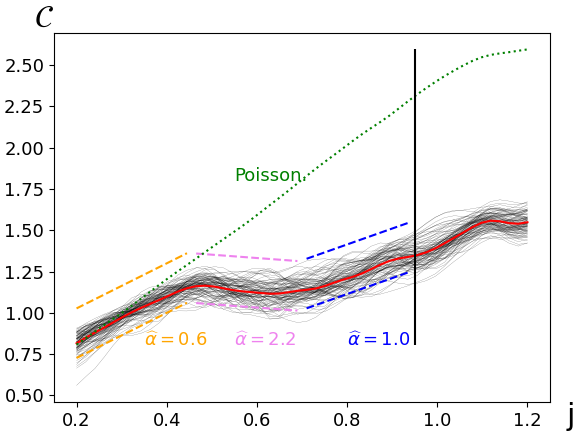}  
	\caption{Estimation of $\alpha$ for an algae system (approximately 900 points) observed in a video sequence of 100 frames. (Left) One configuration of the system, corresponding to one frame extracted from the full video, and, overlayed, support of our largest wavelet (in the $x$-direction) for different scales $j$. For $j > 0.95$, the support of the wavelet is largest than the observation window and  border effects may appear.  (Right) Each thin black line corresponds to the curve $\mathcal{C}$ (defined in \eqref{eq_C}) for each frame. The red line corresponds to the average of the narrow black lines.  The green dotted line corresponds to the curve $\mathcal{C}$ for the Poisson process with the same intensity. The full vertical  black line indicates the scale $j_{max}=0.95$ beyond which border effects occur, in agreement with the maximal wavelet support displayed on the left-hand plot. Yellow, pink and blue lines indicate local slopes depending on the scales, with what would be the estimation of $\alpha$ based on these scales. The proper estimation of the hyperuniformity exponent corresponds to the higher scale and leads to $\hat\alpha=1$.
	 }
	 \label{fig_algues}
	\end{figure} 

\section{Proofs of the results of  Section \ref{sec_alpha}}\label{app_proof}
This section compiles the  proofs of the theoretical statements introduced in Section~\ref{sec_alpha}. These proofs pertain first in Section~\ref{sec_clt} to the variance and the central limit theorem of the truncated wavelet transforms $T_j(f, R)$  of  point processes. Furthermore,  they  address the properties of the multi-scale, multi-taper estimator $\widehat{\alpha}(I, J, R)$, which is based on $T_j(f_i, R)$, with discrete scales  $j\in J$ and tapers $f_i$, $i\in I$.
Specifically, this section covers the following aspects: an explanation why the scales used in $\hat\alpha$ must satisfy $j<1$ (Section~\ref{app_finite_size_eff}), the well-defined nature of our estimator (Section~\ref{app_prop_no_atom_0}), its asymptotic properties (Section~\ref{app_cv}), confidence intervals (Section~\ref{app_ICA}), 
as well as the bias/variance trade-off related to the number $|I|$ of tapers (Section~\ref{app_mult_taper}).
To quantify this trade-off, we employ as tapers Hermite wavelets, i.e., $f_i=\psi_i$ is given by \eqref{eq_herm}. For these tapers, we also develop in the last Section~\ref{app_cov_mat} the asymptotic covariance matrix of the truncated wavelet transforms $T_j(\psi_i, R)$.

\subsection{Multivariate central limit theorem for truncated wavelet transforms}\label{sec_clt}
In this section,  we employ the method of cumulants to prove Theorem~\ref{thm_clt}.
We begin by establishing a crucial auxiliary result regarding the second-order moment, which also directly justifies the statement made in Proposition~\ref{prop_variance}.

\begin{lemma}\label{lemma_cov}
	Let $\Phi$ satisfy Assumptions~\ref{ass_rho_leb} and~\ref{ass_S_0} with intensity $\lambda=1$. Let $f_1, f_2 \in \mathcal{S}(\mathbb{R}^d)$ and $j_1, j_2 \in ]0, 1[$. Then, as $R \to \infty$:
	\begin{align*}
		\operatorname{Cov}[R^{\frac{\alpha - d}2 j_1}T_{j_1}(f_1, R)&, R^{\frac{\alpha - d}2 j_2}T_{j_2}(f_2, R)] \\& =   R^{\frac{d+\alpha}2(j_1+j_2)} \int_{\mathbb{R}^d} \mathcal{F}[f_1](R^{j_1} k) \overline{\mathcal{F}[f_2]}(R^{j_2} k) t|k|^{\alpha} dk + o(1).
	\end{align*}
\end{lemma}
\begin{proof}
	We define for $r > 0$, and $i \in \{i_1, i_2\}$, $f_i^r(x) = f_i(x) \mathbf{1}_{[-r, r]^d}(x)$. Using~\eqref{e.prop_campbell}  and the change of variables $k \leftrightarrow R^{j_1} k$:
	\begin{align*}
		\operatorname{Cov}[R^{\frac{\alpha - d}2 j_1}&T_{j_1}(f, R), R^{\frac{\alpha - d}2 j_2}T_{j_2}(f, R)] \\&= R^{\frac{\alpha - d}2 (j_1+j_2)}\int_{\mathbb{R}^d} R^{dj_1}\mathcal{F}[f_1^{R^{1 -j_1}}](R^{j_1} k) R^{dj_2} \overline{\mathcal{F}[f_2^{R^{1-j_2}}]}(R^{j_2} k) S(k) dk \\& = R^{\frac{d + \alpha}2(j_2 - j_1)} \int_{\mathbb{R}^d} \mathcal{F}[f_1^{R^{1-j_1}}](k) \overline{\mathcal{F}[f_2^{R^{1-j_2}}]}(R^{j_2 - j_1} k) R^{\alpha j_1}S\left(\frac{k}{R^{j_1}}\right ) dk.
	\end{align*}
	In this expression, we would like to replace the truncated functions $\mathcal{F}[f_i^{R^{1-j_i}}]$, $i=1,2$,  by the non-truncated ones $\mathcal{F}[f_i]$, respectively. To do so, we write
	\begin{multline}\label{cov1}
		\operatorname{Cov}[R^{\frac{\alpha - d}2 j_1}T_{j_1}(f_1, R), R^{\frac{\alpha - d}2 j_2}T_{j_2}(f_2, R)] \\ = R^{\frac{d + \alpha}2(j_2 - j_1)} \int_{\mathbb{R}^d}  \mathcal{F}[f_1](k) \overline{\mathcal{F}[f_2]}(R^{j_2 - j_1} k) R^{\alpha j_1}S\left(\frac{k}{R^{j_1}}\right) dk +\delta_R^1 - \delta_R^2,
	\end{multline}
where 
	\begin{align*}
		\delta_R^1 & = R^{\frac{d + \alpha}2(j_2 - j_1)}\int_{\mathbb{R}^d} \{\mathcal{F}[f_1^{R^{1-j_1}}](k) - \mathcal{F}[f_1](k)\} \overline{\mathcal{F}[f_2^{R^{1-j_2}}]}(R^{j_2 - j_1} k) R^{\alpha j_1} S\left(\frac{k}{R^{j_1}}\right) dk,  \\
		\delta_R^2 & =R^{\frac{d + \alpha}2(j_2 - j_1)} \int_{\mathbb{R}^d} \{\overline{\mathcal{F}[f_2^{R^{1-j_2}}]}(R^{j_2 - j_1}k) - \overline{\mathcal{F}[f_2]}(R^{j_2 - j_1}k)\} \mathcal{F}[f_1](k)R^{\alpha j_1} S\left(\frac{k}{R^{j_1}}\right)  dk.
	\end{align*}
	We bound $\delta_R^1$ using Cauchy-Schwarz inequality and Plancherel Theorem:
	\begin{align}\label{eq_delta_1_bound}
		|\delta_R^1| &\leq R^{\frac{d + \alpha}2(j_2-j_1)} \|S\|_{\infty} \Big(\|\mathcal{F}[f_1] - \mathcal{F}[f_1^{R^{1 -j_1}}]\|_2 \times R^{\alpha j_1} \|\mathcal{F}[f_2^{R^{1-j_2}}]\|_2\Big) \nonumber \\& \leq  R^{\frac{d + \alpha}2(j_1+j_2)} \|S\|_{\infty}  \Big(\|f_1 \mathbf{1}_{\mathbb{R}^d \setminus [-R^{1-j_1}, R^{1-j_1}]^d} \|_2 \|f_2\|_2\Big) \nonumber
  		\\& \leq R^{d +\alpha}\|S\|_{\infty}   \|f_1 \mathbf{1}_{\mathbb{R}^d \setminus [-R^{1-j_1}, R^{1-j_1}]^d} \|_2 \|f_2\|_2.
	\end{align}	
	This last term goes to $0$ as $R \to \infty$ because $f_1 \in \mathcal{S}(\mathbb{R}^d)$ and $j_1 < 1$, proving that  $\delta_R^1 \to 0$ as  $R \to \infty$.  For the same reasons, and since $j_2<1$, we also have $\delta_R^2 \to 0$ as  $R \to \infty$. Let us now study the main term in \eqref{cov1}. To exploit the behavior at $k \to 0$ of $S$, we decompose:
\begin{multline}\label{non truncated}
		\int_{\mathbb{R}^d}  \mathcal{F}[f_1](k) \overline{\mathcal{F}[f_2]}(R^{j_2 - j_1} k) R^{\alpha j_1}S\left(\frac{k}{R^{j_1}}\right) dk  \\= \int_{\mathbb{R}^d}  \mathcal{F}[f_1](k) \overline{\mathcal{F}[f_2]}(R^{j_2 - j_1} k) t |k|^{\alpha} dk + \delta^3_R,
	\end{multline} 
where
$$\delta^3_R=\int_{\mathbb{R}^d}  \mathcal{F}[f_1](k) \overline{\mathcal{F}[f_2]}(R^{j_2 - j_1} k) \left(R^{\alpha j_1}S\left(\frac{k}{R^{j_1}}\right) - t|k|^{\alpha}\right) dk.$$
		Let $\epsilon > 0$. Assumption~\ref{ass_S_0}  implies that there exists $\mu > 0$ such that $\forall k' \in \mathbb{R}^d$, if $|k'| \leq \mu$ then  $|S(k') - t|k'|^{\alpha}| \leq \epsilon |k'|^{\alpha}$. It also applies, together with the fact that $S$ is bounded, that there exists $C_S>0$ such that $S(k)|k|^{-\alpha}<C_S$ for all $k\in\mathbb R^d$.  Consequently,
	\begin{align*}
		|\delta^3_R|  \leq & \int_{B(0, \mu R^{j_1})}  \left|\mathcal{F}[f_1](k) \overline{\mathcal{F}[f_2]}(R^{j_2 - j_1} k)\right | R^{\alpha j_1} \left |S\left(\frac{k}{R^{j_1}}\right) - t\left |\frac{k}{R^{j_1}}\right|^{\alpha}\right | dk \\
	&+ \int_{\mathbb{R}^d \setminus B(0, \mu R^{j_1})}  \left|\mathcal{F}[f_1](k) \overline{\mathcal{F}[f_2]}(R^{j_2 - j_1} k)\right | R^{\alpha j_1} \left |S\left(\frac{k}{R^{j_1}}\right) - t\left |\frac{k}{R^{j_1}}\right|^{\alpha}\right | dk \\
	 \leq & \epsilon \int_{B(0, \mu R^{j_1})} \left |\mathcal{F}[f_1](k)\right | \frac1{(2\pi)^{d/2} }\|f_2\|_1 |k|^{\alpha} dk \\
	 &+ \int_{\mathbb{R}^d \setminus B(0, \mu R^{j_1})} \left |\mathcal{F}[f_1](k)\right | \frac1{(2\pi)^{d/2} }\|f_2\|_1 (C_S + t)|k|^{\alpha}dk.
	\end{align*} 
	Since $f_1 \in \mathcal{S}(\mathbb{R}^d)$, we deduce that: $$\forall \epsilon > 0, ~\underset{R \to \infty}{\lim\sup} |\delta^3_R| \leq \epsilon \left\|\mathcal{F}[f_1](\cdot) |\cdot|^{\alpha}\right\|_1 \|f_2\|_1 + 0.$$ Hence, $\delta^3_R \to 0$ as $R \to \infty$. 
	Combining \eqref{cov1} and \eqref{non truncated} with  the change of variable $k \leftrightarrow R^{-j_1} k$, we have proved that:
	\begin{align*}
		\operatorname{Cov}[R^{\frac{\alpha - d}2 j_1}&T_{j_1}(f_1, R), R^{\frac{\alpha - d}2 j_2}T_{j_2}(f_2, R)] \\& = R^{\frac{d+\alpha}2(j_1+j_2)} \int_{\mathbb{R}^d}  \mathcal{F}[f_1](R^{j_1}k) \overline{\mathcal{F}[f_2]}(R^{j_2} k) t |k|^{\alpha} dk +\delta_R^1 - \delta^2_R +  \delta^3_R,
	\end{align*}
	where $\delta_R^1 - \delta^2_R +  \delta^3_R = o(1)$ as $R \to \infty$, which concludes the proof. 
\end{proof}

We will now proceed with the proof of our main multivariate central limit theorem.

\begin{proof}[Proof of Theorem \ref{thm_clt}.]
	Let $(N(i, j, \alpha))_{i \in I, j \in J}$ be a Gaussian vector with  zero mean and covariance matrix $\Sigma$ defined in Equation \eqref{eq_cov_mat}. By the Cram\'er-Wold device, it suffices to prove that 
	\begin{equation}\label{eq_cv_cl}
		A_R:= \sum_{i \in I, j \in J} a_{i, j} R^{\frac{\alpha - d}2 j} T_j(f_i, R) \xrightarrow[R \to \infty]{Law}  \sum_{i \in I, j \in J} a_{i, j} \sqrt t N(i, j, \alpha) =: A,
	\end{equation}
for any family $(a_{i,j})_{i \in I, j \in J}$ of scalars.
	We first establish that the first and second-order moments of $A_R$ converge to those of $A$. Second, we  demonstrate that the cumulants of $A$,  of sufficiently high order, tend to zero as $R$ approaches infinity. This is enough to prove the central limit theorem, as per a classical result attributed  to Marcinkiewicz; see e.g.  \cite[Lemma~3]{soshnikov2002gaussian}.

In addressing the first-order moment 
  we use the assumption that $\int_{\mathbb{R}^d} f_i(x) dx= 0$. Indeed,
	\begin{align*}
		\mathbb{E}\left[A_R\right] &=  \sum_{i \in I, j \in J} a_{i, j}  R^{\frac{\alpha - d}2 j} \int_{[-R, R]^d} f_i\left(\frac{x}{R^j}\right) dx   = \sum_{i \in I, j \in J} a_{i, j}  R^{\frac{\alpha + d}2 j} \int_{[-R^{1- j}, R^{1- j}]^d} f_i(x) dx. 
	\end{align*}
	Since $f_i \in \mathcal{S}(\mathbb{R}^d)$ and $j<1$, we obtain that $\mathbb{E}\left[A_R\right]\to 0 = \mathbb{E}[A]$ as $R\to\infty$.
 
 To address the second order moment of $A_R$, we use  Lemma \ref{lemma_cov}:
\begin{align*}
		\mathbb{E}\left[A_R^2\right] &= \sum_{i_1, i_2 \in I, j_1, j_2 \in J} a_{i_1, j_1} a_{i_2, j_2} 
c_R(j_1, j_2, i_1, i_2) + o(1),
\end{align*}
where $$c_R(j_1, j_2, i_1, i_2):= R^{\frac{d+\alpha}2(j_1+j_2)} \int_{\mathbb{R}^d} \mathcal{F}[f_{i_1}](R^{j_1} k)\overline{\mathcal{F}[f_{i_2}]}(R^{j_2} k) t|k|^{\alpha} dk.$$
	Without lost of generality, we assume that $j_2 \leq j_1$. With the change of variable $k \leftrightarrow R^{j_1} k$:
	\begin{align*}
		c_R(j_1, j_2, i_1, i_2) = R^{\frac{d+\alpha}2(j_2-j_1)} \int_{\mathbb{R}^d} \mathcal{F}[f_{i_1}](k) \overline{\mathcal{F}[f_{i_2}]}(R^{j_2 - j_1} k) t|k|^{\alpha} dk.
	\end{align*} 
	For $j_1 = j_2$, we get $c_R(j_1, j_2, i_1, i_2) =  \int_{\mathbb{R}^d} \mathcal{F}[f_{i_1}](k) \overline{\mathcal{F}[f_{i_2}]}( k) t|k|^{\alpha} dk$, and for $j_2 < j_1$:
	\begin{equation}\label{eq_correlations-j1-j2}
 |c_R(j_1, j_2, i_1, i_2)| \leq R^{\frac{d+\alpha}2(j_2-j_1)} t \|f_{i_2}\|_1 \left\|\mathcal{F}[f_{i_1}](\cdot) | \cdot|^{\alpha}\right \|_1 \xrightarrow[R\to \infty]{} 0.
 \end{equation}
	Accordingly, we have proved that the second order moment converges to the desired limit:
	$$\mathbb{E}\left[A_R^2\right] \xrightarrow[R \to \infty]{} \displaystyle\sum_{i_1,i_2 \in I, j \in J} a_{i_1, j}a_{i_2, j}  \int_{\mathbb{R}^d} \mathcal{F}[f_{i_1}](k) \overline{\mathcal{F}[f_{i_2}]}(k) t|k|^{\alpha} dk  = \mathbb{E}\left[A^2\right].$$

To consider the higher-order cumulants $C_m(A_R)$ of $A_R$ where $m> 2$, we recall the general representation of the cumulants of a random variable $X := \int_{\mathbb{R}^d} g(x)\,\Phi(dx)$, where $g$ is a real-valued, measurable function: 
\begin{equation}
C_m(X)=\sum_{\sigma\in\Pi[m]}
\int_{(\mathbb{R}^{d})^{|\sigma|}}\bigotimes_{i=1}^{|\sigma|}g^{|\sigma(i)|}\,
d\gamma^{(|\sigma|)}\,; \label{eq_cumulant-X}
\end{equation}
here, $\Pi[m]$ is the set of all unordered partitions of the set $\{1,\ldots,m\}$ (and, for a partition $\sigma\in\Pi[m]$ with $|\sigma|$ elements, we arbitrarily order them as $\sigma = \{\sigma(1),\ldots,\sigma(|\sigma| )\}$), $\bigotimes$ denotes the tensor product of functions, and $\gamma^{(r)}$, for $r\geq 1$, denotes the $r$-th order factorial cumulant moment measure of the point process $\Phi$. We refer to  \myappendix{} for further details. 

Note that the random variable  $A_R$ in~\eqref{eq_cv_cl} corresponds to  $X$ with the function 
$$g(x)=\mathbf{1}_{[-R, R]^d}(x)\sum_{i \in I, j \in J} a_{i, j} R^{\frac{\alpha - d}2 j} 
 f_i\left(\frac{x}{R^j}\right).$$
 Using the representation~\eqref{eq_cumulant-X} we have 
 $$|C_m(A_R)|\le \sum_{\sigma\in\Pi[m]} \left(\sum_{i \in I, j \in J} |a_{i, j}| R^{\frac{\alpha - d}2 j} 
 \|f_i\|_\infty\right)^{m}\times |\gamma^{(|\sigma|)}|\left(([-R,R]^d)^{|\sigma|}\right)\,,$$
 
where for $m' \geq 1$, $|\gamma^{(m')}|$ is the total variation of $\gamma^{(m')}$.
Brillinger-mixing condition  implies
$|\gamma^{(m')}|\left(([-R,R]^d)^{m'}\right)=O(R^d)$; see \myappendix{}. Since $\alpha  <d$, we have, for $R > 1$, $R^{\frac{\alpha - d}2 j} \leq R^{\frac{\alpha - d}2 j^*}$, where $j^* = \min\{j|~j \in J\} > 0$, thereby 
$$|C_m(A_R)|\le O(R^{\frac{j^*(\alpha-d)m}2)})\times O(R^d).$$
Consequently, for $m > 2d/(j^*(d - \alpha))$ we have 
$|C_m(A_R)| \to 0$ as $R \to \infty$, which completes the proof. 
 \end{proof}

\subsection{Scales limitations due to  border effects} \label{app_finite_size_eff}

 As mentioned in Remark~\ref{rmk_finite_size_effect}, the limitation of scales to $j < 1$ in assessing the variance rate of $R$-truncated wavelets $T_j(f, R)$, as $R \to \infty$, expressed in Proposition~\ref{prop_variance}, is crucial. 
 An intuitive explanation for this is that when \( j > 1 \), the support of the function \( x \mapsto f(x/R^j) \) eventually extends significantly beyond \( [-R, R]^d \), and the truncation in 
$T_j(f, R) = \sum_{x \in \Phi} \mathbf{1}_{[-R, R]^d}(x) f(x/R^j)$
prevents us from ``capturing'' the variability of the  data across the entire support of \( f \). Indeed, for \( j > 1 \), this truncation reproduces the variance rate of \( \operatorname{Var}\left[\sum_{x \in \Phi} \mathbf{1}_{[-R, R]^d}(x) f(0)\right] \) as \( R \) tends to infinity. 
This differs from the statement of Proposition~\ref{prop_variance}  when considering data from hyperuniformity classes II and III, as discussed in Remark~\ref{rmq_class_HU}.

To present a brief proof regarding this effect, we introduce a smoothly truncated wavelet transform as follows:
\[
\forall j > 0, \quad S_j(R, f, \chi) = \sum_{x \in \Phi} \chi\left(\frac{x}{R}\right) f\left(\frac{x}{R^j}\right),
\]
where \(\chi\) is infinitely differentiable with compact support in $[-1,1]^d$. Here, the function \(\chi(x/R)\) replaces the indicator function \(\mathbf{1}_{[-R, R]^d}\) used in Definition \ref{def_T_j}. The following proposition, especially \eqref{equiv var large scale 0}, demonstrates the previous claim. Using a smooth function $\chi$ instead of the indicator function has yet an impact on the rate of the variance in comparison with the results in Remark~\ref{rmq_class_HU}, as shown in \eqref{equiv var large scale}. Nonetheless \eqref{equiv var large scale 0} remains true for $\chi(x)  = \mathbf{1}_{[-R, R]^d}(x)$, but it requires more technical details that we prefer to skip.  
\begin{proposition}\label{prop_border_eff}
As in Proposition~\ref{prop_variance}, let  $\Phi$ satisfy Assumptions~\ref{ass_rho_leb} and~\ref{ass_S_0}, with intensity $\lambda=1$.
 Let $f \in \mathcal{S}(\mathbb{R}^d)$ with $f(0)\neq 0$, and contrary to Proposition~\ref{prop_variance}, assume that the scale  $j > 1$. Then, as $R\to\infty$, 
\begin{equation}\label{equiv var large scale 0}
 \operatorname{Var}[S_j(R, f,\chi)]\sim \operatorname{Var}\left[\sum_{x \in \Phi} \chi\left(\frac{x}{R}\right) f(0)\right].
\end{equation}
More specifically, we have in this setting 
 	\begin{align}\label{equiv var large scale}
		\lim_{R \to \infty} R^{(\alpha - d)} \operatorname{Var}[S_j(R, f,\chi)] &= |f(0)|^2 \int_{\mathbb{R}^d} |\mathcal{F}[\chi](k)|^2 t|k|^{\alpha} dk.
	\end{align}
 \end{proposition}
\begin{proof}
By inverting the role played by $f$ and $\chi$, we note that $S_j(R, f, \chi) = S_{1/j}(R^j, \chi, f)$. We denote $j' = 1/j$ and $R' = R^j$. The proof of Lemma~\ref{lemma_cov} can be adapted with almost no changes if $x \mapsto \mathbf{1}_{[-1, 1]^d}(x/{R'})$ is replaced therein by $x \mapsto f(x/R')$. Thus, as $j' < 1$, we obtain that
	\begin{align*}
		\lim_{R' \to \infty} (R')^{(\alpha-d)j'}\operatorname{Var}[S_{j'}(R', \chi, f)] &=   |f(0)|^2 \int_{\mathbb{R}^d} |\mathcal{F}[\chi](k)|^2 t|k|^{\alpha} dk.
	\end{align*}	
	Using that $(R')^{(\alpha-d)j'} = R^{\alpha - d}$, we get \eqref{equiv var large scale}. We finally note from the proof of Lemma~\ref{lemma_cov}, in particular the non-truncated case considered in \eqref{non truncated} where $f_1=f_2=\chi$ and $j_1=j_2=1$,  that the asymptotic behavior of  $\operatorname{Var}\left[\sum_{x \in \Phi} \chi\left(x/R\right) f(0)\right]$ is similar, proving  \eqref{equiv var large scale 0}. 
\end{proof}

\subsection{Non-zero truncated wavelets and estimator well-definedness}
\label{app_prop_no_atom_0}
In this section, we prove Proposition~\ref{prop_no_atom_0}, ensuring that $R$-truncated wavelets $T_j(f, R;\Phi)$ are almost surely non-null on a non-null realization of the point process $\Phi_R$, provided $f \in \mathcal{S}(\mathbb{R}^d)$ is analytic and non-null. This justifies the construction of the estimator $\widehat{\alpha}(I, J, R)$ in Definition~\ref{def_mult_scale_est} using the logarithmic function.

\begin{proof}[Proof of Propositon \ref{prop_no_atom_0}]
Our goal is to show that 
	$$\delta := \mathbb{P}\left(\sum_{x \in \Phi_R} f(x/R^j) = 0, |\Phi_R| \geq 1\right) = 0.$$
Without loss of generality, we can assume that $R^j = 1$ by considering $f(\cdot/R^j)$ instead of $f$. Utilizing the fact that $|\Phi_R| \geq 1$ and applying the union bound, we obtain:
	\begin{align*}
		\delta &\leq \mathbb{P}\left(\exists x \in \Phi_R,~f(x) + \sum_{y \in \Phi_R\setminus\{x\}} f(y) = 0\right)  \leq \mathbb{E}\left[\sum_{x \in \Phi_R}  \mathbf{1}_{\left\{f(x) + \sum_{y \in \Phi_R\setminus\{x\}} f(y) = 0\right\}}\right].
	\end{align*}
	We now apply the Campbell-Little-Mecke-Matthes theorem (see for instance  \cite{baccelli2020random}) to the right hand side to get
	\begin{align}
		\delta    \leq \mathbb{E}^{!0}\left[\int_{\mathbb{R}^d} \mathbf{1}_{[-R, R]^d}(x) \mathbf{1}_{\left\{f(x) + \sum_{y-x \in \Phi_R} f(y) = 0\right\}} dx \right],\label{eq-delta-E0}
	\end{align}
	where $\mathbb{E}^{!0}$ denotes the expectation under the reduced Palm probability associated to $\Phi$.

To complete the proof, it is enough to show that the set 
of $x\in \mathbb{R}^d$
which satisfy the equation 
\begin{equation}\label{eq-set-zeros}f(x) + \sum_{y \in \Phi}  \mathbf{1}_{[-R, R]^d}(y +x) f(y+x) = 0
\end{equation}
has  null Lebesgue measure, for any realisation of $\Phi$.
Note that this set is random, as it depends on the realization of $\Phi $ on $[-2R, 2R]^d$. 
Given $\Phi\cap [-2R, 2R]^d=\{y_1, \dots, y_n\}$ with $n\geq 0$ ($n=0$ corresponding to the empty set), denote  $B_\Phi := \{x \in [-R, R]^d \, | \, f(x) + \sum_{i = 1}^n \mathbf{1}_{[-R, R]^d}(y_i +x) f(y_i+x) = 0\}$.
To demonstrate that $B_\Phi$ has zero Lebesgue measure, we leverage the analyticity of the function $f$ and, by seeking a contradiction, assume that $B_\Phi$ has a positive Lebesgue measure. By the regularity of the Lebesgue measure \cite{royden1968real}, there exists a point $x_0 \in \mathbb{R}^d$ and a neighborhood of size $h \in \mathbb{R}^d\setminus\{0\}$ such that equation~\eqref{eq-set-zeros} is satisfied for all $x$ within it, that is
$$
\forall x \in [x_0 - h, x_0 +h]^d, \, f(x) + \sum_{i = 1}^n \mathbf{1}_{[-R, R]^d}(y_i +x) f(y_i+x) = 0.
$$
If necessary, we can reduce $|h|$ to obtain:
$$
\forall x \in [x_0 - h, x_0 +h]^d, \, f(x) + \sum_{i \in L_n} f(y_i+x) = 0,
$$
where $L_n$ is a subset of $\{1, \dots, n\}$, or $L_n = \emptyset$. As $f$ is analytic, we apply Lemma \ref{lemma_ana} to deduce that $f \equiv 0$, which contradicts the assumption that $f$ has at least one non-zero value. Consequently, $B_\Phi$ has zero Lebesgue measure and $\delta = 0$.
\end{proof}

\begin{lemma}\label{lemma_ana}
	Let $f \in \mathcal{S}(\mathbb{R}^d)$. We assume that $f$ is analytic and that there exists $n \geq 0$, $(y_1, \dots, y_n) \in (\mathbb{R}^d)^n$, $x_0 \in \mathbb{R}^d$ and $h \in \mathbb{R}^d\setminus\{0\}$ such that, 
	$$\forall x \in [x_0 - h, x_0+h]^d,~f(x) + \sum_{i = 1}^n f(y_i + x) = 0,$$
	with the convention $\sum_{i = 1}^0 f(y_i + x) = 0$. Then, $f \equiv 0$.
\end{lemma}
\begin{proof}
	Since $f$ is analytic, $x \mapsto f(x) + \sum_{i = 1}^n f(y_i + x)$ is also analytic. Consequently, by the Identity Theorem for analytic functions \cite{mityagin2020zero}, we obtain :
	$$\forall x \in \mathbb{R}^d,~f(x) + \sum_{i = 1}^n f(y_i + x) = 0.$$
	Taking the Fourier transform of the previous equation, we get:
	$$\forall x \in \mathbb{R}^d,~\mathcal{F}[f](x)\sum_{i = 1}^n e^{-\bm{i} x.y_i} = 0.$$
	Thus, $\mathcal{F}[f](x) = 0$ for all $x\in \mathbb{R}^d \setminus Z$ where $Z$ denotes the zero set of the function $h(x) = \sum_{i = 1}^n e^{-\bm{i} x.y_i}$. Because $h$ is analytic, using again the Identity Theorem, we obtain that $Z$ is discrete. Thus, by continuity of $\mathcal{F}[f]$, we deduce $\mathcal{F}[f] \equiv 0$, and so $f \equiv 0$.
\end{proof}

\subsection{Asymptotic results}
\label{app_cv}
In complement to the multivariate central limit theorem proved in Section~\ref{sec_clt}, namely Theorem~\ref{thm_clt}, we address in this  section the proof of Theorem~\ref{thm_cv_alpha} and of Corollary~\ref{cor_ic_mult_scala_est_taper}.
These results yield the consistency of our estimator $\widehat{\alpha}(I, J, R)$, and for the latter its asymptotic distribution derived from Theorem~\ref{thm_clt}.

\begin{proof}[Proof of Theorem \ref{thm_cv_alpha}.]
 We have to show that the random variables $\epsilon(I, J, R)$, defined in~\eqref{eq_decom_al}, 
	converge to $0$ in probability as $R \to \infty$. 
 Firstly, note that the indicator $\mathbf{1}_{|\Phi_R| = 0}$ converges to 0 in probability. This can be observed through the continuity of probability:
\[
\lim_{R\to \infty} \mathbb{P}(|\Phi_R| = 0) = \mathbb{P}(|\Phi| = 0) = 0,
\]
where the last equality is a consequence of the assumption that $\Phi$ is almost surely non-null (which is equivalent to being almost surely  infinite for a stationary point process).
 Now, let's shift our focus to the essential term in~\eqref{eq_decom_al}.

Let $\epsilon > 0$. We now consider: 	$$\mathbb{P}\left(\left|\frac{1}{\log(R)} \log\left(\sum_{i \in I} R^{(\alpha - d)j} T_j(f_i, R)^2\right)\right| \geq \epsilon\right) \leq p_{\infty}(R) + p_0(R),$$
	where $$p_{\infty}(R) = \mathbb{P}\left(\sum_{i \in I} R^{(\alpha - d)j} T_j(f_i, R)^2 \geq R^{\epsilon}\right),$$ 
	$$p_0(R) = \mathbb{P}\left(\sum_{i \in I} R^{(\alpha - d)j} T_j(f_i, R)^2 \leq R^{-\epsilon}\right).$$
	Markov inequality and Proposition \ref{prop_variance} ensure that $p_{\infty}(R) \to 0$ as $R \to \infty$. Concerning $p_0(R)$:
	$$p_0(R) \leq \mathbb{P}\left(R^{(\alpha - d)j} T_j(f_{i_j}, R)^2 \leq R^{-\epsilon}\right).$$
	Let $\delta > 0$. By right-continuity of $t \mapsto \mathbb{P}[X_j^2 \leq t]$, there exists $\mu > 0$, $\mathbb{P}[X_j^2 \leq \mu] \leq \delta$. Let $R_0 = \mu^{-1/\epsilon}$. Then, for all $R \geq R_0,~p_0(R) \leq \mathbb{P}\left(R^{(\alpha - d)j} T_j(f_{i_j}, R)^2 \leq \mu\right)$. Finally, using the convergence in distribution of $R^{(\alpha - d)j} T_j(f_{i_j}, R)^2$ toward $X_j^2$, we obtain, for all $R \geq R_1 \geq~R_0$, then $p_0(R) \leq \mathbb{P}[X_j^2 \leq \mu] + \delta \leq 2 \delta$. Thus, $p_{0}(R) \to 0$ as $R \to \infty$
\end{proof}

\begin{remark}\label{rmk_pb_0_tj}
	The proof of Theorem \ref{thm_cv_alpha} highlights the fact that the assumption concerning the convergence in distribution of one statistic $R^{(\alpha - d)j} T_j(f_{i_j}, R)^2$ toward a non atomic random variable allows one to control  $p_0(R)$ (defined in the proof of Theorem \ref{thm_cv_alpha}), and ensures that $R^{(\alpha - d)j} T_j(f_{i_j}, R)^2$ does not concentrate at $0$. Other assumptions that prevent such a behavior might lead to the consistency of $\widehat{\alpha}(I, J, R)$. For example, if there exists $\beta > 0$ such that for all $j \in J$ there exists $i_j \in I$, such that $\underset{R \to \infty}{\lim\sup}~\mathbb{E}[(R^{(\alpha - d)j} T_j(f_{i_j}, R)^2)^{-\beta} \mathbf{1}_{R^{(\alpha - d)j} T_j(f_{i_j}, R)^2) \leq 1}] < \infty$, then one can prove that, $\widehat{\alpha}(I, J, R)$ converges to $\alpha$ in $L^1(\mathbb{P})$.
\end{remark}

\begin{proof}[Proof of Corollary \ref{cor_ic_mult_scala_est_taper}.]
	Using  the decomposition~\eqref{eq_hat_a-a} with~\eqref{eq_decom_al} we have
	$$\mathbf{1}_{|\Phi_R| \geq 1}\log(R)\left\{\widehat{\alpha}(I, J, R) - \alpha\right\} = \mathbf{1}_{|\Phi_R| \geq 1} \sum_{j \in J} w_j \log\left(\sum_{i \in I} R^{(\alpha - d)j} T_j(f_i, R)^2\right).$$
Note that by Definition~\ref{def_mult_scale_est}, at least one function $f_i$ is not identically zero, so that the limiting distribution in the multivariate convergence stated in Theorem~\ref{thm_clt} is not degenerated. This convergence, combined  with the observation that the indicator $\mathbf{1}_{|\Phi_R| \geq 1}$ converges to 1 in probability (see the proof of Theorem~\ref{thm_cv_alpha}), yields by application of Slutsky's lemma:
	\begin{align*}
		\mathbf{1}_{|\Phi_R| \geq 1} \log(R)\left\{\widehat{\alpha}(I, J, R) - \alpha\right\} &\xrightarrow[R \to \infty]{Law} 1 \times\sum_{j \in J} w_j \log\left(\sum_{i \in I} t N(i, j, \alpha)^2\right)\\&= \log(t)\sum_{j\in J}w_j + \sum_{j \in J} w_j \log\left(\sum_{i \in I} N(i, j, \alpha)^2\right), \\&= 0 + \sum_{j \in J} w_j \log\left(\sum_{i \in I} N(i, j, \alpha)^2\right),
	\end{align*}
	where the last equality is a consequence of  \eqref{eq_sum_wj_j_0}. This completes the proof.
\end{proof}

\subsection{Confidence intervals results}\label{app_ICA}
In this section, we prove Proposition \ref{prop_ICA}. To do so, we introduce some notations and rely on auxiliary results formulated in Lemmas~\ref{lemma_diff_cov},
\ref{lemma_cont_F} and~\ref{lemma_diff_cdf}.

Denote by $Z$ a random variable representing the asymptotic distribution of $\log(R)(\widehat{\alpha} - \alpha)$, as given in Corollary \ref{cor_ic_mult_scala_est_taper}. Specifically,
\[ Z := \sum_{j \in J} w_j \log\left(\sum_{i \in I} N(i, j, \alpha)^2\right), \]
where $(N(i, j, \alpha))_{i \in I, j \in J}$ is a Gaussian vector with zero mean and covariance matrix $\Sigma$ defined in~\eqref{eq_cov_mat}. Let $F_Z(t) = \mathbb{P}(Z \leq t)$ represent the cumulative distribution function of $Z$, and $F_Z^{-1}(q) = \inf\{t\in \mathbb{R} \, | \, F_Z(t) > q\}$ represent its quantile function.
Similarly, let $Z_R(\beta)$ be the random variable defined in Proposition \ref{prop_ICA}, and $F_{Z_R(\beta)}(t)$ and $F^{-1}_{R}(q;\beta)$ be its cumulative distribution and quantile function, respectively. We also denote $\Sigma_R(\beta) :=\Sigma_R$ given by \eqref{eq_cov_mat_R}.

The proof of Proposition \ref{prop_ICA} is divided into several steps as follows:
\begin{itemize}
\item[(i)] Lemma~\ref{lemma_diff_cov}, which is the main technical result, demonstrates that $(\beta, R) \mapsto |\Sigma_R(\beta) - \Sigma|_{\infty}$ is uniformly continuous for $\beta \geq 0$ and $R$ in a neighborhood of $+ \infty$.
\item[(ii)] Before extending this result to the difference of quantile functions, we first show in Lemma \ref{lemma_cont_F} that both the cumulative distribution functions $F_Z$ and the quantile function $F_Z^{-1}$ are continuous.
\item[(iii)] Then, in Lemma~\ref{lemma_diff_cdf}, we extend Lemma~\ref{lemma_diff_cov} by proving that $(\beta, R) \mapsto \|F_{Z_R(\beta)} - F_Z\|_{\infty}$ is uniformly continuous for $\beta \geq 0$ and $R$ in a neighborhood of $+ \infty$. 
\item[(iv)] In the final step, within the proper proof of Proposition \ref{prop_ICA} and thanks to the previous lemmas, we deduce the continuity of $(\beta, R) \mapsto |F_{R}^{-1}(q, \beta) - F_Z^{-1}(q)|$ for $q \in (0, 1)$. \end{itemize} 

\begin{lemma}\label{lemma_diff_cov}
	We suppose the assumptions of Proposition \ref{prop_ICA}. Let $\epsilon > 0$. Then, there exists $\delta > 0$ such that for all $\beta \geq 0$ and $R > 0$ such that $|\alpha - \beta| \leq \delta$ and $R > \delta^{-1}$, then $|\Sigma_R(\beta) - \Sigma|_{\infty} \leq \epsilon$. 
\end{lemma}
\begin{proof}
	Let $0 < \delta < 1/2$, that will be possibly reduced at the end of the proof. Let $\beta \geq~0$ such that $|\alpha - \beta| \leq \delta$. We bound each coefficient of $\Sigma_R(\beta) - \Sigma$. We first consider the coefficients with $j_1 = j_2 = j$. Let $j \in J$, $i_1, i_2 \in I$, and:
	\begin{align*}
		\Delta_1 :=& (\Sigma_R(\beta))_{i_1, i_2, j, j} - (\Sigma)_{i_1, i_2, j, j}  \\=& R^{(\beta + d)j} \int_{\mathbb{R}^d} \mathcal{F}[f_{i_1}](R^{j} k) \overline{\mathcal{F}[f_{i_2}]}(R^{j} k) |k|^{\beta} dk 
		-  \int_{\mathbb{R}^d} \mathcal{F}[f_{i_1}]( k) \overline{\mathcal{F}[f_{i_2}]}( k) |k|^{\alpha} dk.
	\end{align*}
With the change of variable $k \leftrightarrow R^j k$, we get:
	$$|\Delta_1| \leq \int_{\mathbb{R}^d} \left |\mathcal{F}[f_{i_1}](k) \overline{\mathcal{F}[f_{i_2}]}(k)\right|  \left| |k|^{\beta} - |k|^{\alpha}\right | dk.$$
	Splitting the domain of integration whether $|k|\leq 1$ or $|k|>1$, and using the mean value inequality to $a\mapsto |k|^a$ on $[\alpha,\beta]$ (or $[\beta,\alpha]$ depending on the ordering) we obtain:
	\begin{align*}
		|\Delta_1| &\leq |\beta - \alpha| \Bigg(\int_{|k| \leq 1} \left |\mathcal{F}[f_{i_1}](k) \overline{\mathcal{F}[f_{i_2}]}(k)\right| |\log(|k|)| |k|^{\alpha - \delta} dk \\
		& \hspace{4cm} + \int_{|k| > 1} \left |\mathcal{F}[f_{i_1}](k) \overline{\mathcal{F}[f_{i_2}]}(k)\right| |\log(|k|)| |k|^{\alpha + \delta} dk\Bigg) \\& \leq C_1(i_1, i_2)\, \delta,
	\end{align*}
	with $C_1(i_1, i_2) = \int_{\mathbb{R}^d} |\mathcal{F}[f_{i_1}](k) \overline{\mathcal{F}[f_{i_2}]}(k)| |\log(|k|)| (|k|^{\alpha - 1/2} + |k|^{\alpha + 1/2}) dk$. The fact that $f_{i_1}$ and $f_{i_2}$ are Schwartz functions ensures that $C_1(i_1, i_2) < \infty$. We now consider the case where $j_1 \neq j_2$. Let $i_1, i_2 \in I$, $j_1, j_2 \in J$ with $j_1 < j_2$, and
	\begin{align*}
		\Delta_2 & := (\Sigma_R(\beta))_{i_1, i_2, j_1, j_2} - (\Sigma)_{i_1, i_2, j_1, j_2}  = R^{\frac{\beta  + d}2 (j_1 + j_2)} \int_{\mathbb{R}^d} \mathcal{F}[f_{i_1}](R^{j_1} k) \overline{\mathcal{F}[f_{i_2}]}(R^{j_2} k) |k|^{\beta} dk.
	\end{align*}
	With the change of variable $k \leftrightarrow R^{j_2} k$, we obtain:
	$$|\Delta_2| \leq R^{\frac{\beta + d}2 (j_1 - j_2)} \int_{\mathbb{R}^d} \left|\mathcal{F}[f_{i_1}](R^{j_1-j_2} k) \overline{\mathcal{F}[f_{i_2}]}(k)\right | |k|^{\beta} dk.$$
	Consequently:
	\begin{align*}
		|\Delta_2| &\leq R^{\frac{\alpha - \delta + d}2 (j_1 - j_2)}\left( \int_{|k| \leq 1} \|f_{i_1}\|_1 \left |\mathcal{F}[f_{i_2}](k)\right| dk + \int_{|k| \geq 1} \|f_{i_1}\|_1 \left|\mathcal{F}[f_{i_2}](k)\right | |k|^{\alpha + \delta} dk\right) \\& \leq R^{\frac{\alpha + d-1/2}2 (j_1 - j_2)} C_2(i_1, i_2),
	\end{align*}
	where $C_2(i_1, i_2)= \|f_{i_1}\|_1\left( \int_{|k| \leq 1} |\mathcal{F}[f_{i_2}](k)| dk + \int_{|k| \geq 1} |\mathcal{F}[f_{i_2}](k)| |k|^{\alpha + 1/2} dk\right)$. As $f_{i_1}$ and $f_{i_2}$ are Schwartz functions, $C_2(i_1, i_2) < \infty$. Finally, denoting by $C=\max_{i_1,i_2\in I}(C_1(i_1,i_2) + C_2(i_1,i_2))$,  the bounds on $\Delta_1$ and $\Delta_2$ and the fact that $R>1/\delta$ imply:
	\begin{align*}
	|\Sigma_R(\beta) - \Sigma|_{\infty}&\leq C\left(\delta + \max_{j_1, j_2 \in J,\, j_1 < j_2} R^{\frac{\alpha + d -1/2}2 (j_1 - j_2)}\right)\\&\leq C\left(\delta + \max_{j_1, j_2 \in J,\, j_1 < j_2} \delta^{\frac{\alpha + d -1/2}2 (j_2 - j_1)}\right).
	\end{align*}
	The proof is concluded by choosing $0 < \delta < 1/2$ small enough so that the latter bound is less than $\epsilon$, which is possible since $J$ is finite and  $\alpha + d-1/2 \geq \alpha +1/2 > 0$. 
	\end{proof}

\begin{lemma}\label{lemma_cont_F}
	Under the setting of Proposition \ref{prop_ICA},  $F_Z$ and $F_{Z}^{-1}$ are continuous.
\end{lemma}
\begin{proof}
	Before proving both results, we introduce a useful representation of $Z$ with independent random variables. Let $(\lambda(i, j))_{i \in I, j \in J}$ be the non-negative eigenvalues of the covariance matrix $\Sigma$ defined in \eqref{eq_cov_mat}. Then:
	\begin{equation}\label{representation Z}
Z \overset{Law}{=} \sum_{j \in J} w_j \log V_j,
	\end{equation}
	where $$V_j=\sum_{i \in I} \lambda(i, j) U(i,j),$$
	and $(U(i,j))_{i \in I, j \in J}$ are i.i.d. $\chi^2(1)$ random variables.
Note that for all $j \in J$:
	$$\sum_{i \in I} \lambda(i, j) =  \sum_{i \in I} \int_{\mathbb{R}^d} \left |\mathcal{F}[f_i](k)\right|^2 |k|^{\alpha} dk > 0.$$
As a consequence, for any $j \in J$, there exists at least one index $i_j \in I$ such that $\lambda(i_j, j) > 0$, so that $V_j$ is a non-degenerated continuous random variable. We deduce that the random variables $\log V_j$,  $j\in J$, are independent continuous random variables. Since there exists  $j_0$ such that $w_{j_0} \neq 0$, $Z$ is also a continuous random variable and $F_Z$ is continuous.

	We now prove that $F_Z^{-1}$ is continuous. According to the properties of the quantile function, it suffices to prove that $t \mapsto F_Z(t)$ is strictly increasing, i.e. 
	for any $t_1 < t_2$, $F_{Z}(t_2) - F_{Z}(t_1) = \mathbb{P}(t_1 <  Z \leq t_2)>0$. Let $j_0$ such that $w_{j_0} \neq 0$, and let $0 < \epsilon < 1/2$ that will be chosen small enough in the following. 
\begin{align*}
		\mathbb{P}(t_1 < Z \leq t_2) &= \mathbb{P}\left(t_1 < \sum_{j \in J} w_j \log V_j \leq t_2\right)  \\
		&\geq \mathbb{P}\left(t_1 < \sum_{j \in J} w_j \log V_j \leq t_2,\ \forall j \neq j_0,~ \left|V_j- 1\right| \leq \epsilon\right).
	\end{align*}
	According to the mean value inequality, for all $y \in \mathbb{R}$ such that $|y-1| \leq \epsilon < 1/2$ then $|\log(y)| \leq 2 \epsilon$. Therefore, if for all $ j \neq j_0$, $\left|V_j - 1\right| \leq \epsilon$, then $\left|\sum_{j \in J, j \neq j_0} w_j \log V_j \right| \leq 2 \epsilon \sum_{j \in J, j \neq j_0} |w_j|$, so that
	\begin{multline*}
		\mathbb{P}(t_1< Z \leq t_2)  \geq \\ \mathbb{P}\left(t_1 +  2 \epsilon \sum_{j \in J, j \neq j_0} |w_j| < w_{j_0} \log V_{j_0} \leq t_2- 2 \epsilon \sum_{j \in J, j \neq j_0} |w_j|, \ \forall j \neq j_0,~ \left|V_j - 1\right| \leq \epsilon\right).
	\end{multline*}
Let $\epsilon$ be small enough such that $2 \epsilon \sum_{j \in J, j \neq j_0} |w_j| \leq (t_2 -t_1)/4$. Then, by independence,
	\begin{align*}
		\mathbb{P}(t_1 < Z \leq t_2) & \geq \mathbb{P}\Big(\frac{3t_1+t_2}4 < w_{j_0} \log V_{j_0} \leq \frac{t_1+3t_2}4\Big) \times	\prod_{j \neq j_0} \mathbb{P}\left(\left|V_j - 1\right| \leq \epsilon\right) \\
		& \geq \mathbb{P}\left(t_3 \leq  V_{j_0} \leq t_4\right)\prod_{j \neq j_0} \mathbb{P}\left(1 - \epsilon \leq V_j \leq 1+ \epsilon\right),
	\end{align*}
	where $t_3 =u\wedge v$ and $t_4=u\vee v$, with $\log u=(3t_1+t_2)/(4w_{j_0})$ and 
	$\log v=(t_1+3t_2)/(4w_{j_0})$. We use similar arguments to prove that each probability in the above lower-bound is positive. Note that if  for all $i \neq i_{j_0}$, $U(i, j_0) \leq \epsilon$, then $V_{j_0} \leq \epsilon \sum_{i \in I, i \neq  i_{j_0}} \lambda(i, j_0)$. Therefore, as before, we may choose  $\epsilon$ small enough to obtain
	\begin{align*}
		\mathbb{P}\left(t_3 \leq  V_{j_0} \leq t_4\right) &\geq \mathbb{P}\left(\frac{3 t_3+t_4}{4 \lambda(i_{j_0}, j_0)} \leq U(i_{j_0},j_0) \leq \frac{t_3+3t_4}{4 \lambda(i_{j_0}, j_0)}\right)  \prod_{i \neq j_0} \mathbb{P}\left(U(i, j_0) \leq \epsilon\right),
	\end{align*}
	which is a positive quantity  since $(U(i,j))_{i \in I, j \in J}$ are i.i.d. $\chi^2(1)$. We may prove similarly that 	$\prod_{j \neq j_0} \mathbb{P}\left(1 - \epsilon \leq V_j \leq 1+ \epsilon\right)
 > 0$.  We thus obtain that for all $t_1 < t_2$, $F_Z(t_2) - F_2(t_1) > 0$, concluding the proof.
\end{proof}

\begin{lemma}\label{lemma_diff_cdf}
	We suppose the assumptions of Proposition \ref{prop_ICA}. Let $\epsilon > 0$. There exists $\delta > 0$ such that for all $\beta \geq 0$ and $R > 0$ such that $|\alpha - \beta| \leq \delta$ and $R > \delta^{-1}$, then:
	$$\|F_{Z_R(\beta)} - F_{Z}\|_{\infty} \leq \epsilon.$$
\end{lemma}
\begin{proof}
	We start with the same representation as in the beginning of the proof of Lemma~\ref{lemma_cont_F}. 
	Let $(\lambda(i, j, R, \beta))_{i \in I, j \in J}$ (resp.  $(\lambda(i, j))_{i \in I, j \in J}$)  be the non-negative eigenvalues of the covariance matrix $\Sigma_R$ (resp. $\Sigma$) defined in \eqref{eq_cov_mat_R} (resp. \eqref{eq_cov_mat}). We order them using the lexicographic order over $I \times J$, i.e., for instance:
	$$\lambda(i_1, j_1) \leq \dots \leq\lambda(i_{|I|}, j_1) \leq \lambda(i_1, j_2) \leq \dots \leq \lambda(i_{|I|}, j_{|J|}).$$
	Similar to the representation of $Z$ in \eqref{representation Z}, we also have the following equality in distribution:
	$$Z_R(\beta) \overset{Law}{=} \sum_{j \in J} w_j \log\left(\sum_{i \in I} \lambda(i, j, R, \beta) U(i,j)\right),$$
	where $(U(i,j))_{i \in I, j \in J}$ are i.i.d. $\chi^2(1)$ random variables. Moreover, we proved in the beginning of the proof of Lemma~\ref{lemma_cont_F} that for all $j \in J$ there exists at least one index $i_j \in I$ such that $\lambda(i_j, j) > 0$.
	Let $t \in \mathbb{R}$. Using the previous representations,  we write $F_{Z_R(\beta)}(t)$ as:
	\begin{align}\label{fdr Z_R}
		F_{Z_R(\beta)}(t) 
		& =  \mathbb{P}\left(Z \leq t + \Delta(R, \beta)\right),
	\end{align}
	where 
	\begin{align*}
		\Delta(R, \beta) &= \sum_{j \in J} w_j \left(\log\left(\sum_{i \in I} \lambda(i, j) U(i,j)\right)- \log\left(\sum_{i \in I} \lambda(i, j, R, \beta) U(i,j)\right)\right) \\& = -\sum_{j \in J} w_j \log\left( 1+ \frac{\sum_{i \in I} (\lambda(i, j, R, \beta) - \lambda(i, j))U(i,j)}{\sum_{i \in I} \lambda(i, j) U(i,j)}\right).
	\end{align*}
	Let $\epsilon_1 > 0$ that will be chosen at the end of the proof. According to Lemma \ref{lemma_diff_cov}, there exists $\delta$ such that for $|\alpha - \beta| \leq \delta$ and $R > \delta^{-1}$, then $\|\Sigma_R(\beta) -\Sigma\|_{\infty} \leq \epsilon_1 |I|^{-2} |J|^{-2}$. Since $(\lambda(i, j))_{i \in I, j \in J}$ and $(\lambda(i, j, R, \beta))_{i \in I, j \in J}$ are both ordered in the same way, we can apply  Corollary~6.3.8 of~\cite{horn2012matrix} to get
	$$\sum_{i \in I, j \in J}|\lambda(i, j, R, \beta) - \lambda(i, j)|^2 \leq \sum_{i_1, i_2 \in I, j_1, j_2 \in J} \left((\Sigma_R(\beta) -\Sigma)_{i_1, i_2, j_1, j_2}\right)^2 \leq \epsilon_1^2.$$
	We introduce $I_0(j) = \{i \in I| \lambda(i, j) > 0\}$. Recall that $I_0(j)$ is non-empty. Consequently, 
	\begin{align*}
		\left|\frac{\sum_{i \in I} (\lambda(i, j, R, \beta) - \lambda(i, j))U(i,j)}{\sum_{i \in I} \lambda(i, j) U(i,j)}\right| &\leq \frac{\left(\sum_{i \in I, j \in J} |\lambda(i, j, R, \beta) - \lambda(i, j)|^2\right)^{1/2}\sum_{i \in I} U(i,j)}{\sum_{i \in I} \lambda(i, j) U(i,j)} \\& \leq \frac{\epsilon_1}{\min_{i \in I_0(j)} \lambda(i, j)} \frac{\sum_{i \in I} U(i,j)}{\sum_{i \in I_0(j)} U(i,j)} = \epsilon_1 \, Y_j,
		\end{align*} 
where 
\begin{equation}\label{def Y}
 Y_j = \frac{1}{\min_{i \in I_0(j)} \lambda(i, j)} \left(1 + \frac{\sum_{i \notin I_0(j)} U(i,j)}{\sum_{i \in I_0(j)} U(i,j)}\right).
 \end{equation}
If for all $j \in J$, $|Y_j| \leq \epsilon_1^{-1/2}$, then 
$$\left|\frac{\sum_{i \in I} (\lambda(i, j, R, \beta) - \lambda(i, j))U(i,j)}{\sum_{i \in I} \lambda(i, j) U(i,j)}\right| \leq \epsilon_1^{1/2}.$$
Let us choose $\epsilon_1$ so that $\epsilon_1^{1/2} < 1/2$. Using the fact that for all $y \in \mathbb{R}$ such that $|y-1| \leq \epsilon_1^{1/2}< 1/2$ then $|\log(y)| \leq 2 \epsilon_1^{1/2}$, we deduce:
\begin{equation}\label{ineq Delta}
|\Delta(R, \beta)| \leq 2 \epsilon_1^{1/2} \sum_{j \in J} |w_j|.
\end{equation}
Note that for all $j\in J$, $Y_j$ follows (up to a constant) a Fisher-Snedecor distribution with $(|I\setminus I_0(j)|,|I_0(j)|)$ degrees of freedom. Since $|I_0(j)|\geq 1$, $\mathbb{E}[Y_j^{1/4}] < \infty$.
From \eqref{fdr Z_R}, using~\eqref{ineq Delta} and the Markov inequality,  we deduce
\begin{align}\label{bound F_R}
		F_{Z_R(\beta)}(t) &\leq \mathbb{P}\left(Z \leq t + \Delta(R, \beta) ,~\forall j \in J, |Y_j| \leq \epsilon_1^{-1/2}\right) + \sum_{j \in J} \mathbb{P}\left(Y_j \geq \epsilon_1^{-1/2}\right) \nonumber 
		\\& \leq  \mathbb{P}\left(Z  \leq t + 2 \epsilon_1^{1/2} \sum_{j \in J} |w_j|\right) + \sum_{j \in J} \mathbb{E}[Y_{j}^{1/4}]\epsilon_1^{1/8}. 
	\end{align}
	Let $\epsilon>0$. Using the fact that $F_Z$ is continuous  (see Lemma \ref{lemma_cont_F}) with limits at $- \infty$ and $+\infty$, one can prove that $F_Z$ is uniformly continuous. As a consequence, there exists $\delta_2 > 0$ such that for all $|t_1-t_2| \leq \delta_2$ then:
	$$|F_{Z}(t_1) - F_{Z}(t_2)| \leq \epsilon/2.$$
	Take $\epsilon_1$  small enough such that $2 \epsilon_1^{1/2} \sum_{j \in J} |w_j| \leq \delta_2,$ and $\sum_{j \in J} \mathbb{E}[Y_{j}^{1/4}]\epsilon_1^{1/8} \leq \epsilon/2$. We obtain from \eqref{bound F_R}, for all $|\beta - \alpha| \leq \delta = \delta(\epsilon_1)$ and $R > \delta^{-1}$:
	$$\forall t \in \mathbb{R},~F_{Z_R(\beta)}(t) \leq F_Z\left(t + 2 \epsilon_1^{1/2} \sum_{j \in J} |w_j|\right) + \sum_{j \in J} \mathbb{E}[Y_{j}^{1/4}]\epsilon_1^{1/8} \leq F_Z(t) + \epsilon.$$
	We may prove similarly the reverse inequality. Indeed, for all $|\beta - \alpha| \leq \delta$ and $R > \delta^{-1}$:
	\begin{align*}
		F_{Z_R(\beta)}(t) &\geq \mathbb{P}\left(Z \leq t + \Delta(R, \beta) ,~\forall j \in J, |Y_j| \leq \epsilon_1^{-1/2}\right)  
		\\& \geq  \mathbb{P}\left(Z  \leq t - 2 \epsilon_1^{1/2} \sum_{j \in J} |w_j|, ~\forall j \in J, |Y_j| \leq \epsilon_1^{-1/2}\right) \\& = \mathbb{P}\left(Z  \leq t - 2 \epsilon_1^{1/2} \sum_{j \in J} |w_j|\right) - \mathbb{P}\left(Z  \leq t - 2 \epsilon_1^{1/2} \sum_{j \in J} |w_j|, ~\exists j \in J, |Y_j| > \epsilon_1^{-1/2}\right) \\& \geq  F_Z\left(t - 2 \epsilon_1^{1/2} \sum_{j \in J} |w_j|\right) - \sum_{j \in J} \mathbb{P}\left(|Y_j| > \epsilon_1^{-1/2}\right) \geq F_Z(t) - \epsilon,
	\end{align*} 
	which concludes the proof.
\end{proof}

We can now prove Proposition \ref{prop_ICA}.

\begin{proof}[Proof of Proposition \ref{prop_ICA}.]
	We denote for brevity $\widehat{\alpha}_+  = \widehat{\alpha}(I, J, R) \vee 0$ and we let
	$$p_R := \mathbb{P}[F_R^{-1}(a/2; \widehat{\alpha}_+) \mathbf{1}_{|\Phi_R| \geq 1} \leq \mathbf{1}_{|\Phi_R| \geq 1} \log(R)(\widehat{\alpha} - \alpha) \leq \mathbf{1}_{|\Phi_R| \geq 1}F_R^{-1}(1-a/2; \widehat{\alpha}_+)].$$
	We have to prove that $p_R \to 1 -a$ as $R \to \infty$. 
	We can rewrite $p_R$ as:
	$$p_R = \mathbb{P}(|\Phi_R| = 0) + \mathbb{P}[F_R^{-1}(a/2; \widehat{\alpha}_+) \leq  \log(R)(\widehat{\alpha} - \alpha) \leq  F_R^{-1}(1-a/2; \widehat{\alpha}_+), |\Phi_R| \geq 1].$$ 
	
	Let $q \in \{a/2, 1-a/2\}$. We gather the previous lemmas to show that $F_R^{-1}(q; \widehat{\alpha}_+)$ converges to $F_Z^{-1}(q; \alpha)$ in probability. 
		Let $\epsilon > 0$. According to Lemma \ref{lemma_cont_F}, $F_Z^{-1}$ is continuous, so there exists $\epsilon_2 > 0$ such that:
	\begin{equation}\label{eq_cont_quantile}
		F_Z^{-1}(q) - \epsilon \leq F_Z^{-1}(q - \epsilon_2) \leq  F_Z^{-1}(q + \epsilon_2) \leq F_Z^{-1}(q) + \epsilon.
	\end{equation}
Consider the parameter $\delta >0$ given by Lemma \ref{lemma_diff_cdf} applied to $\epsilon_2$ and let $\beta \geq 0$. Accordingly, for all  $|\beta - \alpha| \leq \delta$ and $R > \delta^{-1}$:
	$$\forall t \in \mathbb{R},~F_Z(t) - \epsilon_2 \leq F_{Z_R(\beta)}(t) \leq F_Z(t) +  \epsilon_2,$$
	that implies by definition of the quantile function:
	$$F_Z^{-1}(q - \epsilon_2) \leq F_R^{-1}(q; \beta) \leq F_Z^{-1}(q + \epsilon_2).$$
	Using \eqref{eq_cont_quantile}, we deduce that for all  $|\beta - \alpha| \leq \delta$ and $R > \delta^{-1}$:
	$$F_Z^{-1}(q) - \epsilon \leq F_Z^{-1}(q - \epsilon_2) \leq F_R^{-1}(q; \beta) \leq F_Z^{-1}(q + \epsilon_2) \leq F_Z^{-1}(q) + \epsilon. $$
	We have thus proved that for all $\epsilon > 0$, there exists $\delta$ such that: 
	$$(|\beta - \alpha| \leq \delta \text{ and } R > \delta^{-1}) \implies |F_R^{-1}(q; \beta) - F_{Z}^{-1}(q)| \leq \epsilon.$$
	The contraposition of the previous implication gives the convergence in probability of $F_R^{-1}(q; \widehat{\alpha}_+)$ toward $F_{Z}^{-1}(q)$. Indeed, for all $\epsilon> 0$, according to Corollary \ref{cor_ic_mult_scala_est_taper}:
	$$\underset{R \to \infty}{\operatorname{\lim\sup}}~ \mathbb{P}(|F_R^{-1}(q; \widehat{\alpha}_+) - F_{Z}^{-1}(q)| > \epsilon) \leq \underset{R \to \infty}{\operatorname{\lim\sup}}~ (\mathbb{P}(|\widehat{\alpha} - \alpha| > \delta) + \mathbf{1}\{R < \delta^{-1}\}) = 0.$$
	Finally, we leverage again  the fact that the probability $\mathbb{P}(|\Phi_R| = 0)$ converges to 0 (as shown in the proof of Theorem~\ref{thm_cv_alpha}), along with Corollary \ref{cor_ic_mult_scala_est_taper} and Slutsky's lemma, to obtain:
	$$\lim_{R \to \infty} p_R = 0 + \mathbb{P}(F_{Z}^{-1}(a/2) \leq Z \leq F_{Z}^{-1}(1 -a/2)) = 1 - a.$$ 
\end{proof}

\subsection{Variance and bias with Hermite tapers}
\label{app_mult_taper}

The main objective of this section is to prove Proposition \ref{prop_biais_var}, wherein we utilize Hermite tapers~\eqref{eq_herm} to establish non-asymptotic bounds on both the bias and variance of the estimator $\widehat{\alpha}(I, J, R)$.
The proof is based on Lemma~\ref{lemma_frac_norm_herm}, that  provides a control of  the fractional moments of the Hermite wavelets, on 
 Lemma~\ref{lemma_bound_sum_herm}, that investigates the $i$-dependence of the asymptotic variance of $T_j(\psi_i, R)$ (see Proposition \ref{prop_variance}), and on Lemma~\ref{lemma_remain_herm}, that quantifies the localization properties of the Hermite function by upper bounding their $L^2$ tail.

\begin{lemma}\label{lemma_frac_norm_herm}
	For all $\nu \geq 0$, there exists $0 < c_{\nu} \leq C_{\nu} < \infty$ such that, for all $|i| \neq 0$: 
	\begin{equation}\label{eq_frac_mom_herm_annexe}
		c_{\nu} |i|^{\nu/2} \leq \int_{\mathbb{R}^d} |\psi_i(k)|^2 |k|^{\nu} dk \leq C_{\nu} |i|^{\nu/2}.
	\end{equation}
\end{lemma}
\begin{proof}
	We first focus on the upper bound. It suffices to prove it for $\nu = 2k$ where $k \in \mathbb{N}$. Indeed, if $\nu$ is not an even integer we write $\nu = t \times 2k_{\nu} + (1- t) (2k_{\nu} +2)$ with $k_{\nu} \in \mathbb{N}$ and $t \in (0, 1)$. Then, we deduce the general case from the case where $\nu$ is even, using  H\"older inequality:
	\begin{align*}
		\int_{\mathbb{R}^d} |\psi_i(k)|^2 |k|^{\nu} dk &\leq \left(\int_{\mathbb{R}^d} |\psi_i(k)|^2 |k|^{2k_{\nu}} dk  \right)^{t} \left(\int_{\mathbb{R}^d} |\psi_i(k)|^2 |k|^{ 2k_{\nu}+2} dk  \right)^{1 - t} \\& \leq C_{2k_{\nu}}^t C_{2k_{\nu}+2}^{1 - t} |i|^{\frac{t}2 \times 2k_{\nu} + \frac{(1- t)}2 (2k_{\nu} +2)} = C_{2k_{\nu}}^t C_{2k_{\nu}+2}^{1 - t} |i|^{\nu/2}.
	\end{align*}
	We also note that it suffices to prove the inequality when $d = 1$. Indeed, for $x\in\mathbb R^d$ and $i = (i_1, \dots, i_d) \in \mathbb{N}^d$,   
$\psi_i(x_1, \dots, x_d) = \prod_{l = 1}^d \psi_{i_l}^1(x_l)$ where $\psi_{i_l}^1(x_l)=e^{-x_l^2/2} H_{i_l}(x_l)$. Accordingly, using the fact that $\|\psi_{i_l}^1\|_2 = 1$, we deduce the $d \geq 2$ case from the case $d =1$:
	\begin{align*}
		\int_{\mathbb{R}^d} |\psi_i(k)|^2 |k|^{\nu} dk & \leq d^{\nu/2} \int_{\mathbb{R}^d} |\psi_i(k)|^2 (|k_1|^{\nu} + \dots + |k_d|^{\nu}) dk \\& \leq d^{\nu/2} \sum_{l = 1}^d \int_{\mathbb{R}} |\psi_{i_l}^1(k_l)|^2 |k_l|^{\nu} dk_l \prod_{l' = 1, \dots, d}^{l' \neq l} \int_{\mathbb{R}} |\psi_{i_{l'}}^1(k_{l'})|^2 dk_{l'} \\& \leq d^{\nu/2} \sum_{l = 1}^d C_{\nu}  |i_l|^{\nu/2} \leq d^{\frac{2\nu+\nu}4} C_{\nu} |i|^{\nu/2}.
	\end{align*}
	Consequently, we suppose that $d = 1$ and $\nu = 2k$. We want to prove that for all $k \geq 0$ and $i \neq 0$, then $\int_{\mathbb{R}} |\psi_i(x)|^2 |x|^{2 k} dx \leq C_{2k} i^{k}$. Note that it is sufficient to prove it for $i \geq k$, the remaining cases being finite in number, so uniformly bounded.  Assuming $i\geq k$, we use the following inequality, that comes from a standard recursive relation for the Hermite polynomials, see, e.g., \cite{szeg1939orthogonal}, 
	\begin{align*}
		\left|x \psi_{i}(x)\right|^2 &= \left|\sqrt{\frac{i}2} \psi_{i-1}(x) + \sqrt{\frac{i+1}2} \psi_{i+1}(x)\right|^2 \leq 2 \frac{i}2 |\psi_{i-1}(x)|^2 + 2\frac{i+1}2 |\psi_{i+1}(x)|^2.
	\end{align*}
Hence 	
\begin{align*}
		\int_{\mathbb{R}} |\psi_i(x)|^2 |x|^{2 k} dx   \leq  i \int_{\mathbb{R}} |\psi_{i-1}(x) x^{k-1} |^2 dx  +  (i+1)\int_{\mathbb{R}} |\psi_{i+1}(x) x^{k-1} |^2 dx.
\end{align*}
Since $i\geq k$, we can iterate this inequality, using in the last step $\|\psi_i\|_2= 1$, to obtain
$$\int_{\mathbb{R}} |\psi_i(x)|^2 |x|^{2 k} dx   \leq P_k(i),$$
where  $P_k$ is a polynomial of degree $k$. This yields the upper bound $\int_{\mathbb{R}} |\psi_i(x)|^2 |x|^{2 k} dx \leq C_{2k} i^{k}$, where $C_{2k}$ depends on both $P_k$ and  the uniform bound for the $k-1$ terms associated to $1\leq i\leq k-1$.

	We now prove the lower bound. Arguing as for the upper bound, the general case is deduced from the $d = 1$ case. Consequently, we consider $d = 1$. With the change of variables $x = \sqrt{2i+1} \cos(\phi)$, we have:
	\begin{align*}
		\int_{\mathbb{R}} |\psi_i(x)|^2 |x|^{\nu} dx &\geq \int_{-\frac{\sqrt{2}}2 \sqrt{2i+1}}^{\frac{\sqrt{2}}2 \sqrt{2i+1}} |\psi_i(x)|^2 |x|^{\nu} dx \\& = \int_{\frac{\pi}4}^{\frac{3\pi}4} |\psi_i(\sqrt{2i+1} \cos(\phi))|^2 |\cos(\phi)|^{\nu} (2i+1)^{\nu/2} \sin(\phi) \sqrt{2i+1} d\phi.
	\end{align*}
	Then, we use the following bound (Theorem 8.22.9 of~\cite{szeg1939orthogonal}), valid for all $\phi \in \left[\frac{\pi}4, \frac{3\pi}4\right]$:
	$$\left|\psi_i(\sqrt{2i+1} \cos(\phi)) - \frac{(-1)^i 2^{1/4}}{\pi^{1/2} i^{1/4}} \frac{\sin\left(\frac{3\pi}{4} + (\frac{i}2+\frac{1}4)(\sin(2\phi) - 2 \phi)\right)}{\sin(\phi)^{1/2}}\right| \leq \frac{C}{i+1},$$
	where $C> 0$ does not depend on $i$ and $\phi$, and may change in the following from line to line. Using the reverse triangular inequality for the $L^2(\pi/4, 3\pi/4)$ Lebesgue space with weight $\phi \mapsto |\cos(\phi)|^{\nu} \sin(\phi)$, we get:
	\begin{align*}
		\left(\int_{\mathbb{R}} |\psi_i(x)|^2 |x|^{\nu} dx\right)^{1/2} &\geq C'\left(i^{-1/2}(2i+1)^{\frac{\nu+1}2} J(i)\right)^{1/2} \\& \hspace{2cm} - \left(\int_{\pi/4}^{3\pi/4}  C^2 \frac{(2i+1)^{\frac{\nu+1}2}}{(i+1)^2} |\cos(\phi)|^{\nu} \sin(\phi) d\phi\right)^{1/2} \\&  \geq C i^{\nu/4}\left(J(i)^{1/2} - i^{-3/4}\right),
	\end{align*}
	where $C'>0$ is another generic constant and  $$J(i) = \int_{\pi/4}^{3\pi/4} \sin^2\left(\frac{3\pi}{4} + (\frac{i}2+\frac{1}4)(\sin(2\phi) - 2 \phi)\right) |\cos(\phi)|^{\nu} d\phi.$$ 
	We can write the latter  integral as
	$$J(i) = \frac12 \int_{\pi/4}^{3\pi/4} |\cos(\phi)|^{\nu} d\phi - \frac12 \int_{\pi/4}^{3\pi/4} \cos\left(3\pi+ 2(i+\frac{1}2)(\sin(2\phi) - 2 \phi)\right) |\cos(\phi)|^{\nu} d\phi.$$
	The change of variable $\theta = \sin(2 \phi) - 2\phi$ and an integration by part ensures that the second term above goes to $0$ as $i \to \infty$. Therefore, there exists $i_0=i_0(\nu) \geq 1$ such that for all $i \geq i_0$:
	\begin{align*}
		\left(\int_{\mathbb{R}} |\psi_i(x)|^2 |x|^{\nu} dx\right)^{1/2} &  \geq C i^{\nu/4}\left( \left(\frac14\int_{\pi/4}^{3\pi/4} |\cos(\phi)|^{\nu} d\phi\right)^{1/2} - i^{-3/4}\right)  \geq  C' i^{\nu/4}.
	\end{align*}
Up to the modification of the constant, we may gather in the same inequality the remaining terms associated to $i<i_0$, to obtain for all $i \geq 1$ and for some $c_\nu>0$:
	$$\int_{\mathbb{R}} |\psi_i(x)|^2 |x|^{\nu} dx \geq~c_{\nu} i^{\nu/2}.$$
\end{proof}

\begin{lemma}\label{lemma_bound_sum_herm}
	Let $\alpha \geq 0$ and $(\psi_i)_{i \in I}$ be a family of Hermite wavelets given by~\eqref{eq_herm}. Then, there exist two constants $0 < c \leq C < \infty$ such that: 
	\begin{equation}\label{eq_sum_mom_herm}
		c \sum_{i \in I} |i|^{\alpha/2} \leq \sum_{i \in I}   \int_{\mathbb{R}^d} |\mathcal{F}[\psi_i](k)|^2 |k|^{\alpha} dk \leq C \sum_{i \in I} |i|^{\alpha/2},\end{equation}
	\begin{equation}\label{eq_sum_mom_herm_2}
		\sum_{i_1 , i_2 \in I} \left(\int_{\mathbb{R}^d} \mathcal{F}[\psi_{i_1}](k) \overline{\mathcal{F}[\psi_{i_2}]}(k) |k|^{\alpha} dk\right)^2 \leq  C\sum_{i \in I} |i|^{\alpha}.
	\end{equation}
\end{lemma}  

\begin{proof}
	If $I$ does not contain $0$, then the inequalities in \eqref{eq_sum_mom_herm} are obtained by summing over $i \in I \setminus\{0\}$ the inequalities~\eqref{eq_frac_mom_herm_annexe} of Lemma~\ref{lemma_frac_norm_herm}, since for all $k$, $|\mathcal{F}[\psi_i](k)|=|\psi_i(k)|$. If $I$ contains $0$, we start from these inequalities and we simply add the term $i = 0$: 
	\begin{align*}
		c_0 + c \sum_{i \in I\setminus\{0\}} |i|^{\alpha/2} \leq \sum_{i \in I}   \int_{\mathbb{R}^d} |\mathcal{F}[\psi_i](k)|^2 |k|^{\alpha} dk \leq C \sum_{i \in I\setminus\{0\}} |i|^{\alpha/2} + c_0,
	\end{align*}
	where $c_0 = \int_{\mathbb{R}^d} |\mathcal{F}[\psi_0](k)|^2 |k|^{\alpha} dk$. The latter upper-bound is less than $(C + c_0) \sum_{i \in I} |i|^{\alpha/2}$ since  $1\leq \sum_{i \in I\setminus\{0\}} |i|^{\alpha/2}$, yielding the upper-bound of \eqref{eq_sum_mom_herm}. The lower-bound in \eqref{eq_sum_mom_herm} is in turn obviously deduced.

	Concerning inequality \eqref{eq_sum_mom_herm_2}, using the change of variable $k \leftrightarrow -k$, we note that $$\int_{\mathbb{R}^d} \mathcal{F}[\psi_{i_1}](k) \overline{\mathcal{F}[\psi_{i_2}]}(k) |k|^{\alpha} dk \in \mathbb{R},$$ so the statement makes sense, and we have:
	\begin{align*}
		\sum_{i_1 , i_2 \in I} \left(\int_{\mathbb{R}^d} \mathcal{F}[\psi_{i_1}](k) \overline{\mathcal{F}[\psi_{i_2}]}(k) |k|^{\alpha} dk\right)^2 &\leq \sum_{i_1\in I} \sum_{i_2 \in \mathbb{N}^d} \left \langle \mathcal{F}[\psi_{i_1}](\cdot) | \cdot |^{\alpha}, \mathcal{F}[\psi_{i_2}](\cdot) \right \rangle^2. 
	\end{align*}
We view the latter angle brackets as the coefficients of the function $\mathcal{F}[\psi_{i_1}](\cdot) | \cdot |^{\alpha}$	in the orthonormal basis $(\mathcal{F}[\psi_{i_2}])_{i_2 \in \mathbb N^d}$, so that the sum over $i_2$ above is nothing else than $\|\mathcal{F}[\psi_{i_1}](\cdot) | \cdot |^{\alpha}\|_2^2$. 
Thereby,
	 	\begin{align*}
	 	\sum_{i_1 , i_2 \in I} \left(\int_{\mathbb{R}^d} \mathcal{F}[\psi_{i_1}](k) \overline{\mathcal{F}[\psi_{i_2}]}(k) |k|^{\alpha} dk\right)^2 &\leq  \sum_{i_1\in I} \int_{\mathbb{R}^d} |\mathcal{F}[\psi_{i_1}](k)|^2 |k|^{2\alpha}dk.
	 \end{align*}
	The conclusion  follows from the fact that for all $k$, $|\mathcal{F}[\psi_{i_1}](k)|=|\psi_{i_1}(k)|$ and Lemma~\ref{lemma_frac_norm_herm}. 
	\end{proof}

\begin{lemma}\label{lemma_remain_herm}
	Let $r > 0$ and $i \in \mathbb{N}^d$, then:
	$$\int_{\mathbb{R}^d \setminus [-r, r]^d} |\psi_i(x)|^2 dx \leq \frac{2d}{\pi^{1/4}} e^{|i|_{\infty}} \int_{y \geq r - \sqrt{2 |i|_{\infty}}} e^{-y^2/2} dy.$$
\end{lemma}
\begin{proof}
	With Equation (1.2) of~\cite{van1990new},
	$\forall y \in \mathbb{N},~|H_{n}(y)| \leq \pi^{-1/4} \exp(\sqrt{2n} |y|),$ we obtain that for $i=(i_1,\dots,i_d) \in \mathbb{N}^d$,  $x=(x_1,\dots,x_d) \in \mathbb{R}^d$ and $1 \leq l_0 \leq d$:
	\begin{align*}
		|\psi_i(x)| & = e^{-\frac12|x|^2} |H_{i_{l_0}}(x_{l_0})| \prod_{l = 1, l \neq l_0}^d |H_{i_l}(x_l)| \\& \leq  e^{-\frac12 x_{l_0}^2}  \pi^{-1/4} \exp(\sqrt{2 i_{l_0}} |x_{l_0}|) \prod_{l = 1, l \neq l_0}^d |H_{i_l}(x_l)|e^{-\frac12 x_l^2}.
	\end{align*}
	Using the fact that $H_{i_l}(x_l) e^{-\frac12 x_l^2}$ has a $L^2(\mathbb{R})$ norm equal to $1$ and the inequality:
	$$1 - \prod_{i = 1}^d \mathbf{1}_{|x_l| \leq r} \leq \sum_{l = 1}^d \mathbf{1}_{|x_l| > r},$$
	we get:
	\begin{align*}
		\int_{\mathbb{R}^d \setminus [-r, r]^d} |\psi_i(x)|^2 dx &\leq \sum_{l_0 = 1}^d \int_{\mathbb{R}^d} \mathbf{1}_{|x_{l_0}| \geq r}|\psi_i(x)|^2 dx  \leq \frac{1}{\pi^{1/4}} \sum_{l = 1}^d \int_{|y| \geq r} e^{-\frac12 y^2}  e^{\sqrt{2 i_l} |y|} dy.
	\end{align*}
	To conclude, we use the changes of variable $y \leftrightarrow y- \sqrt{2i_l}$:
	\begin{align*}
		\int_{\mathbb{R}^d \setminus [-r, r]^d} |\psi_i(x)|^2 dx & \leq 2\pi^{-1/4} \sum_{l = 1}^d e^{i_l} \int_{y \geq r} \exp\left(-\frac12 \left(y  - \sqrt{2 i_l}\right)^2\right) dy \\
		& = 2\pi^{-1/4} \sum_{l = 1}^d e^{i_l} \int_{y \geq r - \sqrt{2 i_l}} e^{-y^2/2} dy \\
		& \leq \frac{2d}{\pi^{1/4}}  e^{|i|_{\infty}}\int_{y \geq r - \sqrt{2 |i|_{\infty}}} e^{-y^2/2} dy.
	\end{align*}
\end{proof}
We now turn to the proof of Proposition \ref{prop_biais_var}

\begin{proof}[Proof of Propositon \ref{prop_biais_var}.]
	Let $\epsilon > 0$. We want to upper bound:
	$$p_R := \mathbb{P}\left(\mathbf{1}_{|\Phi_R| \geq 1}\log(R)\left|\widehat{\alpha}(I, J, R) - \alpha\right| \geq \epsilon\right).$$
	According to Equations \eqref{eq_sum_wj_j_0} and \eqref{eq_sum_wj_j_1}, we can normalize the sum inside the logarithm with a quantity close to expectation of the numerator:
	\begin{align*}
		\mathbf{1}_{|\Phi_R| \geq 1}\log(R)\left(\widehat{\alpha}(I, J, R) - \alpha\right) & = \mathbf{1}_{|\Phi_R| \geq 1} \sum_{j \in J} w_j \log\left(\sum_{i \in I} R^{(\alpha - d)j}T_j(\psi_i, R)^2\right) \\& = \mathbf{1}_{|\Phi_R| \geq 1} \sum_{j \in J} w_j \log\left(\frac{\sum_{i \in I} R^{(\alpha - d)j}T_j(\psi_i, R)^2}{\sum_{i \in I} \int_{\mathbb{R}^d} |\mathcal{F}[\psi_i](k)|^2 t|k|^{\alpha} dk}\right).
	\end{align*}
	We deduce:
	\begin{align*}
		p_R &\leq \mathbb{P}\left(\mathbf{1}_{|\Phi_R| \geq 1} \sum_{j \in J} |w_j| \left|\log\left(\frac{\sum_{i \in I} R^{(\alpha - d)j}T_j(\psi_i, R)^2}{\sum_{i \in I} \int_{\mathbb{R}^d} |\mathcal{F}[\psi_i](k)|^2 t|k|^{\alpha} dk}\right)\right| \geq \epsilon\right) \\& \leq \mathbb{P}\left(\mathbf{1}_{|\Phi_R| \geq 1} |w|_{\infty} \sum_{j \in J} \left|\log\left(\frac{\sum_{i \in I} R^{(\alpha - d)j}T_j(\psi_i, R)^2}{\sum_{i \in I} \int_{\mathbb{R}^d} |\mathcal{F}[\psi_i](k)|^2 t|k|^{\alpha} dk}\right)\right| \geq \epsilon\right) \\& \leq \sum_{j \in J} \mathbb{P}\left(\mathbf{1}_{|\Phi_R| \geq 1} \left|\log\left(\frac{\sum_{i \in I} R^{(\alpha - d)j}T_j(\psi_i, R)^2}{\sum_{i \in I} \int_{\mathbb{R}^d} |\mathcal{F}[\psi_i](k)|^2 t|k|^{\alpha} dk}\right)\right| \geq \frac{\epsilon}{|J| |w|_{\infty}}\right).
	\end{align*}
Denoting $\epsilon'=(\epsilon/(|J| |w|_{\infty})) \wedge 1$, we also have
	$$p_R \leq \sum_{j \in J} \mathbb{P}\left(\mathbf{1}_{|\Phi_R| \geq 1} \left|\log\left(\frac{\sum_{i \in I} R^{(\alpha - d)j}T_j(\psi_i, R)^2}{\sum_{i \in I} \int_{\mathbb{R}^d} |\mathcal{F}[\psi_i](k)|^2 t|k|^{\alpha} dk}\right)\right| \geq \epsilon'\right).$$
	Since $\epsilon' > 0$, we can rewrite the previous bound as:
	$$p_R \leq \sum_{j \in J} \mathbb{P}\left[\left|\log\left(\frac{\sum_{i \in I} R^{(\alpha - d)j} T_j(\psi, R)^2}{\sum_{i \in I} \int_{\mathbb{R}^d} |\mathcal{F}[\psi_i](k)|^2 t|k|^{\alpha} dk}\right)\right| \geq \epsilon' , |\Phi_R| \geq 1\right].$$
According to the mean value inequality, for all $y \in \mathbb{R}$ such that $|y-1| \leq \delta$, for some $\delta\leq 1/2$, then $|\log(y)| \leq 2\delta$. The contraposition of this result for $\delta=\epsilon'/2$ implies:
	\begin{align*}
		p_R & \leq \sum_{j \in J} \mathbb{P}\left[\left|\frac{\sum_{i \in I} R^{(\alpha - d)j} T_j(\psi, R)^2}{\sum_{i \in I} \int_{\mathbb{R}^d} |\mathcal{F}[\psi_i](k)|^2 t|k|^{\alpha} dk} - 1\right| \geq \frac{\epsilon'}{2}, |\Phi_R| \geq 1\right] \\
		& \leq \sum_{j \in J} \mathbb{P}\left[\left|\frac{\sum_{i \in I} R^{(\alpha - d)j} T_j(\psi, R)^2}{\sum_{i \in I} \int_{\mathbb{R}^d} |\mathcal{F}[\psi_i](k)|^2 t|k|^{\alpha} dk} - 1\right| \geq \frac{\epsilon'}{2}\right].
	\end{align*}
	Using Markov inequality, we obtain:
	\begin{equation}\label{eq_prop_bias_var_upp_pr}
		p_R \leq \frac{4}{\epsilon'^2}\sum_{j \in J} \mathbb{E}\left[\left(\frac{\sum_{i \in I} R^{(\alpha - d)j} T_j(\psi, R)^2}{\sum_{i \in I} \int_{\mathbb{R}^d} |\mathcal{F}[\psi_i](k)|^2 t|k|^{\alpha} dk} - 1\right)^2\right]= \frac{4}{\epsilon'^2}\sum_{j \in J} \mathcal{R}(I, R),
	\end{equation}
where:
	$$\mathcal{R}(I, R):= \mathbb{E}\left[\left(\frac{\sum_{i \in I} R^{(\alpha - d)j} T_j(\psi, R)^2}{\sum_{i \in I} \int_{\mathbb{R}^d} |\mathcal{F}[\psi_i](k)|^2 t|k|^{\alpha} dk} - 1\right)^2\right].$$
To bound $\mathcal{R}(I, R)$, we split it between a bias term and a variance term:
	\begin{align}\label{split R}
		\mathcal{R}(I, R) &= \frac1{\left(\sum_{i \in I} \int_{\mathbb{R}^d} |\mathcal{F}[\psi_i](k)|^2 t|k|^{\alpha} dk\right)^2} \left(B(I,R)^2 + V(I, R)\right), 
	\end{align}
	with
	$$B(I,R):= \sum_{i \in I} \mathbb{E}\left[R^{(\alpha - d)j} T_j(\psi_i, R)^2 - \int_{\mathbb{R}^d} |\mathcal{F}[\psi_i](k)|^2 t|k|^{\alpha} dk\right],$$
	$$V(i, R):= \operatorname{Var}\left[R^{(\alpha - d)j} \sum_{i \in I} |T_j(\psi_i, R)|^2\right].$$
	We first bound the bias term. We denote as in the proof of Lemma~\ref{lemma_cov}, $\psi_i^{R^{1-j}} := \psi_i \mathbf{1}_{[-R^{1-j}, R^{1-j}]^d}$. By \eqref{e.prop_campbell}, and using the computations leading to Equation \eqref{cov1}, we have
	\begin{align}\label{eq_var_psi_trunc}
	\mathbb{E}\left(R^{(\alpha - d)j} T_j(\psi_i, R)^2\right)&= \operatorname{Var}\left(R^{\frac{\alpha - d}{2} j}T_j(\psi_i, R)^2\right) \nonumber \\&= \int_{\mathbb{R}^d} \left|\mathcal{F}[\psi_i^{R^{1-j}}](k)\right|^2 R^{\alpha j}S\left(k/R^j\right) dk\nonumber \\& = \int_{\mathbb{R}^d} \left|\mathcal{F}[\psi_i](k)\right|^2 R^{\alpha j}S\left(k/R^j\right) dk + \delta_1(i, R) + \delta_2(i, R),
	\end{align}
	where 
	\begin{align*}
		\delta_1(i, R) &= \int_{\mathbb{R}^d} \left(\mathcal{F}[\psi_i^{R^{1-j}}](k) - \mathcal{F}[\psi_i](k)\right) \overline{\mathcal{F}[\psi_i^{R^{1-j}}](k)}  R^{\alpha j}S\left(k/R^j\right) dk,  \\
		\delta_2(i, R) &= \int_{\mathbb{R}^d} \left(\overline{\mathcal{F}[\psi_i^{R^{1-j}}]}(k) - \overline{\mathcal{F}[\psi_i]}(k)\right) \overline{\mathcal{F}[\psi_i](k)}  R^{\alpha j}S\left(k/R^j\right) dk.
	\end{align*}
	Using, as in \eqref{eq_delta_1_bound}, the Cauchy-Schwarz inequality and Plancherel Theorem, we obtain:
	\begin{align}\label{eq_remain_psi_trunc}
		|\delta_1(i, R)|+|\delta_2(i, R)| &\leq 2 \|S\|_{\infty} R^{\alpha j} \|\psi_i \mathbf{1}_{\mathbb{R^d} \setminus [-R^{1-j}, R^{1-j}]^d}\|_2 \|\psi_i\|_2 \nonumber \\
		& =2 \|S\|_{\infty} R^{\alpha j} \|\psi_i \mathbf{1}_{\mathbb{R^d} \setminus [-R^{1-j}, R^{1-j}]^d}\|_2.
	\end{align}
	Accordingly, gathering \eqref{eq_var_psi_trunc} and \eqref{eq_remain_psi_trunc}, the bound on the bias term is decomposed as
	\begin{align*}
	B(I,R) 
	&\leq B_1(I, R) + 2 \|S\|_{\infty} B_2(I, R),
	\end{align*}
	with 
	\begin{align*}
		B_1(I, R) &:= \sum_{i \in I} \int_{\mathbb{R}^d} |\mathcal{F}[\psi_i](k)|^2 R^{\alpha j} \left|S(k/R^j) -t|k/R^j|^{\alpha}\right| dk, \\& 
		B_2(I, R) := \sum_{i \in I} R^{\alpha j} \|\psi_i \mathbf{1}_{\mathbb{R^d} \setminus [-R^{1-j}, R^{1-j}]^d}\|_2.
	\end{align*}
	Concerning the first term, using the assumption that $|S(k) - t|k|^{\alpha}| \leq C_S |k|^{\beta}$ with $\beta > \alpha$ and $C_S> 0$, and then Lemma~\ref{lemma_bound_sum_herm}, we get, for some $C>0$, 
	\begin{align*}
	B_1(I,R)\leq \sum_{i \in I} \int_{\mathbb{R}^d} |\mathcal{F}[\psi_i](k)|^2 R^{(\alpha - \beta)j} C_S |k|^{\beta} dk \leq C R^{(\alpha - \beta)j} \sum_{i \in I} |i|^{\beta/2}.
	\end{align*}
	Concerning the second term of the bias, according to Lemma \ref{lemma_remain_herm} for $r  = R^{1 - j}$:
	\begin{align*}
		B_2(I, R) \leq R^{\alpha j} \sum_{i \in I} \left(\frac{2d}{\pi^{1/4}} e^{|i|_{\infty}} \int_{y \geq R^{1-j} - \sqrt{2 |i|_{\infty}}} e^{-y^2/2} dy\right)^{1/2}.
	\end{align*}
	Note that, for all $(i, j) \in I\times J$, $R^{1-j} - \sqrt{2 |i|_{\infty}} \geq R^{1-j_{\text{max}}} - \sqrt{2 |i|_{\infty}} \geq R^{1-j_{\text{max}}} - \sqrt{2 i_{\text{max}}} > 0$ by assumption, so using standard bound on the complementary error function, see, e.g., ~\cite{abramowitz1968handbook}, we get:
	\begin{align*}
		B_2(I, R) &\leq \frac{\sqrt{2d}}{\pi^{1/8}} R^{\alpha j} \sum_{i \in I} e^{|i|_{\infty}/2} \left(\frac{\exp\left(-(R^{1-j} - \sqrt{2 |i|_{\infty}})^2/2\right)}{R^{1-j} - \sqrt{2 |i|_{\infty}}}\right)^{1/2} \\& \leq \frac{\sqrt{2d}}{\pi^{1/8}} R^{\alpha j_{\text{max}}} |I| e^{i_{\text{max}}/2} \frac{\exp\left(-(R^{1-j_{\text{max}}} - \sqrt{2 i_{\text{max}}})^2/4\right)}{(R^{1-j_{\text{max}}} - \sqrt{2 i_{\text{max}}})^{1/2}}.
	\end{align*}
	
	Concerning the variance term in \eqref{split R}, we demonstrated in the proof of Theorem \ref{thm_clt} (see Section~\ref{sec_clt}) that all moments of the statistics $(R^{(\alpha - d)j/2} T_j(\psi_i, R))_{i \in I}$ converge to the moments of $(N(i))_{i \in I}$ as $R \to \infty$ where $(N(i))_{i \in I}$ is a Gaussian vector with covariance matrix $\left(\int_{\mathbb{R}^d} \mathcal{F}[\psi_{i_1}](k) \overline{\mathcal{F}[\psi_{i_2}]}(k) t|k|^{\alpha} dk\right)_{i_1 \in I, i_2 \in I}$. So there exists $R_0 > 1$ such that for all $R \geq R_0$:
	\begin{align*}
		V(i, R) &\leq 2 \operatorname{Var}\left[\sum_{i \in I} N(i)^2\right] = 2 \sum_{i_1 \in I, i_2 \in I} \operatorname{Cov}[N(i_1)^2, N(i_2)^2].
	\end{align*}
	Moreover, for Gaussian vectors $\operatorname{Cov}[N(i_1)^2, N(i_2)^2] = 2(\operatorname{Cov}[N(i_1), N(i_2)])^2$, so we can apply Lemma \ref{lemma_bound_sum_herm} to get, for some $C>0$, 
	\begin{align*}
		V(i, R) &\leq 4 \sum_{i_1 \in I, i_2 \in I} \left(\int_{\mathbb{R}^d} \mathcal{F}[\psi_{i_1}](k) \overline{\mathcal{F}[\psi_{i_2}]}(k) t|k|^{\alpha} dk\right)^2 \leq C \sum_{i \in I} |i|^{\alpha}.
	\end{align*}
	Finally, using again Lemma \ref{lemma_bound_sum_herm} for the denominator of $\mathcal{R}(I,R)$ in \eqref{split R}, we obtain that there exists $C > 0$ such that, 
	\begin{align*}
		\mathcal{R}(I, R) & \leq C\left(B_1(I, R)^2 +  V(I, R) + B_2(I, R)^2\right) \\& \leq C \frac1{\left(\sum_{i \in I} |i|^{\alpha/2}\right)^2} \Bigg(R^{2(\alpha - \beta)j} \left(\sum_{i \in I} |i|^{\beta/2}\right)^2 + \sum_{i \in I} |i|^{\alpha} \\& \qquad \qquad \qquad \qquad \qquad \qquad + R^{2\alpha j_{\text{max}}} |I|^2 e^{i_{\text{max}}} \frac{\exp\left(-(R^{1-j_{\text{max}}} - \sqrt{2 i_{\text{max}}})^2/2\right)}{R^{1-j_{\text{max}}} - \sqrt{2 i_{\text{max}}}}\Bigg).
	\end{align*} 
	Remember that by assumption, $I$ is of the form $I = \{i \in \mathbb{N}^d |~|i|_{\infty} < i_{\text{max}}\}$ with $i_{\text{max}} \geq 1$. We then use the facts that for $\nu \in \{\alpha/2, \alpha, \beta\}$, 
	\begin{align*}
		\sum_{i\in I} |i|^{\nu} & \leq d^{\nu/2} \sum_{0 \leq i_1, \dots, i_d < i_{\text{max}}} (i_1^\nu + \dots + i_d^\nu)  \leq  d^{\nu/2} i_{\text{max}}^{d} i_{\text{max}}^{\nu} =d^{\nu/2} i_{\text{max}}^{d+\nu}
	\end{align*}
	and
	\begin{align*}
		\sum_{i\in I} |i|^{\nu} & \geq \sum_{0 \leq i_1, \dots, i_d < i_{\text{max}}} i_1^\nu  = i_{\text{max}}^{d-1} \sum_{0 \leq i_1 < i_{\text{max}}} i_1^\nu  \geq \frac{i_{\text{max}}^{d-1}}{\nu +1} i_{\text{max}}^{\nu+1} = \frac{1}{\nu +1} i_{\text{max}}^{d+\nu},
		\end{align*}
	to obtain for $R \geq R_0$, since $|I| = i_{\text{max}}^d$,
	\begin{align*}
		\mathcal{R}(I, R) &\leq C\left(R^{2(\alpha - \beta)j}i_{\text{max}}^{\beta - \alpha} + i_{\text{max}}^{-d} + \frac{R^{2\alpha j_{\text{max}}} |I|^2}{i_{\text{max}}^{2d+\alpha}} e^{i_{\text{max}}} \frac{\exp\left(-(R^{1-j_{\text{max}}} - \sqrt{2 i_{\text{max}}})^2/2\right)}{R^{1-j_{\text{max}}} - \sqrt{2 i_{\text{max}}}} \right) \\& \leq C \left(\left(\frac{|I|^{1/d}}{R^{2{j_{\text{min}}}}}\right)^{\beta - \alpha} + \frac1{|I|} + R^{2\alpha j_{\text{max}}} \frac{e^{|I|^{1/d}}}{|I|^{\alpha/d}} \frac{\exp\left(-(R^{1-j_{\text{max}}} - \sqrt{2 i_{\text{max}}})^2/2\right)}{R^{1-j_{\text{max}}} - \sqrt{2 i_{\text{max}}}} \right).
	\end{align*}
	Because $j_{\text{max}} < 1$, we may choose $R_0$ larger, if necessary, to ensure that for $R \geq R_0$:
	$$R^{2\alpha j_{\text{max}}} \frac{e^{|I|^{1/d}}}{|I|^{\alpha/d}} \exp\left(-(R^{1-j_{\text{max}}} - \sqrt{2 i_{\text{max}}})^2/2\right) \leq  \exp\left(-(R^{1-j_{\text{max}}} - \sqrt{2 i_{\text{max}}})^2/4\right),$$
	and consequently:
	\begin{align*}
		\mathcal{R}(I, R) \leq C \left(\left(\frac{|I|^{1/d}}{R^{2{j_{\text{min}}}}}\right)^{\beta - \alpha} + \frac1{|I|} + \frac{\exp\left(-(R^{1-j_{\text{max}}} - \sqrt{2 i_{\text{max}}})^2/4\right)}{R^{1-j_{\text{max}}} - \sqrt{2 i_{\text{max}}}} \right).
	\end{align*}
	Plugging in this inequality in \eqref{eq_prop_bias_var_upp_pr} concludes the proof.
\end{proof}

\subsection{Asymptotic covariance matrix for Hermite  wavelets}\label{app_cov_mat}

As announced in Remark~\ref{rmk_cov_Hermites}, $\Sigma_R$, defined in equation \eqref{eq_cov_mat_R}, can be implemented numerically without costly numerical integrations. This section focuses on the dimension $d = 2$,  but the method can be adapted in dimension $d =1$ and $d \geq 3$.

Beyond the explicit formula provided in the next proposition, one may also use properties of $\Sigma_R$ in order to avoid the computation of certain coefficients. Specifically, denoting   $\Sigma_R = (\Sigma_R(i_1, i_2, j_1, j_2))_{(i_1, i_2) \in I^2, (j_1, j_2) \in J^2}$, we have $\Sigma_R(i_1, i_2, j_1, j_2) = \Sigma_R(i_2, i_1, j_2, j_1)$. Moreover, denoting $i = (i^1, i^2)$ for $i\in \mathbb N^2$, we have 
by parity that $\Sigma_R(i_1, i_2, j_1, j_2)=0$ whenever $(i_1^1 - i_2^1)$ or $(i_1^2 - i_2^2)$ is odd. 

In the next proposition, we denote
$$\forall z > 0,~\Gamma(z) = \int_{0}^{\infty} t^{z-1} e^{-t} dt, \quad \forall p, q \geq 0,~B(p, q) = \int_{0}^{2 \pi} \cos(\theta)^p \sin(\theta)^q d\theta.$$	
Remark that $B(p, q)$ can be easily computed, since on the one hand $B(p, q) \neq 0$ if and only if $p$ and $q$ are even, and on the other hand, for all $ p, q \geq 2$,
$$B(p, q) = \frac{(p-1)(q-1)}{(p+q)(p+q-2)} B(p-2, q-2), ~B(0, p) = B(p, 0) = \prod_{l = 1}^{p/2} \left(1 - \frac{1}{2l}\right).$$

\begin{proposition}\label{prop_computation_sig_R}
	Assume $d= 2$ and denote $(\Sigma_R(i_1, i_2, j_1, j_2))_{(i_1, i_2) \in I^2, (j_1, j_2) \in J^2}$ the covariance matrix defined in \eqref{eq_cov_mat_R}. Then $\Sigma_R(i_1, i_2, j_1, j_2)$ is given by
	\begin{multline*}
		 \frac{R^{\frac{\beta+2}2 (j_1 +j_2)}}2 \sum_{l_1^1, l_1^2, l_2^1, l_2^2 \geq 0} c_{l_1^1} c_{l_1^2} c_{l_2^1} c_{l_2^2} (-\bm{i})^{|l_1|_1}  \bm{i}^{|l_2|_1} B(|l_1|_1, |l_2|_1)\Gamma\left(\frac{2+\beta+|l_1|_1 + |l_2|_1}2\right) \\  \times R^{j_1|l_1|_1+j_2|l_2|_1}\left(\frac{R^{2j_1}+R^{2j_2}}2\right)^{-(2+\beta+|l_1|_1 + |l_2|_1)/2}, 
	\end{multline*}
	where $(c_{p})_{p \geq 0}$ are the coefficients of the Hermite polynomial $H_{p}$.
	\end{proposition}

\begin{proof}[Proof of Proposition \ref{prop_computation_sig_R}]
	Using the fact that Hermite wavelets are eigenvectors of the Fourier transform, i.e., for any $i\in\mathbb N^2$, $\mathcal{F}[\psi_{i}] =~(-\bm{i})^{i^1+i^2} \psi_{i}$ where $i = (i^1, i^2)$, we get:  
	\begin{align*}
		\Sigma_R(i_1, i_2, j_1, j_2) &= \bm{i}^{i_2^1+i_2^2} (-\bm{i})^{i_1^1+i_1^2} R^{\frac{\beta+d}2 (j_1 +j_2)} \underbrace{\int_{\mathbb{R}^2} \psi_{i_1}(R^{j_1} k) \psi_{i_2}(R^{j_2}k) |k|^{\beta} dk}_{ =: A}.
	\end{align*}
	To compute the latter integral, we use polar coordinates:
	\begin{align*}
		A &= \int_{\mathbb{R}^2} H_{i_1^1}(R^{j_1} k_1) H_{i_1^2}(R^{j_1} k_2) e^{-\frac12 R^{2j_1} k_1^2 -\frac12 R^{2j_2} k_2^2} \\& \qquad \qquad \times H_{i_2^1}(R^{j_2} k_1) H_{i_2^2}(R^{j_2} k_2) e^{-\frac12 R^{2j_2} k_2^2-\frac12 R^{2j_2} k_1^2} |k|^{\beta} dk \\& = \displaystyle\int_{0}^{\infty} \int_{0}^{2\pi} H_{i_1^1}(R^{j_1} r \cos(\theta)) ~H_{i_1^2}(R^{j_1} r \sin(\theta))~H_{i_2^1}(R^{j_2} r \cos(\theta))~H_{i_2^2}(R^{j_2} r \sin(\theta)) \\& \qquad \qquad \times \exp\left(-\frac12 (R^{2j_1}+R^{2j_2})r^2\right) r^{\beta+1} d\theta dr.
	\end{align*}
	To get the announced formula, we expand the product of the four polynomials, perform the change of variable $r \leftrightarrow r \sqrt{(R^{2j_1}+R^{2j_2})/2}$ and finally use the fact that for $\nu > 0$, $\int_{0}^{\infty} r^{\nu} e^{-r^2} dr =\Gamma((\nu +1)/2)/2$.  
\end{proof}

\begin{appendix}
\section*{Cumulant and Brillinger mixing for point processes}\label{appendix}
We recall in this appendix the definition of factorial cumulant moment measures of a simple point process and their reduced counterpart in case of stationarity; see also \cite{daley2007introduction, daley2003introduction}. We then present the property of Brillinger mixing of a simple stationary point process, as assumed in our Theorem~\ref{thm_clt}; see for instance \cite{heinrich2012asymptotic, biscio2016brillinger, blaszczyszyn2019limit}.

As a generalisation of the second order factorial moment measure \eqref{rho2}, the factorial moment measures $\rho^{(r)}$, for $r\geq 1$, of a simple point process $\Phi$ is defined for any bounded measurable sets $A_1, \dots, A_r$ in $\mathbb{R}^d$ by
\begin{align*}
		\rho^{(r)}(A_1 \times  \ldots \times A_r) = \mathbb{E}\left(\sum_{x_1, \ldots, x_r \in \Phi}^{\neq} 1\{x_1 \in A_1, \ldots, x_m \in A_r\}\right).
	\end{align*}

The $m$-th order factorial cumulant moment measure $\gamma^{(m)}$ are in turn defined by the following expression:
\begin{align*}
		\gamma^{(m)}(A_1\times \dots\times A_m) = \sum_{\sigma\in\Pi[p]} (-1)^{|\sigma|-1} (|\sigma| - 1)!  \prod_{i=1}^{|\sigma|} \rho^{(|\sigma_i|)}\left(\prod_{k_i \in \sigma(i)} A_{k_i}\right),
	\end{align*}
where, as in \eqref{eq_cumulant-X}, $\Pi[m]$ denotes the set of all unordered partitions of $\{1,\ldots,m\}$.

Furthermore, when the point process $\Phi$ is  stationary, we may consider   for any $m \geq 2$ the reduced $m$-th order factorial cumulant moment measure $\gamma_{\text{red}}^{(m)}$. This is  a locally finite signed measure on $(\mathbb{R}^{d})^{(m-1)}$ that satisfies
	\[
	\gamma^{(m)}\left(A_1\times \dots\times A_m\right) = \int_{A_m} \gamma_{\text{red}}^{(m)}\left(\prod_{i=1}^{m-1} (A_i - x)\right) \, dx,
	\]
	where for $i = 1, \ldots, m - 1$, $(A_i - x)$ is the translation of the set $A_i$ by $x$. 

Finally, a simple stationary point process $\Phi$ is said Brillinger-mixing if, for $m \geq 2$, and any  bounded measurable set $A$ of $\mathbb{R}^d$, $\mathbb{E}\left[|\Phi \cap A|^m\right] < +\infty$ and if $|\gamma_{\text{red}}^{(m)}|(\mathbb{R}^{d(m-1)}) < +\infty,$ 
	where the total variation of  $\gamma_{\text{red}}^{(m)}$ is defined by $|\gamma_{\text{red}}^{(m)}|  := {\gamma_{\text{red}}^{(m)}}^+ + {\gamma_{\text{red}}^{(m)}}^-$ where $\gamma_{\text{red}}^{(m)}  = {\gamma_{\text{red}}^{(m)}}^+ - {\gamma_{\text{red}}^{(m)}}^-$ is the Jordan decomposition of the signed measure $\gamma_{\text{red}}^{(m)}$.
\end{appendix}

\label{references}
%% if your bibliography is in bibtex format, uncomment commands:
\bibliographystyle{imsart-number} % Style BST file (imsart-number.bst or imsart-nameyear.bst)
%\bibliography{references.bib}       % Bibliography file (usually '*.bib')

%% or include bibliography directly:
% \begin{thebibliography}{}
% \bibitem{b1}
% \end{thebibliography}

\end{document}